\documentclass[preprint,12pt]{elsarticle}

\usepackage{amssymb}
\usepackage{amsmath}
\usepackage{amsthm}
\usepackage{hyperref}
\usepackage{lipsum}
\usepackage{amsfonts}
\usepackage{graphicx}
\usepackage{epstopdf}
\usepackage{algpseudocode}
\usepackage{tensor}
\usepackage{verbatim}
\usepackage{mathtools}
\usepackage{cleveref}
\usepackage{framed}
\usepackage{tikz}
\usetikzlibrary{calc, shapes, angles, quotes, patterns, decorations.pathreplacing, cd, patterns.meta}
\usepackage{subcaption}
\usepackage{bm}
\usepackage{esint}
\usepackage{dsfont}
\usepackage{mathrsfs}
\usepackage[export]{adjustbox}
\usepackage{siunitx}
\usepackage{multirow}
\usepackage{hhline}
\usepackage{diagbox}

\newlength{\leftstackrelawd}
\newlength{\leftstackrelbwd}
\def\leftstackrel#1#2{\settowidth{\leftstackrelawd}%
	{${{}^{#1}}$}\settowidth{\leftstackrelbwd}{$#2$}%
	\addtolength{\leftstackrelawd}{-\leftstackrelbwd}%
	\leavevmode\ifthenelse{\lengthtest{\leftstackrelawd>0pt}}%
	{\kerneln-.5\leftstackrelawd}{}\mathrel{\mathop{#2}\limits^{#1}}}

\newcommand{\bdd}[1]{ \boldsymbol{#1} }
\newcommand{\mathbbb}[1]{\pmb{\mathbb{#1}}}
\newcommand{\unitvec}[1]{\bdd{#1}}
\makeatletter
\def\@tvsp{\mathchoice{{}\mkern-4.5mu}{{}\mkern-4.5mu}{{}\mkern-2.5mu}{}}
\def\ltrivert{\left|\@tvsp\left|\@tvsp\left|}
\def\rtrivert{\right|\@tvsp\right|\@tvsp\right|}
\makeatother
\newcommand\tnorm[1]{\ltrivert #1 \rtrivert }
\newcommand\dinner[2]{[ #1, #2 ]}

\newcommand{\hcurl}{\bdd{H}(\mathrm{curl};\Omega)}
\newcommand{\hcurlgamma}{\bdd{H}_{\Gamma}(\mathrm{curl};\Omega)}

\makeatletter
\let\orgdescriptionlabel\descriptionlabel
\renewcommand*{\descriptionlabel}[1]{%
	\let\orglabel\label
	\let\label\@gobble
	\phantomsection
	\edef\@currentlabel{#1}%
	\let\label\orglabel
	\orgdescriptionlabel{#1}%
}
\makeatother

\usepackage{amsopn}

\DeclareMathOperator{\grad}{\mathbf{grad}}
\DeclareMathOperator{\dive}{div}
\DeclareMathOperator{\curl}{curl}

\DeclareMathOperator{\Tr}{Tr}
\DeclareMathOperator{\image}{Im}
\DeclareMathOperator{\Kernel}{Ker}

\biboptions{sort&compress}

\newtheorem{theorem}{Theorem}[section]
\newtheorem{lemma}{Lemma}[section]
\newtheorem{corollary}{Corollary}[section]
\newdefinition{remark}{Remark}

\crefname{equation}{}{}
\Crefname{equation}{Equation}{Equations}
\crefname{lemma}{Lemma}{Lemmas}
\crefname{theorem}{Theorem}{Theorems}
\crefname{corollary}{Corollary}{Corollaries}
\crefname{figure}{Figure}{Figures}
\Crefname{figure}{Figure}{Figures}

\numberwithin{equation}{section}

\journal{Computer Methods in Applied Mechanics and Engineering}

\begin{document}

\begin{frontmatter}



\title{Two and three dimensional $H^2$-conforming finite element approximations without $C^1$-elements \tnoteref{t1}}
\tnotetext[t1]{The authors acknowledge that this material is based upon work supported by the National Science Foundation under Award No. DMS-2324364 for the first author and Award No. DMS-2201487 for the second author.}


\author[1]{Mark Ainsworth\corref{cor1}}
\ead{mark\_ainsworth@brown.edu}

\author[2]{Charles Parker}
\ead{charles.parker@maths.ox.ac.uk}

\cortext[cor1]{Corresponding author}

\affiliation[1]{organization={Division of Applied Mathematics, Brown University},
            addressline={Box F, 182 George Street}, 
            city={Providence},
            postcode={02912}, 
            state={RI},
            country={USA}}

\affiliation[2]{organization={Mathematical Institute, University of Oxford},
	addressline={Andrew Wiles Building, Woodstock Road}, 
	city={Oxford},
	postcode={OX2 6GG}, 
	country={UK}}  

\begin{abstract}
	We develop a method to compute $H^2$-conforming finite element approximations in both two and three space dimensions using readily available finite element spaces. This is accomplished by deriving a novel, equivalent mixed variational formulation involving spaces with at most $H^1$-smoothness, so that conforming discretizations require at most $C^0$-continuity. The method is demonstrated on arbitrary order $C^1$-splines.
\end{abstract}


\end{frontmatter}




\section{Introduction}
\label{sec:intro}

Fourth-order partial differential equations arise in a variety of applications including thin plates and shells \cite{Timoshenko59}, the separation of binary alloys \cite{CahnHill58}, condensed matter physics \cite{Escudero13}, and liquid crystals \cite{Pevnyi14}, among many others; see e.g. \cite{Greer06} and references therein. The variational formulation of such problems, or their linearizations, are well-posed in subspaces of $H^2$ so that a standard Galerkin scheme would require $H^2$-conforming finite element methods. Conforming schemes are attractive in that they retain the stability of the variational formulation. While the $H^2$-conforming 2D Argyris element appears in Firedrake \cite{FiredrakeUserManual,Kirby18,Kirby19} and scikit-fem \cite{skfem2020} and the 2D Hsieh-Clough-Tocher (HCT) element appears in FreeFEM \cite{FreeFEM}, the choice is otherwise rather limited. For instance, 3D $H^2$-conforming elements seem to be absent from all major packages. 

Why are $H^2$-conforming elements not more widely available? One possible answer is that most $H^2$-conforming elements are $C^1$-continuous \cite{Cia02} and imposing $C^1$-continuity is less straightforward then imposing $C^0$-continuity. Consider, for instance, the simple setting where the domain is the unit cube and the finite element space is the space of $C^1$ piecewise polynomials of degree $p \in \mathbb{N}$ in the special case of tetrahedral Freudenthal partitions \cite{Bey00} of the cube. Even counting the dimension of the associated $H^2$-conforming space is a nontrivial matter. In \cite{Hecklin06}, a count of the dimension of the space is given explicitly for all $p \in \mathbb{N}$ along with a choice of basis. However, the basis cannot be constructed in terms of local basis functions defined on a reference element. Instead, the basis is global in nature and thus not amenable to standard finite element procedures such as subassembly. On a general tetrahedral partition, even counting the dimension of the associated $C^1$ piecewise polynomials is an open problem. Standard $H^2$-conforming finite elements based on locally defined basis functions are, of course, possible but only if the degree is sufficiently high ($p \geq 5$ in 2D \cite{Argyris68} and $p \geq 9$ \cite[Chapter 18.2]{Lai07spline} in 3D).

In this work, we develop a method to compute $H^2$-conforming finite element approximations in both two and three space dimensions using finite element spaces that impose lower continuity than $C^1$, and which are more readily available. One benefit of such a scheme is that $H^2$-conforming approximations can be obtained using existing software packages even when $H^2$-spaces are not offered. This is accomplished by deriving a novel, equivalent mixed variational formulation involving spaces with at most $H^1$-smoothness, so that conforming discretizations require at most $C^0$-continuity. Particular choices of conforming spaces recover $H^2$-conforming approximations of the original $H^2$-variational problem, including  the space in \cite{Hecklin06}. The method generalizes our previous work \cite{AinCP23KirchI} to $H^2$-problems with lower order terms (e.g. time dependent problems) and to three space dimensions.

There are many alternative mixed methods for $H^2$-variational problems currently available -- most of which specifically focus on the biharmonic equation or Kirchhoff plate problem in 2D; see e.g. \cite{Amara02,Bramble83,Ciarlet74,Veiga07,Dest96,Rafetseder18}.  Methods that are applicable to more general $H^2$-variational problems in two and three dimensions are less common; see e.g. \cite{Chen18,Farrell22,Gallistl17}. In any case, none of the above methods will, in general, recover the $H^2$-conforming approximation. In contrast, our method uses standard spaces to recover the $H^2$-conforming approximation in two and three dimensions for problems with general boundary conditions.

The remainder of the paper is organized as follows. In \cref{sec:gen-problem-setting}, we describe the problem setting that naturally leads to the equivalent mixed variational formulation analyzed in \cref{sec:novel-mixed}. \Cref{sec:derive-mixed-fem-gen} details well-posed discretizations of the mixed formulation and their relationship to $H^2$-conforming approximations. The implementation of the method and several numerical examples are given in \cref{sec:iter-pen,sec:numerics}, while \cref{sec:time-dependent-problems} discusses the extension to time-dependent problems. For the 2D Morgan-Scott element \cite{MorganScott75}, we show in \cref{sec:morgan-scott} that the inf-sup condition associated with the discrete mixed problem is uniformly bounded away from zero with respect to both the polynomial degree and mesh size. Finally, in \cref{sec:prev-work}, we show precisely how the proposed method generalizes our previous work \cite{AinCP23KirchI}.

\section{Problem setting}
\label{sec:gen-problem-setting}

Let $\Omega \subset \mathbb{R}^d$, $d \in \{2,3\}$, be a connected Lipschitz polyhedral domain. The boundary of $\Omega$, denoted $\Gamma := \partial \Omega$, is partitioned into disjoint sets $\Gamma_c$, $\Gamma_s$, and $\Gamma_f$. For $d=2$, we assume that there exists simply-connected Lipschitz polygons $\Omega_0, \Omega_1, \ldots, \Omega_{N_H} \subset \mathbb{R}^2$ such that $\bar{\Omega}_j \subsetneq \Omega_0$ for $1 \leq j \leq N_H$, $\bar{\Omega}_i \cap \bar{\Omega}_j = \emptyset$ for $1 \leq i,j \leq N_H$ with $i \neq j$, and $\Omega = \Omega_0 \setminus \cup_{j=1}^{N_H} \bar{\Omega}_j$. For $d=3$, we assume that the geometry of $\Gamma_{cs} := \Gamma_c \cup \Gamma_s$ satisfies the following conditions \cite{Hiptmair09}: For each of the $N$ connected components $\Gamma_{cs}^{(i)}$, $1 \leq i \leq N$, of $\Gamma_{cs}$, there exists  an open Lipschitz domain $\Omega_i \subset \mathbb{R}^3$ such that $\bar{\Omega}_i \cap \bar{\Omega} = \bar{\Gamma}_{cs}^{(i)}$, $\Omega_i \cap \Omega = \emptyset$, the distance between $\Omega_i$ and $\Omega_j$ is strictly positive for $i \neq j$, and $\bar{\Omega} \cup \bar{\Omega}_1 \cup \ldots \cup \bar{\Omega}_{N}$ is a Lipschitz domain.

We consider $H^2(\Omega)$ variational problems of the form
\begin{align}
	\label{eq:h2-bilinear}
	w \in H^2_{\Gamma}(\Omega) : \qquad B(w, v) = F(v) \qquad \forall v \in H^2_{\Gamma}(\Omega),
\end{align} 
where 
\begin{align*}
	H^2_{\Gamma}(\Omega) := \{ v \in H^2(\Omega) : v|_{\Gamma_{cs}} = 0 \text{ and } \partial_n v|_{\Gamma_c} = 0 \},
\end{align*}
$F$ is a linear functional on $H^2_{\Gamma}(\Omega)$, and $B : H^2(\Omega) \times H^2(\Omega) \to \mathbb{R}$  is a bilinear form with the following properties.
\begin{itemize}
	\item (Boundedness) There exists $M_B > 0$ such that
	\begin{align}
		\label{eq:b-bilinear-bounded}
		|B(u, v)| \leq M_B \|u\|_2 \|v\|_2 \qquad \forall u, v \in H^2(\Omega).
	\end{align}
	
	\item (Inf-sup condition) There exists $\alpha > 0$ such that
	\begin{align}
		\label{eq:b-bilinear-infsup}
		\inf_{0 \neq u \in H^2_{\Gamma}(\Omega)} \sup_{0 \neq v \in H^2_{\Gamma}(\Omega)} \frac{ B(u, v) }{ \|u\|_2 \|v\|_2 } \geq \alpha.
	\end{align} 
	
	\item (Uniqueness condition) For every $v \in H^2_{\Gamma}(\Omega)$, there holds
	\begin{align}
		\label{eq:b-bilinear-uniqueness}
		\sup_{u \in H^2_{\Gamma}(\Omega)} B(u, v) > 0.
	\end{align}
\end{itemize}
Then, problem \cref{eq:h2-bilinear} is well-posed \cite[Theorem 2.6]{GuerErn04}.

Let $\mathcal{T}$ be a shape regular \cite[Definition (4.4.13)]{Brenner08} partitioning of $\Omega$ into simplices such that the nonempty intersection of any two distinct elements from $\mathcal{T}$ is a single common sub-simplex of both elements. Given an $H^2$-conforming finite element space $\mathbb{W} \subset H^2(\Omega)$ on $\mathcal{T}$, typically consisting of $C^1$-continuous piecewise polynomials of an arbitrary but fixed degree $p$, the corresponding $H^2$-conforming Galerkin finite element scheme for \cref{eq:h2-bilinear} reads
\begin{align}
	\label{eq:h2-bilinear-fem}
	w_{\mathcal{T}} \in \mathbb{W}_{\Gamma} : \qquad B(w_{\mathcal{T}}, v) = F(v) \qquad \forall v \in \mathbb{W}_{\Gamma},
\end{align}
where $\mathbb{W}_{\Gamma} := \mathbb{W} \cap H^2_{\Gamma}(\Omega)$. If there exists a positive constant $\alpha_{\mathcal{T}} > 0$, possibly depending on the dimension of $\mathbb{W}_{\Gamma}$, such that the following discrete analogue of the inf-sup condition \cref{eq:b-bilinear-infsup} holds:
\begin{align*}
	\alpha_{\mathcal{T}} &:= \inf_{0 \neq u \in\mathbb{W}_{\Gamma}} \sup_{0 \neq v \in \mathbb{W}_{\Gamma}} \frac{ B(u, v) }{ \|u\|_2 \|v\|_2 } > 0, 
\end{align*} 
then \cref{eq:h2-bilinear-fem} has a unique solution. Moreover, standard error estimates (e.g. \cite[Lemma 2.28]{GuerErn04}) show that $w_{\mathcal{T}}$ is, up to a constant, the best approximation to $w$ in $\mathbb{W}_{\Gamma}$:
\begin{align*}
	\| w - w_{\mathcal{T}} \|_{2} \leq \left( 1 + \frac{M_B}{\alpha_{\mathcal{T}}} \right) \inf_{v \in \mathbb{W}_{\Gamma}} \|w - v\|_{2}.
\end{align*}

\subsection{Typical Problems}

In practice, the bilinear form $B(\cdot,\cdot)$ often has more structure beyond that indicated by the abstract assumptions \cref{eq:b-bilinear-bounded,eq:b-bilinear-infsup,eq:b-bilinear-uniqueness}. For example, in the case of the biharmonic equation $B(\cdot,\cdot)$ depends only on the \textit{gradient} of its arguments i.e.
\begin{align}
	\label{eq:biharmonic-a}
	B(w, v) = a(\grad w, \grad v),
\end{align} 
where the bilinear form $a(\bdd{\theta}, \bdd{\psi}) := (\dive \bdd{\theta}, \dive \bdd{\psi})$ is bounded on $\bdd{H}^1(\Omega)$. In a similar vein, the bilinear form $B(\cdot,\cdot)$ for a Kirchhoff plate in equilibrium (see e.g. Chapter 4 of \cite{Timoshenko59}) again has the form \cref{eq:biharmonic-a}, this time with
\begin{align}
	\label{eq:kirchhoff-a}
	a(\bdd{\theta}, \bdd{\psi}) :=  D\left[ (1-\nu)(\bdd{\varepsilon}(\bdd{\theta}), \bdd{\varepsilon}(\bdd{\psi})) + \nu (\dive \bdd{\theta}, \dive \bdd{\psi}) \right],
\end{align}
where $\bdd{\varepsilon}(\bdd{\theta}) = (\grad \bdd{\theta} + (\grad \bdd{\theta})^T)/2$ and $D$ and $\nu$ are positive physical parameters so that the corresponding bilinear form $a(\cdot,\cdot)$ is again bounded on $\bdd{H}^1(\Omega)$. Here, $\Gamma_c$, $\Gamma_s$, and $\Gamma_f$ correspond to the parts of the boundary where the plate is clamped, simply-supported, and left free. 

Discretizations of time dependent problems offer another useful example. Consider, for instance, the time dependent biharmonic equation
\begin{align*}
	\partial_t w + \Delta^2 w = G \quad \text{with} \quad w|_{t=0} = g 
\end{align*}
in variational form: find $w(t) \in H^2_{\Gamma}(\Omega)$ such that
\begin{align*}
	\frac{d}{dt} (w(t), v) + (\Delta w(t), \Delta v) = G(v) \qquad \forall v \in H^2_{\Gamma}(\Omega)
\end{align*}
with $w(0) = g$. This problem can be accommodated in the foregoing framework using Rothe's method \cite{Rothe30} to first discretize in time (using e.g. Crank-Nicolson with a step size $\Delta t > 0$) and thereby obtain the semi-discrete scheme: For $n \in \mathbb{N}_0$, find $w_{n} \in H^2_{\Gamma}(\Omega)$ such that
\begin{align}
	\label{eq:b-time-dependent}
	(w_{n}, v) + \frac{\Delta t}{2} (\Delta w_{n}, \Delta v) =  (w_{n-1}, v) - \frac{\Delta t}{2} (\Delta w_{n-1}, \Delta v) + \Delta t G(v)  \qquad \forall v \in H^2_{\Gamma}(\Omega)
\end{align}
with $w_0 = g$. This problem is of the form \cref{eq:h2-bilinear} where is given by
\begin{align*}
	B(w, v) = \frac{\Delta t}{2} (\Delta w, \Delta v) + (w, v) = \frac{\Delta t}{2} a(\grad w, \grad v) + (w, v),
\end{align*}
where $a(\cdot,\cdot)$ is the bilinear form that appears in the equilibrium case \cref{eq:biharmonic-a}. Other fourth-order problems, such as the Cahn-Hilliard equations \cite{CahnHill58}, give rise to problems of the form \cref{eq:h2-bilinear} with the same structure as in \cref{eq:b-time-dependent}. The same structure also arises for problems with lower order terms. For example, the weak formulation of the PDE $\Delta^2 u - \Delta u + \bdd{b} \cdot \grad u + u = f$ with $\Gamma_c = \Gamma$ reads: Find $u \in H^2_{\Gamma}(\Omega)$ such that
\begin{align}
	\label{eq:gen-fourth-order-weak}
	(\Delta u, \Delta v) + (\grad u, \grad v) + (\bdd{b} \cdot \grad u, v) + (u, v) = (f, v) \qquad \forall v \in H^2_{\Gamma}(\Omega). 
\end{align}

\subsection{Structural assumption}

In view of the above examples, we assume that the bilinear form $B(\cdot,\cdot)$ can be written in the form 
\begin{align}
	\label{eq:b-bilinear-splitting}
	B(w, v) = a(\grad w, \grad v) + c(w, v) \qquad \forall w, v \in H^2(\Omega),
\end{align}
where $a : \bdd{H}^1(\Omega) \times \bdd{H}^1(\Omega) \to \mathbb{R}$ and $c : H^1(\Omega) \times H^1(\Omega) \to \mathbb{R}$ are bilinear forms for which there exists a constant $M > 0$ such that 
\begin{subequations}
	\label{eq:ac-bilinear-bounded}
	\begin{alignat}{2}
		|a(\bdd{\theta}, \bdd{\psi})| &\leq M \|\bdd{\theta}\|_1 \|\bdd{\psi}\|_1 \qquad & &\forall \bdd{\theta}, \bdd{\psi} \in \bdd{H}^1(\Omega), \\
		|c(w, v)| &\leq M \|w\|_1 \|v\|_1 \qquad & &\forall w, v \in H^1(\Omega),
	\end{alignat}
\end{subequations}
where the $H^1(\Omega)$-norm of a vector-valued function is the sum of the $H^1(\Omega)$-norm of its components. Similarly, we assume that the right hand side $F$ can be written as the sum of two linear functionals $F_1 \in \bdd{H}^1(\Omega)^*$ and $F_2 \in H^1(\Omega)^*$ satisfying
\begin{align}
	\label{eq:F-functional-splitting}
	F(v) = F_1(\grad v) + F_2(v) \qquad \forall v \in H^2(\Omega).
\end{align}
The decompositions \cref{eq:b-bilinear-splitting,eq:F-functional-splitting} need not be unique. For instance, the bilinear form appearing in the example in \cref{eq:gen-fourth-order-weak} with $\bdd{b} = \bdd{0}$ satisfies the decomposition condition with either $a(\cdot,\cdot) = (\dive \cdot, \dive \cdot)$ and $c(\cdot,\cdot) = (\grad \cdot, \grad \cdot) + (\cdot,\cdot)$ or alternatively $a(\cdot,\cdot) = (\dive \cdot, \dive \cdot) + (\cdot,\cdot)$ and $c(\cdot,\cdot) = (\cdot,\cdot)$. In both cases, $a(\cdot,\cdot)$ and $c(\cdot,\cdot)$ are bounded on $\bdd{H}^1(\Omega)$ and $H^1(\Omega)$ respectively. The ensuing analysis makes use of the existence of decompositions \cref{eq:b-bilinear-splitting,eq:F-functional-splitting}, but does not require their uniqueness.

\section{A novel mixed formulation for the continuous problem}
\label{sec:novel-mixed}

The original primal variational problem \cref{eq:h2-bilinear} can be written as a mixed formulation involving spaces that have weaker smoothness requirements than $H^2(\Omega)$ as follows. Let $\bdd{\gamma} = \grad w$ and observe that, since $w \in H^2_{\Gamma}(\Omega)$, the gradient $\bdd{\gamma}$ belongs to the space
\begin{align*}
	\bdd{G}_{\Gamma}(\Omega) := \{ \bdd{\psi} \in H^1(\Omega)^d : \bdd{\psi}|_{\Gamma_c} = \bdd{0} \text{ and } \bdd{\pi}_T \bdd{\psi}|_{\Gamma_s} = \bdd{0} \},
\end{align*}
where $\bdd{\pi}_T$ is the tangential projection operator defined by
\begin{align*}
	\bdd{\pi}_{T}\bdd{\psi} :=  \bdd{\psi} - (\bdd{\psi} \cdot \unitvec{n})\unitvec{n} \qquad \text{on } \Gamma.	
\end{align*}
Observing that $w \in H^2_{\Gamma}(\Omega) \subset H^1_{\Gamma}(\Omega)$, where
\begin{align*}
	H^1_{\Gamma}(\Omega) := \{ \tilde{v} \in H^1(\Omega) : \tilde{v}|_{\Gamma_{cs} } = 0 \},
\end{align*}
it follows that $(w, \bdd{\gamma}) \in H^1_{\Gamma}(\Omega) \times \bdd{G}_{\Gamma}(\Omega)$ satisfies
\begin{align}
	\label{eq:h2-variational-with-constraint}
	a(\bdd{\gamma}, \grad v) + c(w, v) &= F_1(\grad v) + F_2(v) \qquad \forall v \in H^2_{\Gamma}(\Omega).
\end{align}
\Cref{eq:h2-variational-with-constraint} holds for functions belonging to $H^2_{\Gamma}(\Omega)$. It is not difficult to see that, with the above notation, $H^2_{\Gamma}(\Omega)$ can be precisely characterized as the following subspace of $H^1_{\Gamma}(\Omega)$:
\begin{align}
	\label{eq:h2gamma-h1gamma-id}
	H^2_{\Gamma}(\Omega) = \{ \tilde{v} \in H^1_{\Gamma}(\Omega) : \grad \tilde{v} \in \bdd{G}_{\Gamma}(\Omega)  \}.
\end{align}
Hence, viewing the relation $\bdd{\gamma} = \grad w$ as a side constraint for \cref{eq:h2-variational-with-constraint} to be imposed using a Lagrange multiplier $\bdd{\phi}$, we arrive at the following mixed problem: Find $(\tilde{w}, \bdd{\gamma}, \bdd{\phi})  \in H^1_{\Gamma}(\Omega) \times \bdd{G}_{\Gamma}(\Omega) \times \bdd{\Phi}_{\Gamma}(\Omega)$ such that
\begin{subequations}
	\label{eq:h2-mixed-problem}
	\begin{alignat}{2}
		\label{eq:h2-mixed-problem-1}
		a(\bdd{\gamma}, \bdd{\psi}) + c(\tilde{w}, \tilde{v}) + \langle \bdd{\Xi}(\tilde{v}, \bdd{\psi}), \bdd{\phi} \rangle &= F_1(\bdd{\psi}) + F_2(\tilde{v}) \qquad & &\forall (\tilde{v}, \bdd{\psi}) \in  H^1_{\Gamma}(\Omega) \times \bdd{G}_{\Gamma}(\Omega), \\
		\label{eq:h2-mixed-problem-2}
		\langle \bdd{\Xi}(\tilde{w}, \bdd{\gamma}), \bdd{\eta} \rangle &= 0 \qquad & &\forall \bdd{\eta} \in \bdd{\Phi}_{\Gamma}(\Omega), 
	\end{alignat}
\end{subequations}
where the operator $\bdd{\Xi}$ is defined by the rule
\begin{align}
	\label{eq:xi-def}
	\bdd{\Xi}(\tilde{v}, \bdd{\psi}) := \grad \tilde{v} - \bdd{\psi} \qquad \forall (\tilde{v}, \bdd{\psi}) \in  H^1_{\Gamma}(\Omega) \times \bdd{G}_{\Gamma}(\Omega).
\end{align}
Observe that we write $\tilde{w}$ rather than $w$ in \cref{eq:h2-mixed-problem} to reflect the fact that, starting with \cref{eq:h2-mixed-problem}, it is not obvious (at this stage) that $\tilde{w}$ satisfies the original variational problem \cref{eq:h2-bilinear}. Later, we shall see that if the function space $\bdd{\Phi}_{\Gamma}(\Omega)$ and duality pairing $\langle \cdot, \cdot \rangle$ are chosen appropriately, then $\tilde{w}$ does indeed satisfy \cref{eq:h2-bilinear} and, as a consequence $w = \tilde{w}$.

It remains to choose an appropriate function space $\bdd{\Phi}_{\Gamma}(\Omega)$ and duality pairing $\langle \cdot, \cdot \rangle$ with which to impose the side constraint. We make the natural choice and take $\bdd{\Phi}_{\Gamma}(\Omega) := (\image \bdd{\Xi})^*$, where
\begin{align*}
	\image \bdd{\Xi} := \{ \bdd{\Xi}(\tilde{v}, \bdd{\psi}) : (\tilde{v}, \bdd{\psi}) \in  H^1_{\Gamma}(\Omega) \times \bdd{G}_{\Gamma}(\Omega) \},
\end{align*}
so that $\image \bdd{\Xi} \subset {L}^2(\Omega)^d$ and the duality pairing $\langle \cdot, \cdot \rangle$ is taken to be the extension of the $L^2(\Omega)^d$ inner-product to $\image \bdd{\Xi} \times (\image \bdd{\Xi})^*$.

\subsection[Characterizing $(\image \Xi)^*$]{Characterizing $(\image \bdd{\Xi})^*$}

The mixed formulation \cref{eq:h2-mixed-problem} achieves the primary goal of reducing the maximum smoothness of the involved spaces: $H^1_{\Gamma}(\Omega) \times \bdd{G}_{\Gamma}(\Omega) \subset H^1(\Omega) \times H^1(\Omega)^d$ and $L^2(\Omega)^d \subset (\image \bdd{\Xi})^*$. However, the presence of the nonstandard space $(\image \bdd{\Xi})^*$ means that neither the well-posedness of \cref{eq:h2-mixed-problem} nor its discretization are clear-cut. For instance, while it is clear from definition \cref{eq:xi-def} that $\image \bdd{\Xi} \subseteq {L}^2(\Omega)^d$, we will show in this section that $\image \bdd{\Xi} \neq {L}^2(\Omega)^d$. That is, $(\image \bdd{\Xi})^*$ contains objects that are less smooth than ${L}^2(\Omega)^d$ which has repercussions both in the analysis of \cref{eq:h2-mixed-problem} and the development of a Galerkin scheme for \cref{eq:h2-mixed-problem}. 

The first step in analyzing the mixed problem \cref{eq:h2-mixed-problem} is to develop a concrete characterization for the image space $\image \bdd{\Xi}$. For $(\tilde{v}, \bdd{\psi}) \in  H^1_{\Gamma}(\Omega) \times \bdd{G}_{\Gamma}(\Omega)$, there holds 
\begin{align*}
	\bdd{\Xi}(\tilde{v}, \bdd{\psi}) = \grad \tilde{v} - \bdd{\psi} \in {L}^2(\Omega)^d \quad \text{and} \quad \curl \bdd{\Xi}(\tilde{v}, \bdd{\psi}) = -\curl \bdd{\psi} \in L^2(\Omega)^{d^*}, 
\end{align*}
where $d^* = 1$ if $d = 2$ and $d^* = 3$ if $d = 3$. Here, we use $\curl$ to denote either the scalar curl of a 2D vector  (i.e. $\curl \bdd{\theta} := \partial_x \theta_2 - \partial_y \theta_1$) or the usual curl operator of a 3D vector depending on the context. Consequently, $\image \bdd{\Xi} \subseteq \hcurl$, where
\begin{align*}
	\hcurl := \{ \bdd{\theta} \in {L}^2(\Omega)^d :  \curl \bdd{\theta}  \in 
	L^2(\Omega)^{d^*} 
	\} \quad \text{with} \quad \|\bdd{\theta}\|_{\curl}^2 := \|\bdd{\theta}\|^2 + \|\curl \bdd{\theta}\|^2.
\end{align*}

We shall also need to characterize the traces of functions in $\image \bdd{\Xi}$. The tangential trace operator $\Tr$ given by
\begin{align*}
	\hcurl \ni \bdd{\theta} \mapsto \Tr \bdd{\theta} := \begin{cases}
		\unitvec{t} \cdot \bdd{\theta}|_{\Gamma} & \text{if } d = 2 \\
		\bdd{\pi}_{T} \bdd{\theta} |_{\Gamma} & \text{if } d = 3
	\end{cases} \in H^{-1/2}(\Gamma)^{d^*}
\end{align*}
is well-defined on $\hcurl$ (see e.g. \cite[p. 27 Theorem 2.5]{GiraultRaviart86} for the case $d = 2$ and \cite{Buffa01I,Buffa01II} for the case $d=3$). For any $\hat{w} \in C^{\infty}(\bar{\Omega})^{d^*}$ with $\hat{w}|_{\Gamma_f} = 0$, there holds
\begin{align*}
	\langle \Tr \bdd{\Xi}(\tilde{v}, \bdd{\psi}), \hat{w} \rangle_{\Gamma} = \langle \Tr \grad \tilde{v}, \hat{w} \rangle_{\Gamma} + (\Tr \bdd{\psi}, \hat{w})_{\Gamma} = \begin{cases}
		-(\tilde{v}, \unitvec{t} \cdot \grad  \hat{w})_{\Gamma} = 0 & d=2, \\
		-(\tilde{v}, \dive \bdd{\pi}_{T } \hat{w})_{\Gamma} = 0 & d=3,
	\end{cases}
\end{align*}
where $\langle \cdot, \cdot \rangle_{\Gamma}$ denotes the extension of the $L^2(\Gamma)^{d^*}$ inner-product to the trace space and its dual. By density (see e.g. \cite[Theorem 3.1]{Bernard11}), we obtain
\begin{align*}
	\langle \Tr \bdd{\Xi}(\tilde{v}, \bdd{\psi}) , \hat{w} \rangle_{\Gamma} = 0 \qquad \forall \hat{w} \in H^1(\Omega)^{d^*} : w|_{\Gamma_f} \equiv 0. 
\end{align*} 
As a result, $\image \bdd{\Xi}$ is contained in the closed subspace of $\hcurl$ consisting of functions with vanishing tangential trace on $\Gamma \setminus \Gamma_f$:
\begin{align*}
	\hcurlgamma := \{ \bdd{\theta} \in \hcurl : \langle \Tr \bdd{\theta}, \hat{w} \rangle_{\Gamma}  = 0 \quad \forall \hat{w} \in H^1(\Omega)^{d^*} : w|_{\Gamma_f} \equiv 0 \}.
\end{align*}
The following result shows that in fact $\image \bdd{\Xi}$ can be identified with $\hcurlgamma$:
\begin{lemma}
	\label{lem:hrotgamma-grad-h10-decomp}
	For every $\bdd{\theta} \in \hcurlgamma$, there exists  $(\tilde{v}, \bdd{\psi}) \in H^1_{\Gamma}(\Omega) \times \bdd{G}_{\Gamma}(\Omega)$ such that
	\begin{align}
		\label{eq:hrotgamma-h10-h10-decomposition-equiv-norm}
		\bdd{\Xi}(\tilde{v}, \bdd{\psi}) = \bdd{\theta} \quad \text{and} \quad C_1 \left( \|\tilde{v}\|_1 + \|\bdd{\psi}\|_1 \right) \leq   \|\bdd{\theta}\|_{\curl} \leq C_2 \left( \|\tilde{v}\|_1 + \|\bdd{\psi}\|_1 \right),
	\end{align}
	where $C_1$ and $C_2$ are positive constants independent of $\bdd{\theta}$. 	Consequently, $\image \bdd{\Xi} = \hcurlgamma$.
\end{lemma}
\begin{proof}
	Let $d = 3$. Thanks to \cite[Theorem 5.2]{Hiptmair09}, for every $\bdd{\theta} \in \hcurlgamma$, there exists $(\tilde{v}, \bdd{\psi}) \in H^1_{\Gamma}(\Omega) \times H^1_{\Gamma}(\Omega)^3 \subset H^1_{\Gamma}(\Omega) \times \bdd{G}_{\Gamma}(\Omega)$ such that \cref{eq:hrotgamma-h10-h10-decomposition-equiv-norm} holds. The same arguments can be appropriately adjusted for the case $d=2$. Thus, $\hcurlgamma \subseteq \image \bdd{\Xi}$. The result now follows on noting that we have already shown that $\image \bdd{\Xi} \subseteq \hcurlgamma$.
\end{proof}

Thanks to \cref{lem:hrotgamma-grad-h10-decomp}, we have $\image \bdd{\Xi} = \hcurlgamma$, and so $\bdd{\Phi}_{\Gamma}(\Omega) = (\image \bdd{\Xi})^* =  \hcurlgamma^*$, which we equip with the standard dual norm
\begin{align*}
	\| \bdd{\eta} \|_{*} := \sup_{ \bdd{0} \neq \bdd{\theta} \in \hcurlgamma } \frac{ \langle \bdd{\theta}, \bdd{\eta} \rangle  }{ \|\bdd{\theta}\|_{\curl} } \qquad \forall \bdd{\eta} \in \hcurlgamma^*.	
\end{align*}
The following inf-sup condition will be useful in establishing the well-posedness of \cref{eq:h2-mixed-problem}:
\begin{corollary}
	\label{cor:inf-sup-continuous}
	There holds
	\begin{align}
		\label{eq:inf-sup-continuous}
		\beta := \inf_{ \substack{\bdd{\eta} \in \hcurlgamma^* \\ \bdd{\eta} \neq \bdd{0} } } \sup_{ \substack{ (\tilde{v}, \bdd{\psi}) \in H^1_{\Gamma}(\Omega) \times \bdd{G}_{\Gamma}(\Omega) \\ (\tilde{v}, \bdd{\psi}) \neq (0, \bdd{0}) } } \frac{ \langle \bdd{\Xi}(\tilde{v}, \bdd{\psi}), \bdd{\eta} \rangle }{ (\|\tilde{v}\|_1 + \|\bdd{\psi}\|_1) \|\bdd{\eta}\|_{*} } > 0.
	\end{align}
\end{corollary}
\begin{proof}
	Let $\bdd{\eta} \in \hcurlgamma^*$ be given and let $\bdd{\theta} \in \hcurlgamma$ be its Riesz representer. By \cref{lem:hrotgamma-grad-h10-decomp}, there exists $(\tilde{v}, \bdd{\psi}) \in H^1_{\Gamma}(\Omega) \times \bdd{G}_{\Gamma}(\Omega)$ such that $\bdd{\Xi}(\tilde{v}, \bdd{\psi}) = \bdd{\theta}$ and $\|\tilde{v}\|_1 + \|\bdd{\psi}\|_1 \leq C \|\bdd{\theta}\|_{\curl} = \| \bdd{\eta} \|_{*}$. Inequality \cref{eq:inf-sup-continuous} now follows since
	\begin{align*}
		\frac{ \langle \bdd{\Xi}(\tilde{v}, \bdd{\psi}), \bdd{\eta} \rangle }{ (\|\tilde{v}\|_1 + \|\bdd{\psi}\|_1) \|\bdd{\eta}\|_{*} } \geq C^{-1} \frac{ \langle \bdd{\theta}, \bdd{\eta} \rangle }{ \| \bdd{\eta}\|_{*}^2 } = C^{-1}.
	\end{align*}
\end{proof}

\subsection[Characterizing $\Kernel \Xi$]{Characterizing $\Kernel \bdd{\Xi}$} 

The kernel of $\bdd{\Xi}$ is defined by 
\begin{align*}
	\Kernel \bdd{\Xi} := \{ (\tilde{v}, \bdd{\psi}) \in  H^1_{\Gamma}(\Omega) \times \bdd{G}_{\Gamma}(\Omega) : \langle \bdd{\Xi}(\tilde{v}, \bdd{\psi}), \bdd{\eta} \rangle = 0 \ \forall \bdd{\eta} \in \hcurlgamma^* \}
\end{align*}
and has the following simple characterization:
\begin{lemma}
	\label{lem:ker-b}
	There holds
	\begin{align}
		\label{eq:ker-b}
		\Kernel \bdd{\Xi} = \{ (w, \grad w) : w \in H^2_{\Gamma}(\Omega) \}.
	\end{align}
\end{lemma}
\begin{proof}
	Let $(\tilde{v}, \bdd{\psi}) \in \Kernel \bdd{\Xi}$. Then, as shown above, $\bdd{\Xi}(\tilde{v}, \bdd{\psi}) \in \hcurlgamma$, and $\langle \bdd{\Xi}(\tilde{v}, \bdd{\psi}), \bdd{\eta} \rangle = 0$ for all $\bdd{\eta} \in \hcurlgamma^*$. Hence, $\bdd{\Xi}(\tilde{v}, \bdd{\psi}) = \bdd{0}$ or, equally well, $\grad \tilde{v} = \bdd{\psi}$, and so $\tilde{v} \in H^2_{\Gamma}(\Omega)$. Thus, $\Kernel \bdd{\Xi} \subseteq \{ (w, \grad w) : w \in H^2_{\Gamma}(\Omega) \}$. The reverse inclusion follows by definition, and so \cref{eq:ker-b} holds.
\end{proof}

\subsection{Well-posedness}

Thanks to the well-posedness of \cref{eq:h2-bilinear}, there exists $w \in H^2_{\Gamma}(\Omega)$ satisfying \cref{eq:h2-bilinear}. Consequently, $(\tilde{w}, \bdd{\gamma}) = (w, \grad w) \in H^1_{\Gamma}(\Omega) \times \bdd{G}_{\Gamma}(\Omega)$ satisfies \cref{eq:h2-mixed-problem-1} for all $(\tilde{v}, \bdd{\psi}) \in \Kernel \bdd{\Xi}$ and satisfies \cref{eq:h2-mixed-problem-2}. The following result establishes the well-posedness of the mixed formulation \cref{eq:h2-mixed-problem} in the case of general data:
\begin{theorem}
	\label{thm:kirchhoff-mixed-gen}
	Let $L \in (H^1_{\Gamma}(\Omega) \times \bdd{G}_{\Gamma}(\Omega))^*$ and $\bdd{g} \in \hcurlgamma$. Then, there exists a unique solution to the following mixed problem: Find $(\tilde{w}, \bdd{\gamma}) \in H^1_{\Gamma}(\Omega) \times \bdd{G}_{\Gamma}(\Omega)$ and $\bdd{\phi} \in \hcurlgamma^*$ such that
	\begin{subequations}
		\label{eq:h2-mixed-gen-old}
		\begin{alignat}{2}
			a(\bdd{\gamma}, \bdd{\psi}) + c(\tilde{w}, \tilde{v}) + \langle \bdd{\Xi}(v,\bdd{\psi}), \bdd{\phi} \rangle &= L(\tilde{v}, \bdd{\psi}) \qquad & &\forall (\tilde{v}, \bdd{\psi}) \in H^1_{\Gamma}(\Omega) \times \bdd{G}_{\Gamma}(\Omega), \\
			\langle \bdd{\Xi}(\tilde{w}, \bdd{\gamma}), \bdd{\eta} \rangle &= \langle \bdd{g}, \bdd{\eta} \rangle  \qquad & &\forall \bdd{\eta} \in \hcurlgamma^*.
		\end{alignat}
	\end{subequations}
	Additionally, the solution satisfies
	\begin{align}
		\label{eq:h2-mixed-gen-stability}
		\| \tilde{w} \|_1 + \|\bdd{\gamma}\|_1 + \| \bdd{\phi}\|_{*} \leq C \left( \|L\|_{(H^1_{\Gamma}(\Omega) \times \bdd{G}_{\Gamma}(\Omega))^*} + \| \bdd{g} \|_{\curl} \right) \quad \text{and} \quad  \bdd{\Xi}(\tilde{w},  \bdd{\gamma}) = \bdd{g},
	\end{align}
	where $C$ is independent of $L$ and $\bdd{g}$.
\end{theorem}
\begin{proof}
	The proof is a standard application of Babu\v{s}ka-Brezzi theory. \Cref{lem:ker-b} and \cref{eq:b-bilinear-infsup} show that the bilinear form $a(\cdot,\cdot) + c(\cdot,\cdot)$ satisfies an inf-sup condition on $\Kernel \bdd{\Xi}$. The inf-sup condition \cref{eq:inf-sup-continuous} then gives the existence and uniqueness of solutions to \cref{eq:h2-mixed-gen-old} and \cref{eq:h2-mixed-gen-stability}; see e.g. \cite[Theorem 4.2.3]{BoffiBrezziFortin13}.
\end{proof}

The next result shows, in the special case where the data in \cref{eq:h2-mixed-gen-old} is of the form in \cref{eq:h2-mixed-problem}, that $\tilde{w} \in H^1_{\Gamma}(\Omega)$ in fact belongs to $H^2_{\Gamma}(\Omega)$ and coincides with the solution $w \in H^2_{\Gamma}(\Omega)$ to the primal problem \cref{eq:h2-bilinear}:
\begin{corollary}
	\label{cor:h2-mixed-rewrite}
	Let $F_1 \in \bdd{G}_{\Gamma}(\Omega)^*$ and $F_2 \in H^1_{\Gamma}(\Omega)^*$. Then, there exists a unique solution to the following mixed problem: Find $(\tilde{w}, \bdd{\gamma}) \in H^1_{\Gamma}(\Omega) \times \bdd{G}_{\Gamma}(\Omega)$ and $\bdd{\phi} \in \hcurlgamma^*$ such that
	\begin{subequations}
		\label{eq:h2-mixed-rewrite}
		\begin{alignat}{2}
			\label{eq:h2-mixed-rewrite-1}
			a(\bdd{\gamma}, \bdd{\psi}) + c(\tilde{w}, \tilde{v}) + \langle \bdd{\Xi}(v,\bdd{\psi}), \bdd{\phi} \rangle &= F_1(\bdd{\psi}) + F_2(\tilde{v}) \qquad & &\forall (\tilde{v}, \bdd{\psi}) \in H^1_{\Gamma}(\Omega) \times \bdd{G}_{\Gamma}(\Omega), \\
			\label{eq:h2-mixed-rewrite-2}
			\langle \bdd{\Xi}(\tilde{w}, \bdd{\gamma}), \bdd{\eta} \rangle &=  0 \qquad & &\forall \bdd{\eta} \in \hcurlgamma^*.
		\end{alignat}
	\end{subequations}
	Moreover, the solution satisfies $\tilde{w} = w$ and $\bdd{\gamma} = \grad w$, where $w \in H^2_{\Gamma}(\Omega)$ is the solution to problem \cref{eq:h2-bilinear}.
\end{corollary}
\begin{proof}
	The existence, uniqueness, and stability of $(\tilde{w}, \bdd{\gamma}, \bdd{\phi}) \in H^1_{\Gamma}(\Omega) \times \bdd{G}_{\Gamma}(\Omega) \times \hcurlgamma^*$ and the relation $\grad \tilde{w} = \bdd{\gamma}$ follows from \cref{thm:kirchhoff-mixed-gen}. Thanks to \cref{lem:ker-b}, $\tilde{w} \in H^2_{\Gamma}(\Omega)$ and testing \cref{eq:h2-mixed-rewrite-1} with $(\tilde{v}, \bdd{\psi}) = (v, \grad v)$ for $v \in H^2_{\Gamma}(\Omega)$ shows that $\tilde{w}$ satisfies \cref{eq:h2-bilinear}. Since solutions to \cref{eq:h2-bilinear} are unique, $\tilde{w} = w$ and $\bdd{\gamma} = \grad \tilde{w} = \grad w$.
\end{proof}

The variational formulation \cref{eq:h2-mixed-rewrite} can be expressed in a form that avoids the presence of the dual spaces altogether. Specifically, replacing $\bdd{\phi} \in \hcurlgamma^*$ by its Riesz representer $\bdd{q} \in \hcurlgamma$ and identifying the duality pairing with the inner product on $\hcurlgamma$ leads to the following restatement of \cref{cor:h2-mixed-rewrite}:
\begin{corollary}
	\label{cor:h2-mixed-nodual}
	Let $F_1 \in \bdd{G}_{\Gamma}(\Omega)^*$ and $F_2 \in H^1_{\Gamma}(\Omega)^*$. Then, there exists a unique solution to the following mixed problem: Find $(\tilde{w}, \bdd{\gamma}) \in H^1_{\Gamma}(\Omega) \times \bdd{G}_{\Gamma}(\Omega)$ and $\bdd{q} \in \hcurlgamma$ such that
	\begin{subequations}
		\label{eq:h2-mixed-nodual}
		\begin{alignat}{2}
			\label{eq:h2-mixed-nodual-1}
			a(\bdd{\gamma}, \bdd{\psi}) + c(\tilde{w}, \tilde{v}) + 
			(\bdd{\Xi}(\tilde{v},\bdd{\psi}), \bdd{q})_{\curl} &= F_1(\bdd{\psi}) + F_2(\tilde{v}) \qquad & &\forall (\tilde{v}, \bdd{\psi}) \in H^1_{\Gamma}(\Omega) \times \bdd{G}_{\Gamma}(\Omega), \\
			\label{eq:h2-mixed-nodual-2}
			(\bdd{\Xi}(\tilde{w}, \bdd{\gamma}), \bdd{\eta})_{\curl} &=  0 \qquad & &\forall \bdd{\eta} \in \hcurlgamma,
		\end{alignat}
	\end{subequations}
	where $(\cdot, \cdot)_{\curl}$ is the inner-product on $\hcurl$:
	\begin{align*}
		(\bdd{\theta}, \bdd{\eta})_{\curl} := (\bdd{\theta}, \bdd{\eta}) + (\curl \bdd{\theta}, \curl \bdd{\eta}) \qquad \forall \bdd{\theta}, \bdd{\eta} \in \hcurl.
	\end{align*}
	Moreover, the solution satisfies $\tilde{w} = w$ and $\bdd{\gamma} = \grad w$, where $w \in H^2_{\Gamma}(\Omega)$ is the solution to problem \cref{eq:h2-bilinear}.
\end{corollary}

The system \cref{eq:h2-mixed-nodual} is not the only possible mixed formulation of problem \cref{eq:h2-bilinear}. Indeed, simpler mixed formulations are possible \cite{Bramble83,Ciarlet74} depending on the particular form of $B(\cdot,\cdot)$ and the boundary conditions. For example, in the case of the biharmonic equation \cref{eq:biharmonic-a} with boundary conditions $w = \Delta w = 0$ on $\Gamma$, one can use a mixed formulation based on introducing an auxiliary variable $\sigma = -\Delta w$: Find $(w, \sigma) \in H^1_{\Gamma}(\Omega)^2$ such that
\begin{subequations}
	\label{eq:ss-biharmonic-simple-mixed}
	\begin{alignat}{2}
		(\sigma, v) + (\grad w, \grad v) &= 0 \qquad & &\forall v \in H^1_{\Gamma}(\Omega), \\
		(\grad \sigma, \grad u) &= G(u) \qquad & &\forall u \in H^1_{\Gamma}(\Omega),
	\end{alignat}
\end{subequations}
which can be discretized using only $H^1(\Omega)$-conforming finite element spaces. However, the apparently more complicated mixed formulation \cref{eq:h2-mixed-nodual} has an unexpected advantage over \cref{eq:ss-biharmonic-simple-mixed}, and other, mixed schemes in that under quite mild conditions, the \textit{resulting discretization produces the $H^2_{\Gamma}(\Omega)$-conforming solution to the original problem} \cref{eq:h2-bilinear}. This is not the case with other mixed schemes such as \cref{eq:ss-biharmonic-simple-mixed}. The question arises of whether a finite element discretization of the new mixed form might inherit this property at the discrete level and thus provide a way to compute the $H^2_{\Gamma}(\Omega)$-conforming approximation $w_{\mathcal{T}}$ defined by \cref{eq:h2-bilinear-fem}.

\section{Discretization of novel mixed formulation}
\label{sec:derive-mixed-fem-gen}

In order to construct conforming discretizations of the mixed problem \cref{eq:h2-mixed-nodual}, we select finite dimensional spaces $\tilde{\mathbb{W}}_{\Gamma} \subset H_{\Gamma}^1(\Omega)$, $\mathbbb{G}_{\Gamma} \subset \bdd{G}_{\Gamma}(\Omega)$, and $\mathbbb{Q}_{\Gamma} \subset \hcurlgamma$ and consider the following scheme: Find $(\tilde{w}_X, \bdd{\gamma}_X, \bdd{q}_X) \in \tilde{\mathbb{W}}_{\Gamma} \times \mathbbb{G}_{\Gamma} \times \mathbbb{Q}_{\Gamma}$ such that
\begin{subequations}
	\label{eq:h2-mixed-problem-fem}
	\begin{alignat}{2}
		\label{eq:h2-mixed-problem-fem-1}
		a(\bdd{\gamma}_X, \bdd{\psi}) + c(\tilde{w}_X, \tilde{v}) + ( \bdd{\Xi}(\tilde{v}, \bdd{\psi}), \bdd{q}_X )_{\curl} &= F_1(\bdd{\psi}) + F_2(\tilde{v}) \qquad & &\forall (\tilde{v}, \bdd{\psi}) \in  \tilde{\mathbb{W}}_{\Gamma} \times \mathbbb{G}_{\Gamma}, \\
		\label{eq:h2-mixed-problem-fem-2}
		(\bdd{\Xi}(\tilde{w}_X, \bdd{\gamma}_X), \bdd{\eta} )_{\curl} &= 0 \qquad & &\forall \bdd{\eta} \in \mathbbb{Q}_{\Gamma}. 
	\end{alignat}
\end{subequations}
Here, and in what follows, the subscript ``$X$" indicates the dependence of the discrete variables on the finite dimensional spaces $\tilde{\mathbb{W}}_{\Gamma}$, $\mathbbb{G}_{\Gamma}$, and $\mathbbb{Q}_{\Gamma}$.

Of course, in order for this scheme to be well-posed, some additional conditions must be imposed on $\tilde{\mathbb{W}}_{\Gamma}$, $\mathbbb{G}_{\Gamma}$, and $\mathbbb{Q}_{\Gamma}$.  Fortunately, the derivation of the mixed problem \cref{eq:h2-mixed-nodual} in the previous section provides some guidance as to how to choose these spaces. We shall assume the subspaces $\tilde{\mathbb{W}}_{\Gamma}$ and $\mathbbb{G}_{\Gamma}$ satisfy the following condition: 

\vspace{0.5em}

\begin{description}
	\item[(A1)\label{hp:tildewgamma-ggamma-def}] $\tilde{\mathbb{W}} \subset H^1(\Omega)$ and $\mathbbb{G} \subset H^1(\Omega)^d$ are conforming finite-dimensional subspaces and 
	\begin{align*}
		\tilde{\mathbb{W}}_{\Gamma} := \tilde{\mathbb{W}} \cap H^1_{\Gamma}(\Omega) \quad \text{and} \quad \mathbbb{G}_{\Gamma} := \mathbbb{G} \cap \bdd{G}_{\Gamma}(\Omega).	
	\end{align*} 
\end{description}

\noindent Examples of suitable choices for $\tilde{\mathbb{W}}_{\Gamma}$ and $\mathbbb{G}_{\Gamma}$ will be given in \cref{sec:elements-verify-assumption}.

The choice of $\mathbbb{Q}_{\Gamma} \subset \hcurlgamma$ is less obvious. The space $\hcurlgamma$ arose from considering the mapping properties of the operator $\bdd{\Xi}$ in \cref{eq:h2-mixed-problem} which led to the conclusion $\image \bdd{\Xi} = \hcurlgamma$. The situation in the discrete setting \cref{eq:h2-mixed-problem-fem} mirrors that of the continuous problem \cref{eq:h2-mixed-problem} with the main difference being that the operator $\bdd{\Xi}$ is replaced by its restriction to the subspace $\tilde{\mathbb{W}}_{\Gamma} \times \mathbbb{G}_{\Gamma} \subset H^1_{\Gamma}(\Omega) \times \bdd{G}_{\Gamma}(\Omega)$ defined by the rule $\bdd{\Xi}_X : \tilde{\mathbb{W}}_{\Gamma} \times \mathbbb{G}_{\Gamma}  \to \image \bdd{\Xi}_X$,  
\begin{align}
	\label{eq:xip-def}
	\bdd{\Xi}_X(\tilde{v}, \bdd{\psi}) := \grad \tilde{v} - \bdd{\psi} \qquad \forall (\tilde{v}, \bdd{\psi}) \in \tilde{\mathbb{W}}_{\Gamma} \times \mathbbb{G}_{\Gamma}.
\end{align}
The analysis in the continuous setting suggests choosing $\mathbbb{Q}_{\Gamma}$ as follows: 

\vspace{0.5em}

\begin{description}
	\item[(A2)\label{hp:qchoice}] $\mathbbb{Q}_{\Gamma} := \image \bdd{\Xi}_X$.
\end{description}

\vspace{0.5em}

\noindent While attractive from the theoretical viewpoint, in order to implement this choice in practice, it will be useful to study the operator $\bdd{\Xi}_{X}$ more closely.

\subsection[Characterizing $\Kernel \Xi$]{Characterization of $\Kernel \bdd{\Xi}_X$}

The kernel of $\bdd{\Xi}_X$ given by
\begin{align*}
	\Kernel \bdd{\Xi}_X := \{ (\tilde{v}, \bdd{\psi}) \in  \tilde{\mathbb{W}}_{\Gamma} \times \mathbbb{G}_{\Gamma} : ( \bdd{\Xi}_X(\tilde{v}, \bdd{\psi}) , \bdd{\eta} )_{\curl} = 0 \ \forall \bdd{\eta} \in \mathbbb{Q}_{\Gamma} \},
\end{align*}
has the same structure as the kernel of $\bdd{\Xi}$ \cref{eq:ker-b}. In particular, we have the following analogue of \cref{lem:ker-b}:
\begin{lemma}
	\label{lem:ker-xix}
	Let $\tilde{\mathbb{W}}_{\Gamma}$, $\mathbbb{G}_{\Gamma}$, and $\mathbbb{Q}_{\Gamma}$  satisfy \ref{hp:tildewgamma-ggamma-def} and \ref{hp:qchoice} and define
	\begin{align}
		\label{eq:wx-tildewx-gx}
		\mathbb{W}_{\Gamma} := \{ v \in \tilde{\mathbb{W}}_{\Gamma} : \grad v \in \mathbbb{G}_{\Gamma} \}.
	\end{align}	
	Then, $\mathbb{W}_{\Gamma}$ is a finite dimensional subspace of $H_{\Gamma}^2(\Omega)$, and there holds
	\begin{align}
		\label{eq:ker-xix}
		\Kernel \bdd{\Xi}_X = \{ (v, \grad v) : v \in \mathbb{W}_{\Gamma}  \}.
	\end{align}
\end{lemma}
\begin{proof}
	Since $\tilde{W}_{\Gamma} \subset H^1_{\Gamma}(\Omega)$ and $\mathbbb{G}_{\Gamma} \subset \bdd{G}_{\Gamma}(\Omega)$ are finite dimensional by \ref{hp:tildewgamma-ggamma-def}, $\mathbb{W}_{\Gamma}$ is a finite dimensional subspace of $H_{\Gamma}^2(\Omega)$ thanks to \cref{eq:h2gamma-h1gamma-id}.
	
	Now suppose that $(\tilde{v}, \bdd{\psi}) \in \Kernel \bdd{\Xi}_X$. The condition $( \bdd{\Xi}_X(\tilde{v}, \bdd{\psi}) , \bdd{\eta} )_{\curl} = 0$  for all $\bdd{\eta} \in \image \bdd{\Xi}_X$ then means that $\grad \tilde{v} = \bdd{\psi} \in \mathbbb{G}_{\Gamma}$. By definition, $\tilde{v} \in \mathbb{W}_{\Gamma}$, and so $\Kernel \bdd{\Xi}_X \subseteq \{ (v, \grad v) : v \in \mathbb{W}_{\Gamma} \}$. The reverse inclusion follows by definition, and so \cref{eq:ker-xix} holds.
\end{proof}

\subsection{Well-posedness}

The space $\mathbb{W}_{\Gamma}$ that arose in \cref{eq:wx-tildewx-gx} from considering the discrete kernel provides a convenient way to express necessary and sufficient conditions for  the discrete mixed problem \cref{eq:h2-mixed-problem-fem} to be well-posed.
\begin{lemma}
	\label{lem:h2-mixed-fem-well-posed}
	Let $\tilde{\mathbb{W}}_{\Gamma}$, $\mathbbb{G}_{\Gamma}$, and $\mathbbb{Q}_{\Gamma}$ satisfy \ref{hp:tildewgamma-ggamma-def} and \ref{hp:qchoice} with $\mathbb{W}_{\Gamma}$ defined as in \cref{eq:wx-tildewx-gx}. Then, the mixed problem \cref{eq:h2-mixed-problem-fem} is uniquely solvable for any $F_1 \in \mathbbb{G}_{\Gamma}^*$ and $F_2 \in \tilde{\mathbb{W}}_{\Gamma}^*$ if and only if $\mathbb{W}_{\Gamma}$ satisfies the inf-sup condition: 
	
	\vspace{0.5em}
	
	\begin{description}
		\item[(A3)\label{hp:b-bilinear-infsup}] There exists $\alpha_X > 0$ such that
		\begin{align}
			\label{eq:b-bilinear-infsup-discrete-gen}
			\alpha_{X} &:= \inf_{0 \neq u \in\mathbb{W}_{\Gamma}} \sup_{0 \neq v \in \mathbb{W}_{\Gamma}} \frac{ B(u, v) }{ \|u\|_2 \|v\|_2 }. 
		\end{align}
	\end{description}
\end{lemma}
\begin{proof}
	Suppose first that \ref{hp:b-bilinear-infsup} holds. Since \cref{eq:h2-mixed-problem-fem} is a finite dimensional linear system, existence follows from uniqueness. To this end, let  $(\tilde{w}_X, \bdd{\gamma}_X, \bdd{q}_X) \in \tilde{\mathbb{W}}_{\Gamma} \times \mathbbb{G}_{\Gamma} \times \mathbbb{Q}_{\Gamma}$ be a solution to \cref{eq:h2-mixed-problem-fem} with $F_1 \equiv 0$ and $F_2 \equiv 0$. Choosing $\bdd{\eta} = \bdd{\Xi}_X(\tilde{w}_X, \bdd{\gamma}_X)$ in \cref{eq:h2-mixed-problem-fem-2} shows that $\bdd{\Xi}_X(\tilde{w}_X, \bdd{\gamma}_X) \equiv 0$. Consequently, $\tilde{w}_X \in \mathbb{W}_{\Gamma}$ by \cref{lem:ker-xix} and $\grad \tilde{w}_X = \bdd{\gamma}_X$. Choosing $(\tilde{v}, \bdd{\psi}) = (v, \grad v)$ for $v \in \mathbb{W}_{\Gamma}$ in \cref{eq:h2-mixed-problem-fem-1} and applying \cref{eq:b-bilinear-infsup-discrete-gen} gives $\tilde{w}_X \equiv 0$ and so $\bdd{\gamma}_X \equiv \bdd{0}$. Finally, choosing any  $(\tilde{v}, \bdd{\psi}) \in \tilde{\mathbb{W}}_{\Gamma} \times \mathbbb{G}_{\Gamma}$ such that $\bdd{\Xi}_X(\tilde{v}, \bdd{\psi}) = \bdd{q}_X$ in \cref{eq:h2-mixed-problem-fem-1}  gives $\bdd{q}_X \equiv \bdd{0}$. Therefore, solutions to \cref{eq:h2-mixed-problem-fem} are unique and thus exist.
	
	Now suppose that \cref{eq:h2-mixed-problem-fem} is uniquely solvable. Thanks to \cref{lem:ker-xix}, the problem
	\begin{align}
		\label{eq:proof:h2gamma-gen-problem}
		w_X \in \mathbb{W}_{\Gamma} : \qquad B(w_X, v) = F_1(\grad v) + F_2(v) \qquad \forall v \in \mathbb{W}_{\Gamma}
	\end{align}
	is also solvable. Again, since \cref{eq:proof:h2gamma-gen-problem} is a finite dimensional square system, solutions to \cref{eq:proof:h2gamma-gen-problem} are therefore also unique, and as a result, $\alpha_X$ defined in \cref{eq:b-bilinear-infsup-discrete-gen} is positive. Indeed, if this were not the case, then there would exist $w_X \in \mathbb{W}_{\Gamma} \setminus \{0\}$ such that $B(w_X, v) = 0$ for all $v \in \mathbb{W}_{\Gamma}$, which would contradict the uniqueness of solutions to \cref{eq:proof:h2gamma-gen-problem}.
\end{proof}

\Cref{lem:h2-mixed-fem-well-posed} shows that given spaces that satisfy \ref{hp:tildewgamma-ggamma-def} and \ref{hp:qchoice}, the well-posedness of the mixed problem \cref{eq:h2-mixed-problem-fem} boils down to checking the inf-sup condition \cref{eq:b-bilinear-infsup-discrete-gen}. As shown in the proof of \cref{lem:h2-mixed-fem-well-posed}, the strict positivity of $\alpha_X$ is equivalent to the well-posedness of \cref{eq:proof:h2gamma-gen-problem}. 

The variational problem \cref{eq:proof:h2gamma-gen-problem} naturally arose from the novel mixed formulation and is nothing more than a $H^2_{\Gamma}(\Omega)$-conforming Galerkin approximation to the original primal problem \cref{eq:h2-bilinear}. The following result shows that the component $\tilde{w}_X$ of the solution to \cref{eq:h2-mixed-problem-fem} coincides exactly with this $\mathbb{W}_{\Gamma}$-conforming approximation to \cref{eq:h2-bilinear-fem}, and mirrors the result at the continuous level given in \cref{cor:h2-mixed-nodual}.

\begin{theorem}
	\label{lem:h2-mixed-problem-fem}
	Let $\tilde{\mathbb{W}}_{\Gamma}$, $\mathbbb{G}_{\Gamma}$, and $\mathbbb{Q}_{\Gamma}$ satisfy \ref{hp:tildewgamma-ggamma-def}-\ref{hp:b-bilinear-infsup}. Then, for any $F_1 \in \mathbbb{G}_{\Gamma}^*$ and $F_2 \in \tilde{\mathbb{W}}_{\Gamma}^*$, the unique solution $(\tilde{w}_X, \bdd{\gamma}_X, \bdd{q}_X) \in \tilde{\mathbb{W}}_{\Gamma} \times \mathbbb{G}_{\Gamma} \times \mathbbb{Q}_{\Gamma}$ to \cref{eq:h2-mixed-problem-fem} satisfies $\tilde{w}_X = w$ and $\bdd{\gamma}_X = \grad w_X$, where $w_X$  is the $H^2$-conforming $\mathbb{W}_{\Gamma}$ \cref{eq:wx-tildewx-gx} approximation defined by \cref{eq:proof:h2gamma-gen-problem}. In particular, $w_X$ is independent of the  splitting \cref{eq:b-bilinear-splitting} of the bilinear form $B(\cdot,\cdot)$. 
\end{theorem}
\begin{proof}
	\Cref{lem:h2-mixed-fem-well-posed} shows that \cref{eq:h2-mixed-problem-fem} is well-posed. The same arguments used in the proof of \cref{lem:h2-mixed-fem-well-posed} also show that for any $F_1 \in \mathbbb{G}_{\Gamma}^*$ and $F_2 \in \tilde{\mathbb{W}}_{\Gamma}^*$, $\grad \tilde{w}_X = \bdd{\gamma}_X$. By definition, $\tilde{w}_X \in \mathbb{W}_{\Gamma} \subset H^2_{\Gamma}(\Omega)$. Choosing $(\tilde{v}, \bdd{\psi}) =(v, \grad v)$ for $v \in \mathbb{W}_{\Gamma}$ in \cref{eq:h2-mixed-problem-fem-1} then shows that $B(\tilde{w}_X, v) = F_1(\grad v) + F_2(v)$ for all $v \in \mathbb{W}_{\Gamma}$. Since \ref{hp:b-bilinear-infsup} ensures that solutions to \cref{eq:proof:h2gamma-gen-problem} are unique, $\tilde{w}_X = w_X$.
\end{proof}
\Cref{thm:kirchhoff-mixed-gen} shows that under assumptions \ref{hp:tildewgamma-ggamma-def}-\ref{hp:b-bilinear-infsup}, the mixed formulation \cref{eq:h2-mixed-problem-fem} recovers the $\mathbb{W}_{\Gamma}$-conforming approximation \cref{eq:proof:h2gamma-gen-problem} of the original $H^2_{\Gamma}(\Omega)$ problem, where $\mathbb{W}_{\Gamma}$ is defined by \cref{eq:wx-tildewx-gx}.

\begin{remark}
	\label{rem:any-inner}
	Let $\dinner{\cdot}{\cdot}$ denote any inner-product on $\mathbbb{Q}_{\Gamma}$ and consider the following generalization of \cref{eq:h2-mixed-problem-fem}: Find $(\tilde{w}_X, \bdd{\gamma}_X, \bdd{q}_X) \in \tilde{\mathbb{W}}_{\Gamma} \times \mathbbb{G}_{\Gamma} \times \mathbbb{Q}_{\Gamma}$ such that
	\begin{subequations}
		\label{eq:h2-mixed-problem-fem-any-inner}
		\begin{alignat}{2}
			\label{eq:h2-mixed-problem-fem-any-inner-1}
			a(\bdd{\gamma}_X, \bdd{\psi}) + c(\tilde{w}_X, \tilde{v}) + \dinner{\bdd{\Xi}(\tilde{v}, \bdd{\psi})}{\bdd{q}_X} &= F_1(\bdd{\psi}) + F_2(\tilde{v}) \qquad & &\forall (\tilde{v}, \bdd{\psi}) \in  \tilde{\mathbb{W}}_{\Gamma} \times \mathbbb{G}_{\Gamma}, \\
			\label{eq:h2-mixed-problem-fem-any-inner-2}
			\dinner{\bdd{\Xi}(\tilde{w}_X, \bdd{\gamma}_X)}{\bdd{\eta}} &= 0 \qquad & &\forall \bdd{\eta} \in \mathbbb{Q}_{\Gamma}. 
		\end{alignat}
	\end{subequations} 
	Then, \cref{lem:ker-xix}, \cref{lem:h2-mixed-fem-well-posed}, and \cref{lem:h2-mixed-problem-fem} all hold with \cref{eq:h2-mixed-problem-fem} replaced by \cref{eq:h2-mixed-problem-fem-any-inner} with the exact same proof. The flexibility to choose the inner-product will turn out to be helpful later when we consider a time-dependent problem in \cref{sec:time-dependent-problems}.
\end{remark}

Remarkably, one obtains $w \in \mathbb{W}_{\Gamma}$ without having to construct a basis for the $H^2_{\Gamma}(\Omega)$-conforming space $\mathbb{W}_{\Gamma}$. As such, the preceding framework provides a vehicle whereby one can compute an $H^2_{\Gamma}(\Omega)$-conforming approximation while avoiding having to deal with $H^2_{\Gamma}(\Omega)$-conforming spaces.
Of course, this begs the question of what exactly one is computing given the rather indirect construction of the space $\mathbb{W}_{\Gamma}$. Ideally, one would prefer to start by specifying the $H^2_{\Gamma}(\Omega)$-conforming space $\mathbb{W}_{\Gamma}$. The problem then is how to choose appropriate spaces $\tilde{\mathbb{W}}_{\Gamma}$ and $\mathbbb{G}_{\Gamma}$ satisfying \ref{hp:tildewgamma-ggamma-def} that give rise to the desired space $\mathbb{W}_{\Gamma}$.

\subsection{Choices of $\tilde{\mathbb{W}}_{\Gamma}$ and $\mathbbb{G}_{\Gamma}$ for common $H^2$-conforming spaces $\mathbb{W}_{\Gamma}$}
\label{sec:elements-verify-assumption}

We now show how to choose finite dimensional spaces $\tilde{\mathbb{W}}_{\Gamma}$ and $\mathbbb{G}_{\Gamma}$ so that the space $\mathbb{W}_{\Gamma}$ defined by \cref{eq:wx-tildewx-gx} coincides with existing finite element spaces. A useful tool will be the following result that removes the need to explicitly consider boundary conditions:
\begin{lemma}
	\label{lem:fem-spaces-compat-bcs}
	Suppose that $\mathbb{W} \subset H^2(\Omega)$, $\tilde{\mathbb{W}} \subset H^1(\Omega)$, and $\mathbbb{G} \subset \bdd{H}^1(\Omega)$ satisfy 
	\begin{align}
		\label{eq:w-tildew-g-char}
		\mathbb{W} = \{ \tilde{v} \in \tilde{\mathbb{W}} : \grad \tilde{v} \in \mathbbb{G} \}.
	\end{align}
	Then, $\mathbb{W}_{\Gamma} := \mathbb{W} \cap H^2_{\Gamma}(\Omega)$, $\tilde{\mathbb{W}}_{\Gamma} := \tilde{\mathbb{W}} \cap H^1_{\Gamma}(\Omega)$, and $\mathbbb{G}_{\Gamma} := \mathbbb{G} \cap \bdd{G}_{\Gamma}(\Omega)$ satisfy \cref{eq:wx-tildewx-gx}.
\end{lemma}
\begin{proof}
	Let $v \in \mathbb{W}_{\Gamma}$ be given. Thanks to \cref{eq:w-tildew-g-char}, there holds $v \in \tilde{\mathbb{W}}$ and $\grad v \in \mathbbb{G}$. Since $v|_{\Gamma_{cs}} = 0$ and $\partial_n v|_{\Gamma_c} = 0$, there holds $v \in \tilde{\mathbb{W}}_{\Gamma}$ and $\grad v \in \mathbbb{G}_{\Gamma}$. Thus, $\mathbb{W}_{\Gamma} \subseteq \{ \tilde{v} \in \tilde{\mathbb{W}}_{\Gamma} : \grad \tilde{v} \in \mathbbb{G}_{\Gamma} \}$.
	
	Now let $\tilde{v} \in \tilde{\mathbb{W}}_{\Gamma}$ with $\grad \tilde{v} \in \mathbbb{G}_{\Gamma}$. By \cref{eq:w-tildew-g-char}, $\tilde{v} \in \mathbb{W}$. Since $\tilde{v}|_{\Gamma_{cs}} = 0$ and $\grad \tilde{v}|_{\Gamma_c} = \bdd{0}$, there holds $\tilde{v} \in \mathbb{W}_{\Gamma}$, and so $\{ \tilde{v} \in \tilde{\mathbb{W}}_{\Gamma} : \grad \tilde{v} \in \mathbbb{G}_{\Gamma} \} \subseteq \mathbb{W}_{\Gamma}$, which completes the proof.
\end{proof}

\subsubsection{Argyris/TUBA element} 

The classical $H^2$-conforming finite element in two dimensions is the Argyris/TUBA element \cite{Argyris68} consisting of degree $p \geq 5$ polynomials with $C^2$ degrees of freedom at element vertices. The element has been extended to three dimensions by {\v{Z}}en{\'\i}{\v{s}}ek \cite{Zenivsek73} with $p=9$ and generalized to elements of order $p > 9$ by Lai and Schumaker \cite[Chapter 18.2]{Lai07spline}:
\begin{align}
	\label{eq:w-argyis}
	\begin{aligned}
		\mathbb{W} = \{&v \in C^1(\Omega) : v|_{K} \in \mathcal{P}_{p}(K) \ \forall K \in \mathcal{T} \text{ and } \\
		& \text{$v$ is $C^2$ at element vertices and $p \geq 5$ if $d=2$, or } \\
		& \text{$v$ is $C^4$ at element vertices, $C^2$ on element edges, and $p \geq 9$ if $d = 3$}  \}.
	\end{aligned}
\end{align}
The space $\tilde{\mathbb{W}}$ may be taken to be the standard $H^1$-conforming continuous piecewise polynomial space, but the space $\mathbbb{G}$ reflects the additional continuity in the space:
\begin{lemma}
	Let $p \geq 5$ if $d=2$ and $p \geq 9$ if $d=3$ and $\mathbb{W}$ be given by \cref{eq:w-argyis}. If
	\begin{subequations}
		\label{eq:w-argyris-tilde-w-g-choice}
		\begin{align}
			\tilde{\mathbb{W}} &= \{ v \in C(\Omega) : v|_{K} \in \mathcal{P}_{p}(K) \ \forall K \in \mathcal{T} \}, \\
			\mathbbb{G} &= \begin{aligned}[t]
				\{ &\bdd{\theta} \in \bdd{C}(\Omega) : \bdd{\theta}|_{K} \in [\mathcal{P}_{p-1}(K)]^d \ \forall K \in \mathcal{T} \\
				&\text{$\bdd{\theta}$ is $\bdd{C}^1$ at element vertices if $d=2$, or } \\
				&\text{$\bdd{\theta}$ is $\bdd{C}^3$ on element vertices and $\bdd{C}^1$ at element edges if $d=3$}	 \},
			\end{aligned}
		\end{align}
	\end{subequations}
	then \cref{eq:w-tildew-g-char} holds.
\end{lemma}	
\begin{proof}
	The inclusion $\mathbb{W} \subseteq \{ v \in \tilde{W} : \grad v \in \mathbbb{G} \}$ follows by definition. Now let $v \in \tilde{W}$ satisfy $\grad v \in \mathbb{G}$. Then, $v \in C^1(\Omega)$ and $v$ is $C^2$ at element vertices if $d=2$ or $v$ is $C^4$ at element vertices and $C^2$ at element edges if $d=3$, and so $v \in \mathbb{W}$. Consequently, $\{ v \in \tilde{W} : \grad v \in \mathbbb{G} \}\subseteq \mathbb{W}$, which completes the proof of \cref{eq:w-tildew-g-char}.
\end{proof}

\subsubsection{$C^1$ spline space}

Let $\mathbb{W}$ to be the space of $C^1$ polynomials on a simplicial partitioning $\mathcal{T}$ satisfying the conditions in \cref{sec:gen-problem-setting}:
\begin{align}
	\label{eq:w-c1-spline}
	\mathbb{W} = \{ v \in C^1(\Omega) : v|_{K} \in \mathcal{P}_{p}(K) \ \forall K \in \mathcal{T} \}, \qquad p \in \mathbb{N}.
\end{align}
In two dimensions, the space $\mathbb{W}$ is often called the Morgan-Scott space \cite{MorganScott75}. It is known that $\mathbb{W}$ has a local (elementwise) basis for polynomial orders $p \geq 5$ in two dimensions and for $p \geq 9$ in three dimensions \cite[Theorem 17.29]{Lai07spline}. For smaller values of $p$, these spaces need not have a local basis; see e.g. \cite{AlfredPiperSchumaker87} for the case $d=2$ and $p=4$. 
Regardless, $\mathbb{W}$ admits the decomposition \cref{eq:w-tildew-g-char} with $\tilde{\mathbb{W}}$ and $\mathbbb{G}$ consisting of standard continuous piecewise polynomials for all polynomial orders $p \in \mathbb{N}$:
\begin{lemma}
	\label{lem:w-c1-spline-tilde-w-g-choice}
	Let $p \geq 1$ and $\mathbb{W}$ be given by \cref{eq:w-c1-spline}. If
	\begin{subequations}
		\label{eq:w-c1-spline-tilde-w-g-choice}
		\begin{align}
			\tilde{\mathbb{W}} &= \{ v \in C(\Omega) : v|_{K} \in \mathcal{P}_{p}(K) \ \forall K \in \mathcal{T} \}, \\
			\mathbbb{G} &= \{ \bdd{\theta} \in \bdd{C}(\Omega) : \bdd{\theta}|_{K} \in [\mathcal{P}_{p-1}(K)]^d \ \forall K \in \mathcal{T} \},
		\end{align}
	\end{subequations}
	then \cref{eq:w-tildew-g-char} holds.
\end{lemma}	
\begin{proof}
	The inclusion $\mathbb{W} \subseteq \{ v \in \tilde{W} : \grad v \in \mathbbb{G} \}$ follows by definition. Now let $v \in \tilde{\mathbb{W}}$ satisfy $\grad v \in \mathbb{G}$. Then, $v \in C^1(\Omega)$ and so $v \in \mathbb{W}$. Consequently, $\{ v \in \tilde{\mathbb{W}} : \grad v \in \mathbbb{G} \}\subseteq \mathbb{W}$, which completes the proof of \cref{eq:w-tildew-g-char}.
\end{proof}

\subsubsection{$H^2$-conforming macroelements}

$H^2$-conforming spaces are often constructed using macroelements (see e.g. \cite[Chapters 6 \& 18]{Lai07spline}). Let $\mathcal{T}$ be a partition as before, and let $\mathcal{T}_{A}$ denote the Alfeld splitting obtained by subdividing every simplex of $\mathcal{T}$ into $d+1$ subsimplices by connecting the barycenter (or other internal point) of each element to the vertices of the element. Let 
\begin{align*}
	\mathbb{W} = \{ v \in C^1(\Omega) : v|_{K} \in \mathcal{P}_p(K) \ \forall K \in \mathcal{T}_A \}, \qquad p \geq 3,
\end{align*}
so that, in the 2D case with $p=3$, $\mathbb{W}$ coincides with the Hseih-Clough-Tocher (HCT) space \cite{Clough65}. Applying \cref{lem:w-c1-spline-tilde-w-g-choice} to the partition $\mathcal{T}_A$ shows that this choice of $\mathbb{W}$ satisfies \cref{eq:w-tildew-g-char} with
\begin{subequations}
	\label{eq:w-c1-alfeld-tildew-g}
	\begin{align*}
		\tilde{\mathbb{W}} &= \{ v \in C(\Omega) : v|_{K} \in \mathcal{P}_{p}(K) \ \forall K \in \mathcal{T}_A \}, \\
		\mathbbb{G} &= \{ \bdd{\theta} \in \bdd{C}(\Omega) : \bdd{\theta}|_{K} \in [\mathcal{P}_{p-1}(K)]^d \ \forall K \in \mathcal{T}_A \}.
	\end{align*}
\end{subequations}

In the 3D case, one could also use the Worsey-Farin splitting $\mathcal{T}_{WF}$ obtained by dividing every tetrahedron in $\mathcal{T}$ into 12 subtetrahedra; see \cite[Chapter 16.7]{Lai07spline} for the precise definition. The Worsey-Farin $C^1$-spline space for $p \geq 3$ is then defined by
\begin{align*}
	\mathbb{W} = \{ v \in C^1(\Omega) : v|_{K} \in \mathcal{P}_p(K) \ \forall K \in \mathcal{T}_{WF}, \ \text{$v$ is $C^2$ at the incenter of $T$} \ \forall T \in \mathcal{T} \}.
\end{align*}
Theorem 18.11 of \cite{Lai07spline} shows that $\mathbb{W}$ has a local basis when $p=3$, while similar arguments show that the result holds more generally for $p \geq 3$. In fact, it is not necessary to explicitly require $C^2$-continuity at the incenter since, as shown in \cite[Remark 9]{Hecklin2008} and \cite[Theorem 4]{FloaterHu20}, the $C^2$-continuity automatically holds, and so $\mathbb{W}$ may be written in the equivalent form
\begin{align*}
	\mathbb{W} = \{ v \in C^1(\Omega) : v|_{K} \in \mathcal{P}_p(K) \ \forall K \in \mathcal{T}_{WF} \}.
\end{align*}
Applying \cref{lem:w-c1-spline-tilde-w-g-choice} to the partition $\mathcal{T}_{WF}$ shows that the Worsey-Farin space may be characterized in the form \cref{eq:w-tildew-g-char} with $\tilde{\mathbb{W}}$ and $\mathbbb{G}$ given by \cref{eq:w-c1-alfeld-tildew-g} after replacing $\mathcal{T}_A$ with $\mathcal{T}_{WF}$.

\section{Algorithmic implementation}
\label{sec:iter-pen}

Summarizing thus far, we have succeeded in our primary goal of producing a finite element scheme involving at most $H^1$-conforming spaces that delivers $H^2$-conforming approximations to \cref{eq:h2-bilinear}. In principle, by choosing spaces that satisfy \ref{hp:tildewgamma-ggamma-def}-\ref{hp:b-bilinear-infsup} and solving a mixed problem \cref{eq:h2-mixed-problem-fem-any-inner} (which makes no mention of any  $H^2$-conforming spaces), we can recover the conforming approximation to \cref{eq:h2-bilinear}. However, it is not so straightforward to compute the solution to \cref{eq:h2-mixed-problem-fem-any-inner}. One problem is that a basis for the space $\image \bdd{\Xi}_{X}$ in \ref{hp:qchoice} is often not known or available, which means that standard direct discretizations of the mixed problem \cref{eq:h2-mixed-problem-fem-any-inner} cannot be applied. Instead, we propose to apply the iterated penalty method \cite{Brenner08,FortinGlow83,Glowinski84} to solve \cref{eq:h2-mixed-problem-fem-any-inner} which has the advantage of not requiring the construction of a basis for $\image \bdd{\Xi}_{X}$.

Let $\tilde{\mathbb{W}}_{\Gamma}$, $\mathbbb{G}_{\Gamma}$, and $\mathbbb{Q}_{\Gamma}$ satisfy \ref{hp:tildewgamma-ggamma-def}-\ref{hp:b-bilinear-infsup}, let $\mathbb{W}_{\Gamma}$ be given by \cref{eq:wx-tildewx-gx}, and let $\dinner{\cdot}{\cdot}$ be any inner-product on $\image \bdd{\Xi}_X$. We define the following bilinear form on $(\tilde{\mathbb{W}}_{\Gamma} \times \mathbbb{G}_{\Gamma}) \times (\tilde{\mathbb{W}}_{\Gamma} \times \mathbbb{G}_{\Gamma})$:
\begin{align}
	\label{eq:Alambda-def}
	A_{\lambda}(\tilde{w}, \bdd{\gamma}; \tilde{v}, \bdd{\psi}) &:= a(\bdd{\gamma}, \bdd{\psi}) + c(\tilde{w}, \tilde{v}) + \lambda \dinner{ \bdd{\Xi}_X(\tilde{w}, \bdd{\gamma})}{\bdd{\Xi}_X(\tilde{v}, \bdd{\psi})}.
\end{align}
Given $\lambda > 0$  and an initial approximation $(\tilde{u}^0, \bdd{\phi}^0) \in \tilde{\mathbb{W}}_{\Gamma} \times \mathbbb{G}_{\Gamma}$, the iterated penalty method for \cref{eq:h2-mixed-problem-fem-any-inner} reads as follows: For $n = 0,1,\ldots$, define $(\tilde{w}^n, \bdd{\gamma}^n), (\tilde{u}^{n+1}, \bdd{\phi}^{n+1}) \in \tilde{\mathbb{W}}_{\Gamma} \times \mathbbb{G}_{\Gamma}$ by
\begin{subequations}
	\label{eq:iter-penalty-morgan-scott-gen}
	\begin{alignat}{2}
		\label{eq:iter-penalty-morgan-scott-gen-1}
		A_{\lambda}(\tilde{w}^n, \bdd{\gamma}^n; \tilde{v}, \bdd{\psi})  &= F_1(\bdd{\psi}) + F_2(\tilde{v}) \qquad & & \notag \\
		&\qquad - \dinner{ \bdd{\Xi}_X(\tilde{u}^{n}, \bdd{\phi}^{n})}{ \bdd{\Xi}_X (\tilde{v}, \bdd{\psi})  } \qquad & & \forall (\tilde{v}, \bdd{\psi}) \in \tilde{\mathbb{W}}_{\Gamma} \times \mathbbb{G}_{\Gamma}, \\
		\label{eq:iter-penalty-morgan-scott-gen-2}
		\tilde{u}^{n+1} &= \tilde{u}^{n} + \lambda \tilde{w}^n, \qquad & & \\
		\bdd{\phi}^{n+1} &= \bdd{\phi}^{n} + \lambda \bdd{\gamma}^n. \qquad & &
	\end{alignat}
\end{subequations}
Under appropriate conditions (given below) the iterates converge to the solution to \cref{eq:h2-mixed-problem-fem-any-inner}: 
\begin{align*}
	\tilde{w}^n \to \tilde{w}_X, \quad \bdd{\gamma}^n \to \bdd{\gamma}_{X}, \quad \text{and} \quad \bdd{\Xi}_X(\tilde{u}^n, \bdd{\phi}^n) \to \bdd{q}_X.
\end{align*}
The rate of convergence will depend on the choice of spaces $\tilde{\mathbb{W}}_{\Gamma}$ and $\mathbbb{G}_{\Gamma}$ and the choice of inner-product $\dinner{\cdot}{\cdot}$ and can be quantified in terms of the following constants
\begin{alignat}{2}
	\label{eq:inf-sup-xi}
	\beta_X &:= \inf_{ \substack{\bdd{\eta} \in \image \bdd{\Xi}_X \\ \bdd{\eta} \neq \bdd{0} } } \sup_{ \substack{ (\tilde{v}, \bdd{\psi}) \in \tilde{\mathbb{W}}_{\Gamma} \times \mathbbb{G}_{\Gamma} \\ (\tilde{v}, \bdd{\psi}) \neq (0, \bdd{0}) } } \frac{ \dinner{\bdd{\Xi}(\tilde{v}, \bdd{\psi})}{ \bdd{\eta} } }{(\|\tilde{v}\|_1 + \|\bdd{\psi}\|_1) \tnorm{\bdd{\eta}} }, \qquad & &\text{where } \tnorm{\bdd{\eta}}^2 := \dinner{\bdd{\eta}}{\bdd{\eta}},\\
	\Upsilon &:= \sup_{ (\tilde{v}, \bdd{\psi}) \in \tilde{\mathbb{W}}_{\Gamma} \times \mathbbb{G}_{\Gamma} } \frac{ \tnorm{ \bdd{\Xi}_X(\tilde{v}, \bdd{\psi})} }{ \| \tilde{v} \|_1 + \|\bdd{\gamma}\|_1 }. \qquad & & \notag
\end{alignat}
as follows:
\begin{theorem}
	\label{thm:iter-penalty-gen}
	Let $(\tilde{w}_X, \bdd{\gamma}_X, \bdd{q}_X) \in \tilde{\mathbb{W}}_{\Gamma} \times \mathbbb{G}_{\Gamma} \times \image \bdd{\Xi}_X$ satisfy \cref{eq:h2-mixed-problem-fem-any-inner} and let 
	\begin{align}
		\label{eq:lambda0-cond}
		\lambda_0 := \frac{5 M}{4 \beta_X^2} \left( 1 + \frac{2 M}{\alpha_X}  \right)^2,
	\end{align}
	where $M$ and $\alpha_X$ are defined in \cref{eq:ac-bilinear-bounded} and \cref{eq:b-bilinear-infsup-discrete-gen}, respectively.
	For $\lambda > \lambda_0$, the iterates $(\tilde{w}^n, \bdd{\gamma}^n), (\tilde{u}^{n}, \bdd{\phi}^{n}) \in \tilde{\mathbb{W}}_{\Gamma} \times \mathbbb{G}_{\Gamma}$ given by \cref{eq:iter-penalty-morgan-scott-gen} are well-defined, and satisfy
	\begin{align}
		\label{eq:iter-penalty-gen-convergence-1}
		\| \tilde{w}^n - \tilde{w}_X \|_1 + \|\bdd{\gamma}^n - \bdd{\gamma}_X \|_1 &\leq \frac{1}{\beta_X} \left( 1 + \frac{2 M}{\alpha_X} \right) \tnorm{\bdd{\Xi}_X(\tilde{w}^n, \bdd{\gamma}^n)}, \\
		\label{eq:iter-penalty-gen-convergence-2}
		\tnorm{ \bdd{\Xi}_X(\tilde{u}^n, \bdd{\phi}^n) - \bdd{q}_X } &\leq \frac{1}{\beta_X} \left\{  \frac{M}{\beta_X} \left( 1 + \frac{2 M}{\alpha_X} \right) + \Upsilon \lambda \right\} \tnorm{ \bdd{\Xi}_X(\tilde{w}^n, \bdd{\gamma}^n)},
	\end{align}
	where
	\begin{align}
		\label{eq:iter-penalty-gen-xip}
		\tnorm{ \bdd{\Xi}_X(\tilde{w}^n, \bdd{\gamma}^n) } \leq \Upsilon \left\{ \frac{8}{9} \left( \frac{\lambda_0}{\lambda - \lambda_0} \right) \right\}^n  \left( 	\| \tilde{w}^0 - \tilde{w}_X \|_1 + \|\bdd{\gamma}^0 - \bdd{\gamma}_X \|_1 \right).
	\end{align}
	
\end{theorem}
\Cref{thm:iter-penalty-gen} shows that if $\lambda$ is sufficiently large, then the iterative process \cref{eq:iter-penalty-morgan-scott-gen} is well-defined and converges at a geometric rate to the solution to \cref{eq:h2-mixed-problem-fem-any-inner}. The remainder of this section is concerned with the proof of \cref{thm:iter-penalty-gen}, which extends the analysis in \cite[Chapter 13]{Brenner08} to cases where the bilinear form $B(\cdot,\cdot)$ need not be positive semidefinite.

\subsection{Auxiliary results}

For the remainder of this section, let $\dinner{\cdot}{\cdot}$ be an inner-product on $\image \bdd{\Xi}_X$ and let $\tnorm{\cdot}$ denote the induced norm \cref{eq:inf-sup-xi}. The following result is a consequence of \cite[Lemma 12.5.10]{Brenner08}:
\begin{lemma}
	\label{lem:invert-xi-inf-sup}
	Let $\bdd{\Xi}_X$ and $\beta_X$ be defined as in \cref{eq:xip-def} and \cref{eq:inf-sup-xi}. Then $\bdd{\Xi}_X$ has a bounded right inverse $\bdd{L}_X : \image \bdd{\Xi}_X \to \tilde{\mathbb{W}}_{\Gamma} \times \mathbbb{G}_{\Gamma}$ satisfying
	\begin{align*}
		\bdd{\Xi}_X \bdd{L}_X \bdd{\eta} = \bdd{\eta} \quad \text{and} \quad \| \bdd{L}_X \bdd{\eta} \|_1 \leq \beta_X^{-1} \tnorm{\bdd{\eta}}  \qquad \forall \bdd{\eta} \in  \image \bdd{\Xi}_X.
	\end{align*}
\end{lemma}
Next, we examine the inclusion $\ker \bdd{\Xi}_X \subset \tilde{\mathbb{W}}_{\Gamma} \times \mathbbb{G}_{\Gamma}$ where $\ker \bdd{\Xi}_X$ is defined in \cref{eq:ker-xix}. If the bilinear forms $a(\cdot,\cdot)$ and $c(\cdot,\cdot)$ are symmetric positive definite, then one can define the orthogonal complement of $\ker \bdd{\Xi}_X$ in $\tilde{\mathbb{W}}_{\Gamma} \times \mathbbb{G}_{\Gamma}$ in the usual way. The next result extends this notion to more general bilinear forms $a(\cdot,\cdot)$ and $c(\cdot,\cdot)$ satisfying conditions \cref{eq:ac-bilinear-bounded}. 

\begin{lemma}
	\label{lem:ker-xi-ker-xi-perp-decomp}
	The space $\tilde{\mathbb{W}}_{\Gamma} \times \mathbbb{G}_{\Gamma}$ admits the following decomposition:
	\begin{align}
		\label{eq:ker-xi-perp-decomp}
		\tilde{\mathbb{W}}_{\Gamma} \times \mathbbb{G}_{\Gamma} = \mathbb{K}^{\circ} \oplus \Kernel \bdd{\Xi}_X,
	\end{align}
	where 
	\begin{align*}
		\mathbb{K}^{\circ} := \{ (\tilde{w}, \bdd{\gamma}) \in \tilde{\mathbb{W}}_{\Gamma} \times \mathbbb{G}_{\Gamma} : a(\bdd{\gamma}, \bdd{\psi}) + c(\tilde{w}, \tilde{v}) = 0 \ \forall (\tilde{v}, \bdd{\psi}) \in \Kernel \bdd{\Xi}_X \}.
	\end{align*}
	In particular, for every $(\tilde{w}, \bdd{\gamma}) \in \tilde{\mathbb{W}}_{\Gamma} \times \mathbbb{G}_{\Gamma}$, there exists unique $(\tilde{v}, \bdd{\psi}) \in \mathbb{K}^{\circ}$ and $(\tilde{n}, \bdd{\eta}) \in \Kernel \bdd{\Xi}_X$ satisfying $\tilde{w} = \tilde{v} + \tilde{n}$ and $\bdd{\gamma} = \bdd{\psi} + \bdd{\eta}$, with
	\begin{subequations}
		\begin{align}
			\label{eq:ker-xi-perp-decomp-cont}
			\| \tilde{v} \|_1 + \|\bdd{\psi}\|_1 &\leq \frac{1}{\beta_X} \left(  1 + \frac{2 M}{\alpha_X}  \right)  \tnorm{\bdd{\Xi}_X(\tilde{w}, \bdd{\gamma})}, \\
			\label{eq:ker-xi-decomp-cont}
			\| \tilde{n} \|_1 + \|\bdd{\eta}\|_1 &\leq  \| \tilde{w} \|_1 + \| \bdd{\gamma} \|_1 + \frac{1}{\beta_X} \left(  1 + \frac{2 M}{\alpha_X}  \right)  \tnorm{\bdd{\Xi}_X(\tilde{w}, \bdd{\gamma})}.
		\end{align}
	\end{subequations}
\end{lemma}
\begin{proof}
	Let $(\tilde{w}, \bdd{\gamma}) \in \tilde{\mathbb{W}}_{\Gamma} \times \mathbbb{G}_{\Gamma}$ be given. Thanks to \cref{lem:invert-xi-inf-sup}, there exist $(\tilde{u}, \bdd{\gamma}) \in \tilde{\mathbb{W}}_{\Gamma} \times \mathbbb{G}_{\Gamma}$ satisfying
	\begin{align*}
		\bdd{\Xi}_X(\tilde{u}, \bdd{\gamma}) = \bdd{\Xi}_X(\tilde{w}, \bdd{\gamma}) \quad \text{and} \quad \| \tilde{u} \|_1 + \|\bdd{\gamma}\|_1 \leq \beta_X^{-1} \tnorm{ \bdd{\Xi}_X(\tilde{w}, \bdd{\gamma}) }.
	\end{align*}
	The well-posedness of \cref{eq:proof:h2gamma-gen-problem} then gives the existence of $z \in \mathbb{W}_{\Gamma}$ such that
	\begin{align*}
		B(z, v) = a(\bdd{\gamma}, \grad v) + c(\tilde{u}, v) \qquad \forall v \in \mathbb{W}_{\Gamma}
	\end{align*} 
	and
	\begin{align*}
		\|z\|_2 \leq 2 M \alpha_X^{-1} \left( \|\tilde{u}\|_1 + \|\bdd{\gamma}\|_1 \right).
	\end{align*}
	Thanks to \cref{eq:ker-xix}, $(\tilde{v}, \bdd{\psi}) := (\tilde{u} - z, \bdd{\gamma} - \grad z) \in \mathbb{K}^{\circ}$ satisfies $\bdd{\Xi}_X(\tilde{v}, \bdd{\psi}) = \bdd{\Xi}_X(\tilde{w}, \bdd{\gamma})$ and
	\begin{align*}
		\| \tilde{v} \|_1 + \| \bdd{\psi} \|_1 \leq \frac{1}{\beta_X} \left( 1 + \frac{2 M }{\alpha_X}  \right) \tnorm{ \bdd{\Xi}_X(\tilde{w}, \bdd{\gamma}) }.
	\end{align*}
	Choosing $\tilde{n} = \tilde{w} - \tilde{v}$ and $\bdd{\eta} = \bdd{\gamma} - \bdd{\psi}$ gives $(\tilde{n}, \bdd{\eta}) \in \Kernel \bdd{\Xi}_X$ satisfying \cref{eq:ker-xi-decomp-cont}.
\end{proof}

Next, we give a stability result for the $A_{\lambda}(\cdot,\cdot)$ bilinear form:
\begin{lemma}
	\label{lem:ip-iterates-defined}
	Suppose $\lambda > \lambda_0$, where $\lambda_0$ is defined in \cref{eq:lambda0-cond}. Let $(\tilde{w}, \bdd{\gamma}) \in \tilde{\mathbb{W}}_{\Gamma} \times \mathbbb{G}_{\Gamma}$ be decomposed according to \cref{eq:ker-xi-perp-decomp}: i.e. $\tilde{w} = w + \tilde{u}$ and $\bdd{\gamma} = \grad w + \bdd{\phi}$ with $w \in \mathbb{W}_{\Gamma}$ and $(\tilde{u}, \bdd{\phi}) \in \mathbb{K}^{\circ}$. Then, there exists $(\tilde{v}, \bdd{\psi}) \in \mathbb{W}_{\Gamma} \times \mathbbb{G}_{\Gamma}$ such that 
	\begin{align}
		\label{eq:Alam-lower-bound}
		\frac{8}{9} \left( \frac{\lambda_0}{\lambda + \lambda_0 } \right) \|w\|_{2}^2 +  \left( \| \tilde{u} \|_1 + \| \bdd{\phi}\|_1 \right)^2 \leq \frac{8}{9} \left(  \frac{\lambda_0}{\lambda - \lambda_0} \right) A_{\lambda}(\tilde{w}, \bdd{\gamma}; \tilde{v}, \bdd{\psi}),
	\end{align}
	and
	\begin{align}
		\label{eq:Alam-lower-bound-cont-v-psi}
		\|\tilde{v}\|_1 + \|\bdd{\psi}\|_1 \leq 2 \alpha_X^{-1} \|w\| + M^{-1} \left( \|\tilde{u}\|_1 + \|\bdd{\phi}\|_{1} \right),
	\end{align}
	where $A_{\lambda}(\cdot,\cdot;\cdot,\cdot)$ is defined in \cref{eq:Alambda-def}.

	Consequently, for $\lambda > \lambda_0$ and $G \in (\tilde{\mathbb{W}}_{\Gamma} \times \mathbbb{G}_{\Gamma})^*$, there exists a unique solution to the problem: Find $(\tilde{w}, \bdd{\gamma}) \in \tilde{\mathbb{W}}_{\Gamma} \times \mathbbb{G}_{\Gamma}$ such that
	\begin{align}
		\label{eq:Alam-gen-eq}
		A_{\lambda}(\tilde{w}, \bdd{\gamma}; \tilde{v}, \bdd{\psi})  &= G(\tilde{v}, \bdd{\psi}) \qquad \forall (\tilde{v}, \bdd{\psi}) \in \tilde{\mathbb{W}}_{\Gamma} \times \mathbbb{G}_{\Gamma}.
	\end{align}
\end{lemma}
\begin{proof}
	Let $\lambda > \lambda_0$, $(\tilde{w}, \bdd{\gamma}) \in \tilde{\mathbb{W}}_{\Gamma} \times \mathbbb{G}_{\Gamma}$, $w \in \mathbb{W}_{\Gamma}$, and $(\tilde{u}, \bdd{\phi}) \in \mathbb{K}^{\circ}$ be as in the statement of the lemma. The well-posedness of \cref{eq:h2-bilinear-fem} means that there exists a unique solution $v \in \mathbb{W}_{\Gamma}$ of the adjoint problem 
	\begin{align*}
		B(z, v) = (w, z)_2 \qquad \forall z \in \mathbb{W}_{\Gamma}
	\end{align*}
	that depends continuously on the data $\|v\|_2 \leq \alpha_X^{-1} \|w\|_2$, where $\alpha_X$ is the inf-sup constant defined in \cref{eq:b-bilinear-infsup-discrete-gen}. Let $(\tilde{v}, \bdd{\psi})$ be given by
	\begin{align}
		\label{eq:Alam-tildev-psi-choice}
		\tilde{\mathbb{W}}_{\Gamma} \times \mathbbb{G}_{\Gamma} \ni (\tilde{v}, \bdd{\psi}) = (v, \grad v) + M^{-1} (\tilde{u}, \bdd{\phi}).
	\end{align}
	Then, \cref{eq:Alam-lower-bound-cont-v-psi} holds and direct computation gives
	\begin{align*}
		A_{\lambda}(\tilde{w}, \bdd{\gamma}; \tilde{v}, \bdd{\psi}) &= a(\grad w, \grad v) + c(w, v) + \epsilon \lambda \tnorm {\bdd{\Xi}_X(\tilde{u}, \bdd{\phi})}^2 \\
		&\qquad + a(\bdd{\phi}, \grad v) + c(\tilde{u}, v) +  \epsilon(  a(\grad w, \bdd{\phi}) + c(w, \tilde{u}) ) \\
		&\qquad + \epsilon ( a(\bdd{\phi}, \bdd{\phi}) + c(\tilde{u}, \tilde{u}) ) \\
		&= \| w \|_2^2 + \epsilon \lambda \tnorm{ \bdd{\Xi}_X(\tilde{u}, \bdd{\phi}) }^2 + \epsilon(  a(\grad w, \bdd{\phi}) + c(w, \tilde{u}) ) \\
		&\qquad + \epsilon (a(\bdd{\phi}, \bdd{\phi}) + c(\tilde{u}, \tilde{u}) )
	\end{align*}
	with $\epsilon := M^{-1}$. Applying \cref{eq:ker-xi-perp-decomp-cont} gives
	\begin{align*}
		A_{\lambda}(\tilde{w}, \bdd{\gamma}; \tilde{v}, \bdd{\psi}) &\geq \| w \|_2^2 + \frac{\lambda  \beta_X^2}{ M }  \left( 1 + \frac{2 M}{\alpha_X}  \right)^{-2} \left( \| \tilde{u} \|_1 + \| \bdd{\phi}\|_1 \right)^2 \\
		&\qquad - \|w\|_1 (\|\bdd{\phi}\|_1 + \|\tilde{u}\|_1) - \|\bdd{\phi}\|_1^2 - \|\tilde{u}\|^2 \\
		&\geq \left( 1 - \frac{1}{4 \delta} \right) \| w \|_2^2 + \frac{5}{4} \left( \frac{\lambda}{\lambda_0} - \frac{4}{5}(1+\delta) \right) \left( \| \tilde{u} \|_1 + \| \bdd{\phi}\|_1 \right)^2
	\end{align*}
	for any $\delta > 0$. Inequality \cref{eq:Alam-lower-bound} now follows on choosing $\delta = \frac{1}{8}(1 + \frac{\lambda}{\lambda_0})$.
	
	In order to show uniqueness of solutions to \cref{eq:Alam-gen-eq}, suppose that $(\tilde{w}, \bdd{\gamma}) \in \tilde{\mathbb{W}}_{\Gamma} \times \mathbbb{G}_{\Gamma}$ satisfies $A_{\lambda}(\tilde{w}, \bdd{\gamma}; \tilde{v}, \bdd{\psi}) = 0$ for all $(\tilde{v}, \bdd{\psi}) \in \tilde{\mathbb{W}}_{\Gamma} \times \mathbbb{G}_{\Gamma}$. Then \cref{eq:Alam-lower-bound} implies $(\tilde{w}, \bdd{\gamma}) = (0, \bdd{0})$ and uniqueness follows.
\end{proof}

\subsection{Proof of \cref{thm:iter-penalty-gen}}

Let $\tilde{\epsilon}^n := \tilde{w}_X - \tilde{w}^n$, $\bdd{e}^n := \bdd{\gamma}_X - \bdd{\gamma}^n$, and $\bdd{r}^n := \bdd{q}_X - \bdd{\Xi}_X(\tilde{u}^n, \bdd{\phi}^n)$. Subtracting \cref{eq:iter-penalty-morgan-scott-gen-1} from \cref{eq:h2-mixed-problem-fem-any-inner-1} and using that  $\bdd{\Xi}_X(\tilde{w}_X, \bdd{\gamma}_X) \equiv \bdd{0}$ by \cref{eq:h2-mixed-problem-fem-any-inner-2} gives, for $n \in \mathbb{N}_0$,
\begin{align}
	\label{eq:proof:ip-error-n}
	A_{\lambda}(\tilde{\epsilon}^n, \bdd{e}^n; \tilde{v}, \bdd{\psi} ) = -\dinner{ \bdd{r}^n}{\bdd{\Xi}_X(\tilde{v}, \bdd{\psi})} \qquad \forall (\tilde{v}, \bdd{\psi}) \in \tilde{\mathbb{W}}_{\Gamma} \times \mathbbb{G}_{\Gamma}
\end{align}
and $\bdd{r}^{n+1} = \bdd{r}^n - \lambda \bdd{\Xi}_X(\tilde{w}^n, \bdd{\gamma}^n) = \bdd{r}^{n} + \lambda \bdd{\Xi}_X(\tilde{\epsilon}^n, \bdd{e}^n)$. Consequently, we have
\begin{alignat*}{2}
	A_{\lambda}(\tilde{\epsilon}^{n+1}, \bdd{e}^{n+1}; \tilde{v}, \bdd{\psi} ) &= - \dinner{\bdd{r}^n}{\bdd{\Xi}_X(\tilde{v}, \bdd{\psi}))} - \lambda  \dinner{\bdd{\Xi}_X(\tilde{\epsilon}^n, \bdd{e}^n)}{ \bdd{\Xi}_X(\tilde{v}, \bdd{\psi})} \\
	&= A_{\lambda}(\tilde{\epsilon}^{n}, \bdd{e}^{n}; \tilde{v}, \bdd{\psi} ) - \lambda \dinner{ \bdd{\Xi}_X(\tilde{\epsilon}^n, \bdd{e}^n)}{ \bdd{\Xi}_X(\tilde{v}, \bdd{\psi}) } \\
	&= a(\bdd{e}^{n},\bdd{\psi}) + c(\tilde{\epsilon}^{n}, \tilde{v})
\end{alignat*}
for all $(\tilde{v}, \bdd{\psi}) \in \tilde{\mathbb{W}}_{\Gamma} \times \mathbbb{G}_{\Gamma}$. 
Since $(\tilde{\epsilon}^{n+1}, \bdd{e}^{n+1}) \in \mathbb{K}^{\circ}$ by \cref{eq:proof:ip-error-n}, \cref{lem:ip-iterates-defined} and \cref{eq:Alam-tildev-psi-choice} shows that the choice $\tilde{v} = M^{-1} \tilde{\epsilon}^{n+1}$ and $\bdd{\psi} = M^{-1} \bdd{e}^{n+1}$ satisfies
\begin{align*}
	\left( \| \tilde{\epsilon}^{n+1} \|_1 + \| \bdd{e}^{n+1} \|_1 \right)^2 &\leq \frac{8}{9} \left( \frac{\lambda_0}{\lambda - \lambda_0} \right) A_{\lambda}(\tilde{\epsilon}^{n+1}, \bdd{e}^{n+1}; \tilde{v}, \bdd{\psi}) \\
	&=  \frac{8}{9} \left( \frac{\lambda_0}{\lambda - \lambda_0} \right) \left( a(\bdd{e}^{n},\bdd{\psi}) + c(\tilde{\epsilon}^{n}, \tilde{v}) \right) \\
	&\leq \frac{8}{9} \left( \frac{\lambda_0}{\lambda - \lambda_0} \right) \left( \|\bdd{e}^n\|_1 \|\bdd{e}^{n+1}\|_1 + \|\tilde{\epsilon}^n\|_1  \| \tilde{\epsilon}^{n+1} \|_1 \right) \\
	&\leq\frac{8}{9} \left( \frac{\lambda_0}{\lambda - \lambda_0} \right) \left( \| \tilde{\epsilon}^{n} \|_1 + \| \bdd{e}^{n} \|_1 \right) \left( \| \tilde{\epsilon}^{n+1} \|_1 + \| \bdd{e}^{n+1} \|_1 \right). 
\end{align*}
Therefore,
\begin{align*}
	\| \tilde{\epsilon}^{n} \|_1 + \| \bdd{e}^{n} \|_1 \leq \left\{ \frac{8}{9} \left( \frac{\lambda_0}{\lambda - \lambda_0} \right) \right\}^n \left( \| \tilde{\epsilon}^{0} \|_1 + \| \bdd{e}^{0} \|_1 \right),
\end{align*}
and so \cref{eq:iter-penalty-gen-xip} follows from the relations
\begin{align*}
	\tnorm{ \bdd{\Xi}_X(\tilde{w}^n, \bdd{\gamma}^n) } = \tnorm{ \bdd{\Xi}_X(\tilde{\epsilon}^n, \bdd{e}^n) } \leq \Upsilon (\| \tilde{\epsilon}^{n} \|_1 + \| \bdd{e}^{n} \|_1).
\end{align*}

Again using that $(\tilde{\epsilon}^{n}, \bdd{e}^{n}) \in \mathbb{K}^{\circ}$ and applying \cref{eq:ker-xi-perp-decomp-cont} gives \cref{eq:iter-penalty-gen-convergence-1}. Thanks to the inf-sup condition \cref{eq:inf-sup-xi}, there holds
\begin{align*}
	\beta_X \tnorm{ \bdd{r}^n } &\leq \sup_{ \substack{ (\tilde{v}, \bdd{\psi}) \in \tilde{\mathbb{W}}_{\Gamma} \times \mathbbb{G}_{\Gamma} \\ (\tilde{v}, \bdd{\psi}) \neq (0, \bdd{0}) } } \frac{ \dinner{\bdd{\Xi}_X(\tilde{v}, \bdd{\psi})}{ \bdd{r}^n}  }{\|\tilde{v}\|_1 + \|\bdd{\psi}\|_1 } 
	=  \sup_{ \substack{ (\tilde{v}, \bdd{\psi}) \in \tilde{\mathbb{W}}_{\Gamma} \times \mathbbb{G}_{\Gamma} \\ (\tilde{v}, \bdd{\psi}) \neq (0, \bdd{0}) } } \frac{A_{\lambda}(\tilde{\epsilon}^n, \bdd{e}^n; \tilde{v}, \bdd{\psi}) }{\|\tilde{v}\|_1 + \|\bdd{\psi}\|_1 }. 
\end{align*}
Since
\begin{align*}
	A_{\lambda}(\tilde{\epsilon}^n, \bdd{e}^n; \tilde{v}, \bdd{\psi}) 
	&\leq  M \|\bdd{e}^n\|_1 \|\bdd{\psi}\|_1 + M \|\tilde{e}^n\|_1 \|\tilde{v}\|_1 + \lambda \tnorm{ \bdd{\Xi}_X(\tilde{\epsilon}^n, \bdd{e}^n) } \tnorm{ \bdd{\Xi}_X(\tilde{v}, \bdd{\psi}) } \\
	&\leq  \left\{ M (\|\tilde{e}^n\|_1 + \|\bdd{e}^n\|_1) + \Upsilon \lambda \tnorm{ \bdd{\Xi}_X(\tilde{\epsilon}^n, \bdd{e}^n) }  \right\} (\|\tilde{v}\|_1 + \|\bdd{\psi}\|_1) \\
	&\leq \left\{  \frac{M}{\beta_X} \left( 1 + \frac{2 M}{\alpha_X} \right) + \Upsilon\lambda \right\} \tnorm{ \bdd{\Xi}_X(\tilde{\epsilon}^n, \bdd{e}^n) } (\|\tilde{v}\|_1 + \|\bdd{\psi}\|_1).
\end{align*}
Inequality \cref{eq:iter-penalty-gen-convergence-2} now follows. \hfill \qedsymbol

\section{Numerical examples}
\label{sec:numerics}

Firedrake \cite{FiredrakeUserManual,Nixonhill23} is a freely available package that provides a wide variety of finite element schemes. However, it does not offer high order $C^1$-finite element schemes and, as such, it is typical of the types of packages which the current methodology is targeted. In this section, we use Firedrake to compute high order $C^1$-conforming approximations of some representative problems. The code is available at \cite{zenodo/3DC1}.

\subsection{Kirchhoff plate with point load}
\label{sec:kirchhoff-static}

Let $\Omega \times [-\tau/2, \tau/2]$ be an isotropic steel Kirchhoff plate with density $\rho =7820\unit[per-mode=symbol]{\kilogram\per\meter^{3}}$ and thickness $\tau = 10^{-2} \unit{\meter}$, where $\Omega$ is an L-shaped plate with 3 holes as in \cref{fig:kirchhoff-pointload}. The plate is simply-supported on the outer boundary and free on the boundary of the holes.  We apply a constant point-load to the mesh vertex located at $\bdd{z} := (0.66, 0.33)$ and let the plate reach equilibrium. We wish to compute the approximation to the transverse displacement $w_X \in \mathbb{W}_{\Gamma}$ that would be obtained using the Morgan-Scott space \cref{eq:w-c1-spline} of degree $p \geq 2$ as follows:
\begin{align}
	\label{eq:kirchhoff-ex-pointload}
	w_X \in \mathbb{W}_{\Gamma} : \qquad a_{s}(\grad w_X, \grad v) &= 10^3 \cdot v(\bdd{z}) \qquad \forall v \in \mathbb{W}_{\Gamma},
\end{align}
where $a_s(\cdot,\cdot)$ denotes the bilinear form defined in \cref{eq:kirchhoff-a}, the Young's modulus $E = 2.1 \cdot 10^{11} \unit{\pascal}$, the Poisson ration $\nu = 0.3$, and the bending stiffness $D = E \tau^3 / (12 (1-\nu^2))$.

We compute the solution to \cref{eq:kirchhoff-ex-pointload} by solving the mixed problem \cref{eq:h2-mixed-problem-fem} with $\tilde{\mathbb{W}}$ and $\mathbbb{G}$ chosen as in \cref{eq:w-c1-spline-tilde-w-g-choice} with $a(\cdot,\cdot) = a_s(\cdot,\cdot) / D$, $c(\cdot,\cdot) \equiv 0$, $F_1 \equiv 0$, and $F_2(\tilde{v}) = (10^3 / D) \cdot \tilde{v}(\bdd{z})$ using the iterated penalty method \cref{eq:iter-penalty-morgan-scott-gen} with penalty parameter $\lambda = 10^3$ and terminating when $\|\bdd{\Xi}_X(\tilde{w}^n, \bdd{\gamma}^n)\|_{\curl} < 10^{-8}$. The numbers of iterated penalty iterations to achieve convergence were 5 for $p=3$ and 3 for $4 \leq p \leq 15$. The method did not converge in 100 iterations for $p=2$. These results are consistent with the convergence estimate \cref{eq:iter-penalty-gen-xip} given that the inf-sup constant $\beta_X$ \cref{eq:inf-sup-xi} is independent of $p$ as shown in \cref{thm:morgan-scott-inf-sup}. The $p=10$ solution $\tilde{w}_X$ given by the iterated penalty method and its gradient $\grad \tilde{w}_X$, displayed in \cref{fig:kirchhoff-pointload-combo}, show that the upper left corner of the plate displaces in the opposite direction of the load and that $\tilde{w}_X$ is indeed $C^1$.

\begin{figure}[htb]
	\centering
	\begin{subfigure}[b]{0.48\linewidth}
		\centering
		\includegraphics[width=\linewidth]{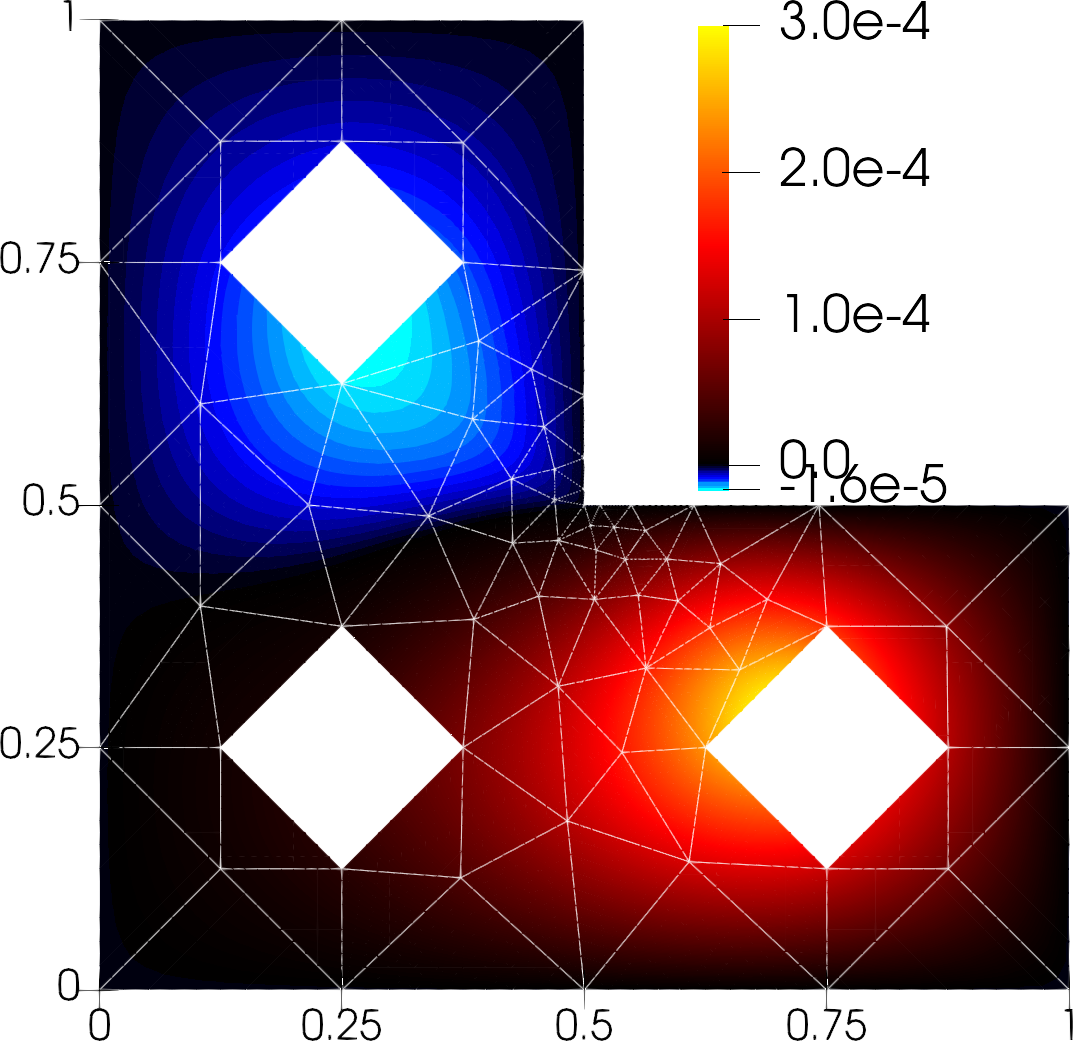}
		\caption{$\tilde{w}_X$}
		\label{fig:kirchhoff-pointload}
	\end{subfigure}
	\hfill
	\begin{subfigure}[b]{0.48\linewidth}
		\centering
		\includegraphics[width=\linewidth]{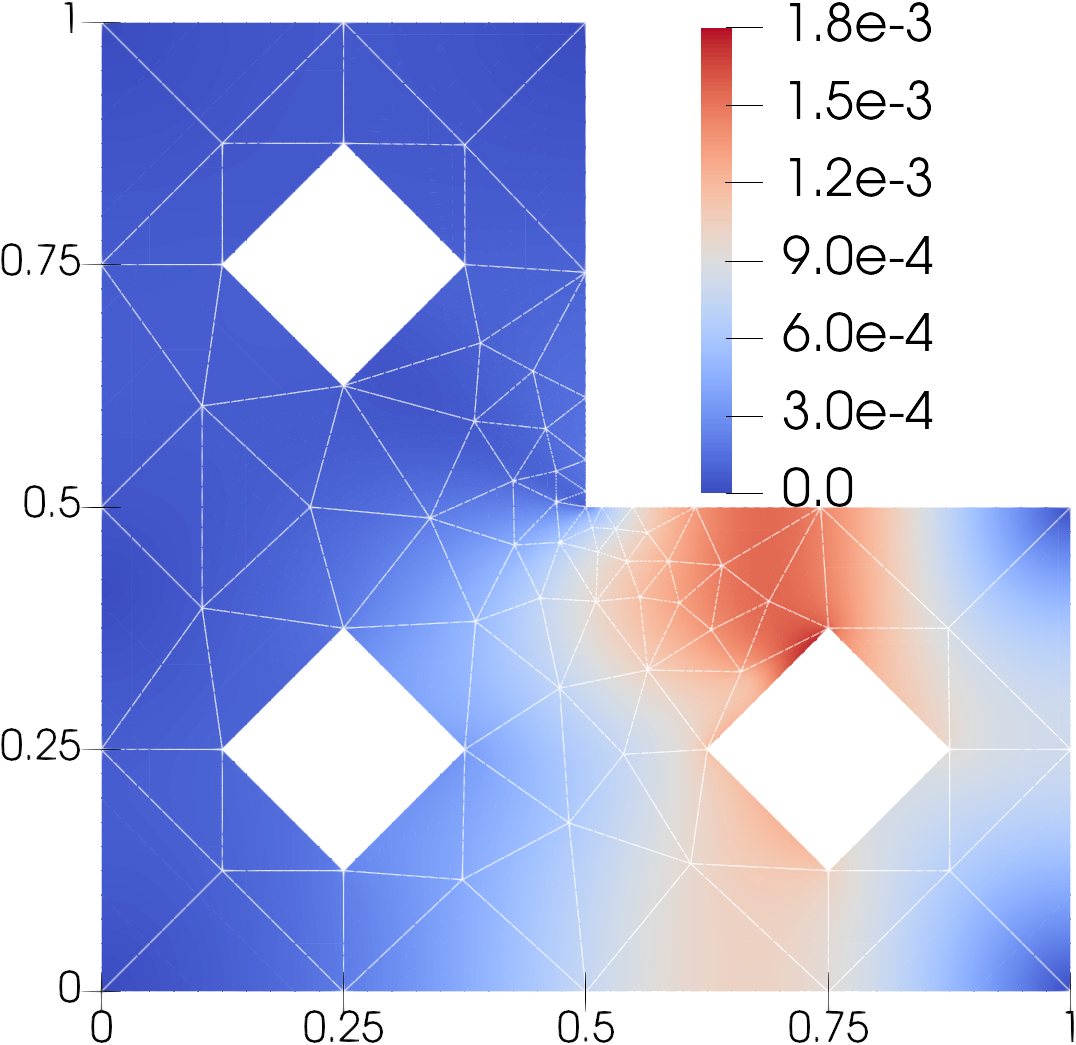}
		\caption{$|\grad \tilde{w}_X|$}
		\label{fig:kirchhoff-pointload-grad}
	\end{subfigure}
	\caption{(a) The displacement $\tilde{w}_X$ and (b) $|\grad \tilde{w}_X|$ for the $p=10$ solution to the Kirchhoff plate problem \cref{eq:kirchhoff-ex-pointload}.}
	\label{fig:kirchhoff-pointload-combo}
\end{figure}

\subsection{Three dimensional example}

Let $\Omega = (0, 1)^3$ be the unit cube and let $\mathbb{W}$ be the $C^1$-spline space \cref{eq:w-c1-spline} defined on a sequence of Freudenthal partitions $\{ \mathcal{T}_m \}$ generated as follows: Given $m \in \mathbb{N}$, subdivide $\Omega$ into $m^3$ congruent subcubes and divide each of these cubes into six tetrahedra (see e.g. \cite{Bey00} for a precise description), so that $\mathcal{T}_m$ is quasi-uniform of size $h := 1/m$. We consider the $H^2$-projection problem
\begin{align}
	\label{eq:3d-projection-problem}
	w_X \in \mathbb{W}_{\Gamma} : \qquad B(w_X, v) &= B(w, v) \qquad \forall v \in \mathbb{W}_{\Gamma},
\end{align}
where $\Gamma_f = \Gamma$, $w = \sin(\pi x) \sin(\pi y) \sin(\pi z)$, and
\begin{align*}
	B(u, v) := (D^2 u, D^2 v) + (\grad u, \grad v) + (u, v) \qquad \forall u, v \in H^2(\Omega).
\end{align*}
Very few finite element packages provide a capability to use $C^1$-conforming elements in the three dimensional setting. Nevertheless, the current approach can be used to obtain the solution to \cref{eq:3d-projection-problem} using Firedrake \cite{FiredrakeUserManual}.

We decompose the bilinear form $B(\cdot,\cdot)$ as in \cref{eq:b-bilinear-splitting} by choosing
\begin{alignat*}{2}
	a(\bdd{\theta}, \bdd{\psi}) &:= (\grad \bdd{\theta}, \grad \bdd{\psi}) \qquad & &\forall \bdd{\theta}, \bdd{\psi} \in \bdd{H}^1(\Omega), \\
	c(u, v) &:= (\grad u, \grad v) + (u, v) \qquad & &\forall u,v \in H^1(\Omega).
\end{alignat*}
The linear functional $F(v) := B(w, v)$ is decomposed as in \cref{eq:F-functional-splitting} in the same manner:
\begin{align*}
	F_1(\bdd{\psi}) := a(\grad w, \bdd{\psi}) \qquad \forall \bdd{\psi} \in \bdd{H}^1(\Omega) \quad \text{and} \quad F_2(v) := c(w, v) \qquad \forall v \in H^1(\Omega).
\end{align*} 
When $\mathbb{W}$ consists of piecewise degree $p \geq 9$ polynomials, the optimal error estimate $\|w - w_X\|_2 \leq C h^{p-1}$ holds on an arbitrary shape-regular partition \cite[Theorem 18.3]{Lai07spline}, in part because a local basis can be constructed. Whether this bound holds for $p < 9$ in general is unclear.

Nevertheless, using the method developed in the current work, we are in a position to compute the approximation in the cases $p < 9$. For $p \in \{2,3,\ldots,9\}$ and the sequence of meshes $\{\mathcal{T}_m\}$, we compute the solution to \cref{eq:3d-projection-problem} by solving the mixed problem \cref{eq:h2-mixed-problem-fem} with $\tilde{\mathbb{W}}$ and $\mathbbb{G}$ chosen as in \cref{eq:w-c1-spline-tilde-w-g-choice} using the iterated penalty method \cref{eq:iter-penalty-morgan-scott-gen} with penalty parameter $\lambda = 10^4$ and terminating when $\|\bdd{\Xi}_X(\tilde{w}^n, \bdd{\gamma}^n)\|_{\curl} < 10^{-8}$. The relative $H^2$ errors $\|w_X - w\|_2 / \|w\|_2$ displayed in \cref{fig:3d-h2-error-h-version} show that the optimal error estimate appears to holds for $p \geq 6$, while the sub-optimal error estimate $\|w - w_X\|_2 \leq C h^{\max\{ p-3, 0\}}$ appears to hold for $p \in \{2,\ldots, 5\}$. The number of iterated penalty iterations displayed in \cref{tab:3d-ip-iters} tell a similar story -- the number of iterations is bounded as the mesh is refined for $p \geq 6$, but grows for $p \in \{2,3,\ldots, 5\}$. Using the same methods and parameters, we also compute the solution to \cref{eq:3d-projection-problem} for $p = \{2,3,\ldots, 12\}$ on the mesh $\mathcal{T}_2$. The relative $H^2$ errors in \cref{fig:3d-h2-error-p-version} appear to decay at an exponential rate, while the number of iterated penalty iterations in \cref{tab:3d-ip-iters} remains bounded as $p$ grows.

\begin{figure}[htb]
	\centering
	\begin{subfigure}[b]{0.48\linewidth}
		\centering
		\includegraphics[width=\linewidth]{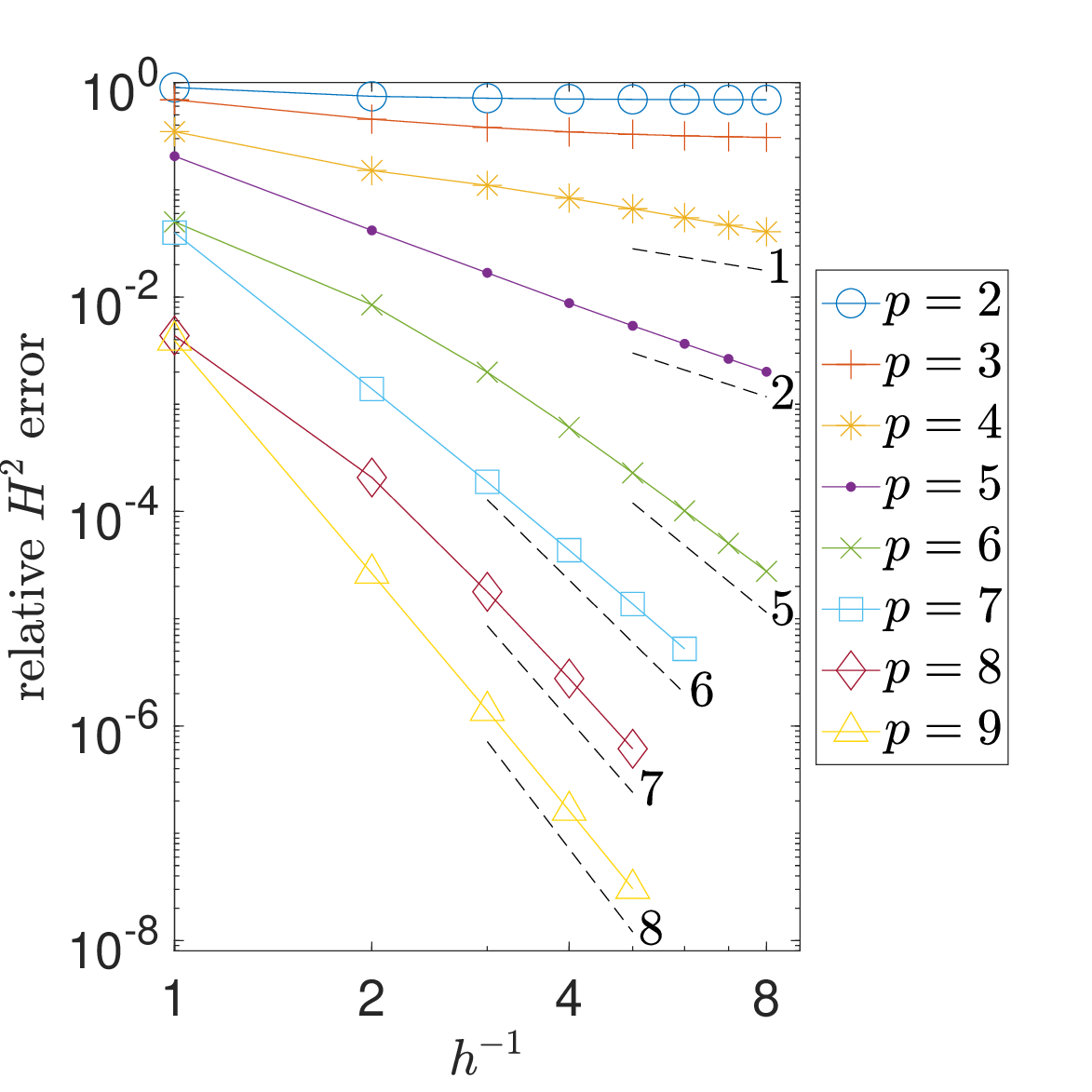}
		\caption{}
		\label{fig:3d-h2-error-h-version}
	\end{subfigure}
	\hfill
	\begin{subfigure}[b]{0.48\linewidth}
		\centering
		\includegraphics[width=\linewidth]{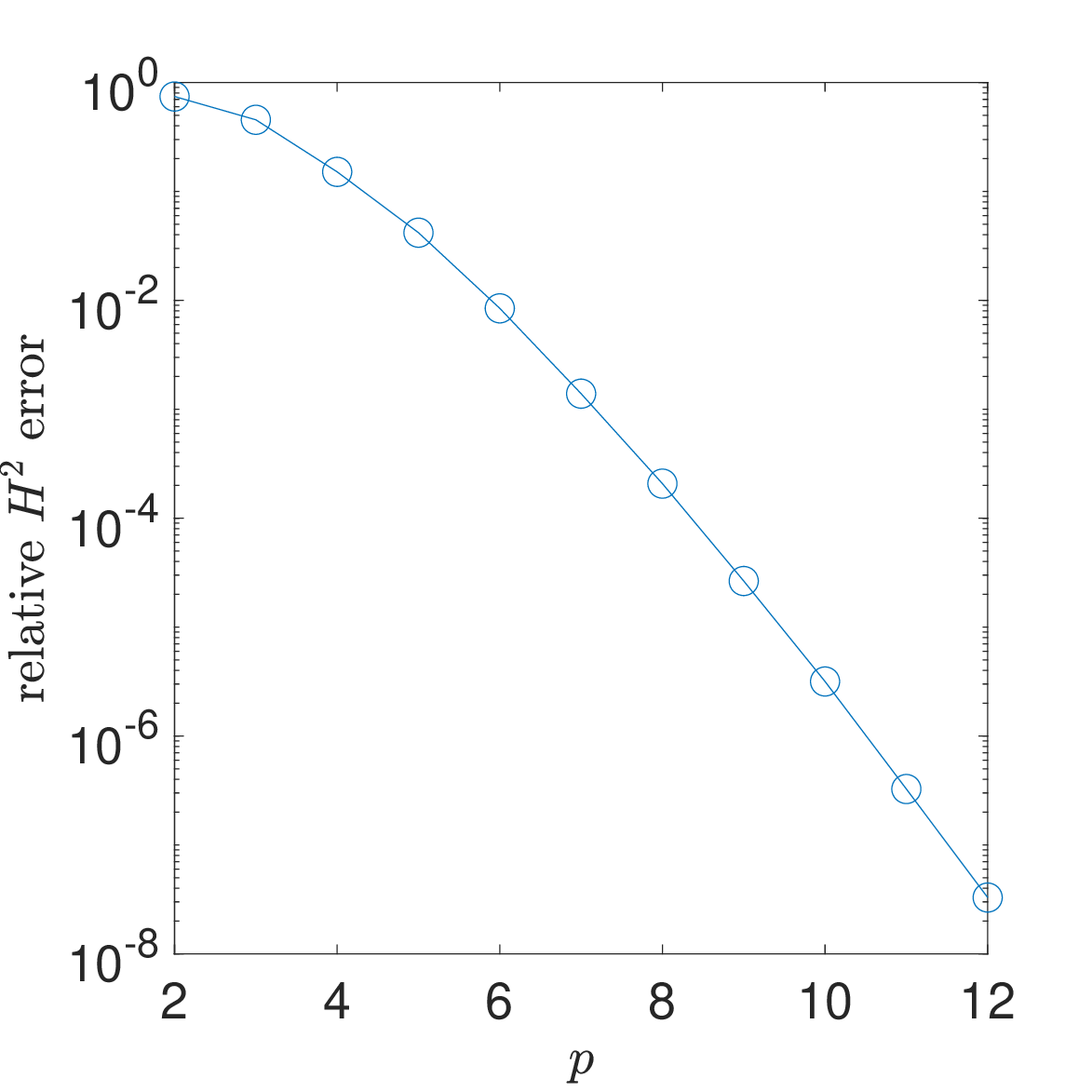}
		\caption{}
		\label{fig:3d-h2-error-p-version}
	\end{subfigure}
	\caption{Relative $H^2$ errors and for the projection problem \cref{eq:3d-projection-problem} with degree-$p$ $C^1$-spline elements \cref{eq:w-c1-spline} (a) on a sequence of meshes with $2 \leq p \leq 9$ and (b) on the single mesh $\mathcal{T}_2$ with $2 \leq p \leq 12$. The number next to each dashed line in (a) indicates its slope.}
\end{figure}

\begin{table}[htb]
	\centering
	\begin{tabular}{ c | c  c  c  c  c  c  c  c  c  c  c  }
		\diagbox{$m$}{$p$} & $2$ & $3$ & $4$ & $5$ & $6$ & $7$ & $8$ & $9$ & $10$ & $11$ & $12$ \\
		\hline
		1 & 3 & 3 & 3 & 3 & 3 & 3 & 2 & 3 & 2 & 2 & 2  \\
		2 & 3 & 4 & 4 & 4 & 4 & 3 & 3 & 2 & 2 & 1 & 1  \\
		3 & 3 & 5 & 5 & 5 & 4 & 3 & 2 & 2 & 1 & 1 & \\
		4 & 3 & 6 & 6 & 8 & 3 & 3 & 2 & 1 & & & \\
		5 & 3 & 8 & 8 & 11 & 3 & 3 & 1 & 1 & & &  \\
		6 & 3 & 11 & 11 & 17 & 3 & 3 & & & & & \\
		7 & 3 & 15 & 16 & 27 & 3 & & & & & & \\
		8 & 4 & 21 & 22 & 40 & 3 & & & & & & \\
	\end{tabular} 
	\caption{Iterated penalty iterations for the projection problem \cref{eq:3d-projection-problem} with degree-$p$ $C^1$-spline elements \cref{eq:w-c1-spline}, $2 \leq p \leq 12$, on the Freudenthal meshes $\{\mathcal{T}_m\}$.}
	\label{tab:3d-ip-iters}
\end{table}

\section{Extension to time-dependent problems?}
\label{sec:time-dependent-problems}

The numerical examples in the previous section show that the $H^2$-conforming finite element approximation \cref{eq:h2-bilinear-fem} can be readily computed without having to implement $C^1$-continuous elements. In this section, we explore the possibility of applying the method to time-dependent problems even though such problems do not satisfy the previous assumptions.

Right off the bat, we run into the problem that the computation of the initial data entails computing the $L^2$-projection of the initial data $g$ onto the $H^2$-conforming space:
\begin{align}
	\label{eq:l2-projection}
	w_0 \in \mathbb{W}_{\Gamma} : \qquad (w_0, v) &= (g, v) \qquad \forall v \in \mathbb{W}_{\Gamma},
\end{align}
where $g \in L^2(\Omega)$. The $L^2$-inner product  $(\cdot,\cdot)$ is not $H^2_{\Gamma}(\Omega)$-elliptic, and hence does not satisfy condition \cref{eq:b-bilinear-infsup}. However, thanks to equivalence of norms on the finite dimensional space $\mathbb{W}_{\Gamma}(\Omega)$, there exists $\varpi_X > 0$, depending on the dimension of $\mathbb{W}_{\Gamma}$, such that
\begin{align}
	\label{eq:l2-equiv-norm}
	\varpi_X \|v\|_2^2 \leq \| v \|^2 \qquad \forall v \in \mathbb{W}_{\Gamma}.
\end{align}
This means that, at the discrete level, the $L^2$-inner product on $\mathbb{W}_{\Gamma}$ satisfies assumptions \cref{eq:b-bilinear-bounded,eq:b-bilinear-infsup-discrete-gen} meaning that we can apply our method and solve the mixed problem \cref{eq:h2-mixed-problem-fem-any-inner} with $c(\cdot,\cdot) = (\cdot,\cdot)$ and $a(\cdot,\cdot) \equiv 0$ to recover the solution to \cref{eq:l2-projection}. However, we expect the number of iterated penalty iterations to degenerate as the discretization is refined since $\varpi_X$ is not uniformly bounded away from zero with respect to the dimension of $\mathbb{W}_{\Gamma}$. 

A similar situation arises in time-stepping, where advancing the approximation in time involves solving problems of the form
\begin{align}
	\label{eq:time-stepping-problem}
	w_n \in \mathbb{W}_{\Gamma} : \qquad (w_n, v) + (\Delta t)^{m} B(w_n, v) = F_n(v) \qquad \forall v \in \mathbb{W}_{\Gamma},
\end{align}
where  $B(\cdot,\cdot)$ satisfies conditions \cref{eq:b-bilinear-bounded,eq:b-bilinear-infsup,eq:b-bilinear-uniqueness,eq:b-bilinear-splitting,eq:b-bilinear-infsup-discrete-gen}, $m$ is the order of the time derivative of the problem, and $F_n \in \mathbb{W}_{\Gamma}^*$ depends on the data from previous time steps. If the bilinear form $B(\cdot,\cdot)$ is elliptic, i.e. $B(v,v) \geq \alpha \|v\|_2^2$ for all $v \in H^2_{\Gamma}(\Omega)$ for some $\alpha > 0$, then there holds
\begin{align*}
	(\varpi_X + \alpha (\Delta t)^m) \|v\|_2^2 \leq \|v\|^2 + (\Delta t)^{m} B(v, v) \qquad \forall v \in  \mathbb{W}_{\Gamma},
\end{align*}
where $\varpi_X$ is given by \cref{eq:l2-equiv-norm}. Once again, this means that the conditions needed for our approach are satisfied at the discrete level, so that the method can be used to solve the mixed problem \cref{eq:h2-mixed-problem-fem-any-inner} to recover the solution to \cref{eq:time-stepping-problem}. However, the fact that the constants in the assumptions depend on the dimension of the discrete space manifests itself through growth in the number of iterated penalty iterations as $\Delta t \to 0$ or as the discretization is refined. Nevertheless, the method can be applied to time-dependent problems subject to the above proviso. We illustrate the performance in the case of a dynamic Kirchhoff plate.

\subsection{Dynamic Kirchhoff plate}

We continue the example in \cref{sec:kirchhoff-static} by removing the point load that was responsible for the initial deflection and consider the resulting vibrations of the plate which are governed by the problem: Find $w(t) \in \mathbb{W}_{\Gamma}$ such that
\begin{align*}
	\frac{d^2}{dt^2} \left( w(t), v \right) + \frac{1}{\rho\tau} a_s(\grad w(t), \grad v)&= 0 \quad  \forall v \in \mathbb{W}_{\Gamma}, 
\end{align*}
subject to initial conditions $w(0) = w_X$ and $\partial_t w(0) = 0$, where $w_X$ is the initial deflection given by \cref{eq:kirchhoff-ex-pointload}. As in the previous section, $a_s(\cdot,\cdot)$ denotes the bilinear form defined in \cref{eq:kirchhoff-a}.

We discretize in time with a Newmark method (see e.g. \cite{Krenk06}) as follows. Let $\Delta t > 0$ be a given time step and $\beta, \delta \geq 0$ be given parameters. For $n = 1,2,\ldots,$ the approximation to $\partial_{tt} w$ at time $t_n := n\cdot \Delta t$ is updated by the following variational formulation: Find $w_n^{[2]} \in \mathbb{W}_{\Gamma}$ such that
\begin{align}
	\label{eq:newmark-dtt-update}
	(w_n^{[2]}, v) +  \frac{\beta (\Delta t)^2}{\rho \tau} a_s(\grad w_n^{[2]}, \grad v) = -\frac{1}{\rho\tau} a_s(\grad \hat{w}_{n-1}, \grad v) \qquad  \forall v \in \mathbb{W}_{\Gamma}, 
\end{align}
where 
\begin{align*}
	\hat{w}_{n-1} :=  w_{n-1} + (\Delta t) w_{n-1}^{[1]} + \left( \frac{1}{2} - \beta \right) (\Delta t)^2 w_{n-1}^{[2]}.
\end{align*}
The approximations to $w$ and $\partial_t w$ at time $t_n$ are then updated by
\begin{align}
	\label{eq:newmark-pos-velocity-update}
	w_{n} &= \hat{w}_{n-1} + \beta (\Delta t)^2 w_{n}^{[2]} 
	\quad \text{and} \quad
	w_{n}^{[1]} = w_{n-1}^{[1]} +  (\Delta t)\left( (1-\delta) w_{n-1}^{[1]} + \delta w_{n}^{[2]} \right).
\end{align}
The initial displacement $w_0 := w_X$ is solution from \cref{sec:kirchhoff-static} and the initial velocity $w_0^{[1]} = 0$, while the initial acceleration $w_0^{[2]}$ is determined by
\begin{align}
	\label{eq:newmark-initial-dtt}
	w_0^{[2]} \in \mathbb{W}_{\Gamma} : \qquad  (w_0^{[2]}, v) = -\frac{1}{\rho \tau} a_s(\grad w_0, \grad v) \qquad \forall v \in \mathbb{W}_{\Gamma},
\end{align}
which is precisely the form of problem in \cref{eq:l2-projection}. If $\delta = 1/2$ and $\beta = 1/4$, then \cite{Krenk06} shows that the energy 
\begin{align}
	\label{eq:newmark-energy}
	\mathcal{E}_n := \frac{1}{2} (w_n^{[1]}, w_n^{[1]}) + \frac{1}{2\rho \tau} a_s(\grad w_n, \grad w_n), \qquad n \geq 0,
\end{align}
remains constant for all $n \geq 0$.

\subsubsection{$L^2$-projection}

In order to carry out the timestepping of the scheme, it is necessary to solve \cref{eq:newmark-initial-dtt} to determine the initial acceleration. A standard inverse estimate $\|u\|_1 \leq C h^{-1} p^2 \|u\|$ for $u \in \{ u \in L^2(\Omega) : u|_{K} \in \mathcal{P}_p(K) \ \forall \mathcal{T} \}$ (see e.g. \cite{Schwab98}) gives
\begin{align}
	\label{eq:l2-projection-ms-degen}
	C h^4 p^{-8} \|v\|_2^2 \leq \|v\|^2 \leq \|v\|_2^2  \qquad \forall v \in \mathbb{W}_{\Gamma},
\end{align}
where $C$ is a positive constant independent of $h$ and $p$, which means that the bilinear form appearing in \cref{eq:newmark-initial-dtt} satisfies assumptions \cref{eq:b-bilinear-bounded,eq:b-bilinear-infsup-discrete-gen}, with a stability constant $\alpha_X := C h^4 p^{-8}$ that degenerates as the discretization is refined. To help compensate for this degeneration, we choose the inner-product in \cref{eq:h2-mixed-problem-fem-any-inner} to be the $L^2$ inner-product; i.e. $\dinner{\cdot}{\cdot} = (\cdot,\cdot)$.

We compute the solution to \cref{eq:newmark-initial-dtt} by solving the mixed problem \cref{eq:h2-mixed-problem-fem-any-inner} with $\tilde{\mathbb{W}}$ and $\mathbbb{G}$ chosen as in \cref{eq:w-c1-spline-tilde-w-g-choice} with $a(\cdot,\cdot) \equiv 0$, $c(\cdot,\cdot) = (\cdot,\cdot)$, $F_1(\cdot) = -a_s(\grad w_0, \cdot)/(\rho \tau)$, $F_2(\cdot) \equiv 0$, and $[\cdot,\cdot] = (\cdot,\cdot)$ using the iterated penalty method \cref{eq:iter-penalty-morgan-scott-gen} with penalty parameter $\lambda = 2\cdot 10^4$ and terminating when $\|\bdd{\Xi}_X(\tilde{w}^n, \bdd{\gamma}^n)\| < 10^{-8}$. The iterated penalty method required 2 iterations to converge for $3 \leq p \leq 7$ and 3 for $8 \leq p \leq 12$, but failed to converge for $p \geq 13$. Consequently, the degeneracy predicted by \cref{eq:l2-projection-ms-degen} does indeed manifest in practice for sufficiently large $p$.  However, the flexibility of the inner-product $\dinner{\cdot}{\cdot}$ does aid to some extent -- if we instead choose $\dinner{\cdot}{\cdot}$ to be the $\hcurl$ inner-product $(\cdot,\cdot)_{\curl}$, then the iterated penalty method fails to converge for $p \geq 5$. The $p=10$ solution, denoted $\tilde{w}_0^{[2]}$ and displayed in \cref{fig:kirchhoff-init-wtt-combo}, is concentrated at the point load and is indeed $C^1$-conforming.

\begin{figure}[htb]
	\centering
	\begin{subfigure}[b]{0.48\linewidth}
		\centering
		\includegraphics[width=\linewidth]{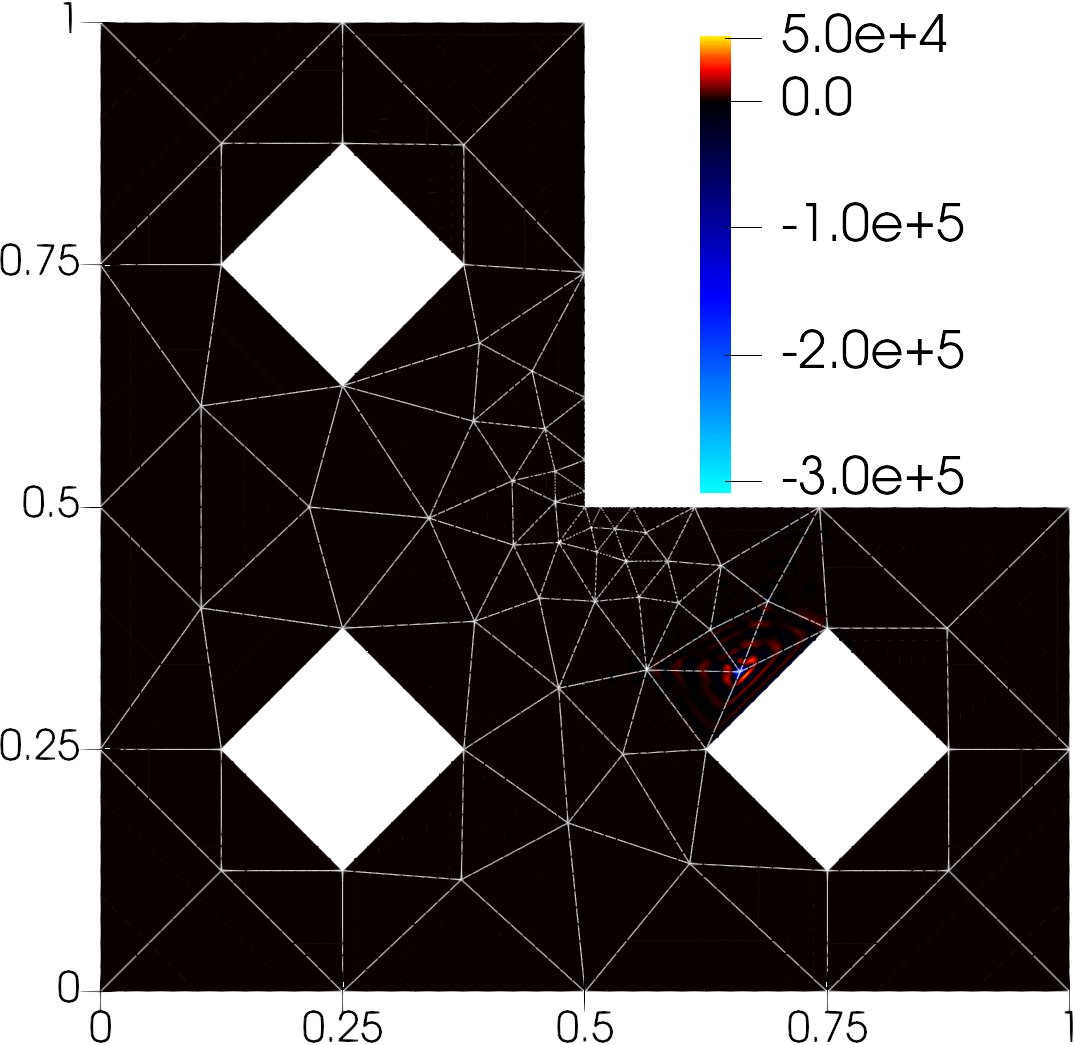}
		\caption{$\tilde{w}_0^{[2]}$}
		\label{fig:kirchhoff-init-wtt}
	\end{subfigure}
	\hfill
	\begin{subfigure}[b]{0.48\linewidth}
		\centering
		\includegraphics[width=\linewidth]{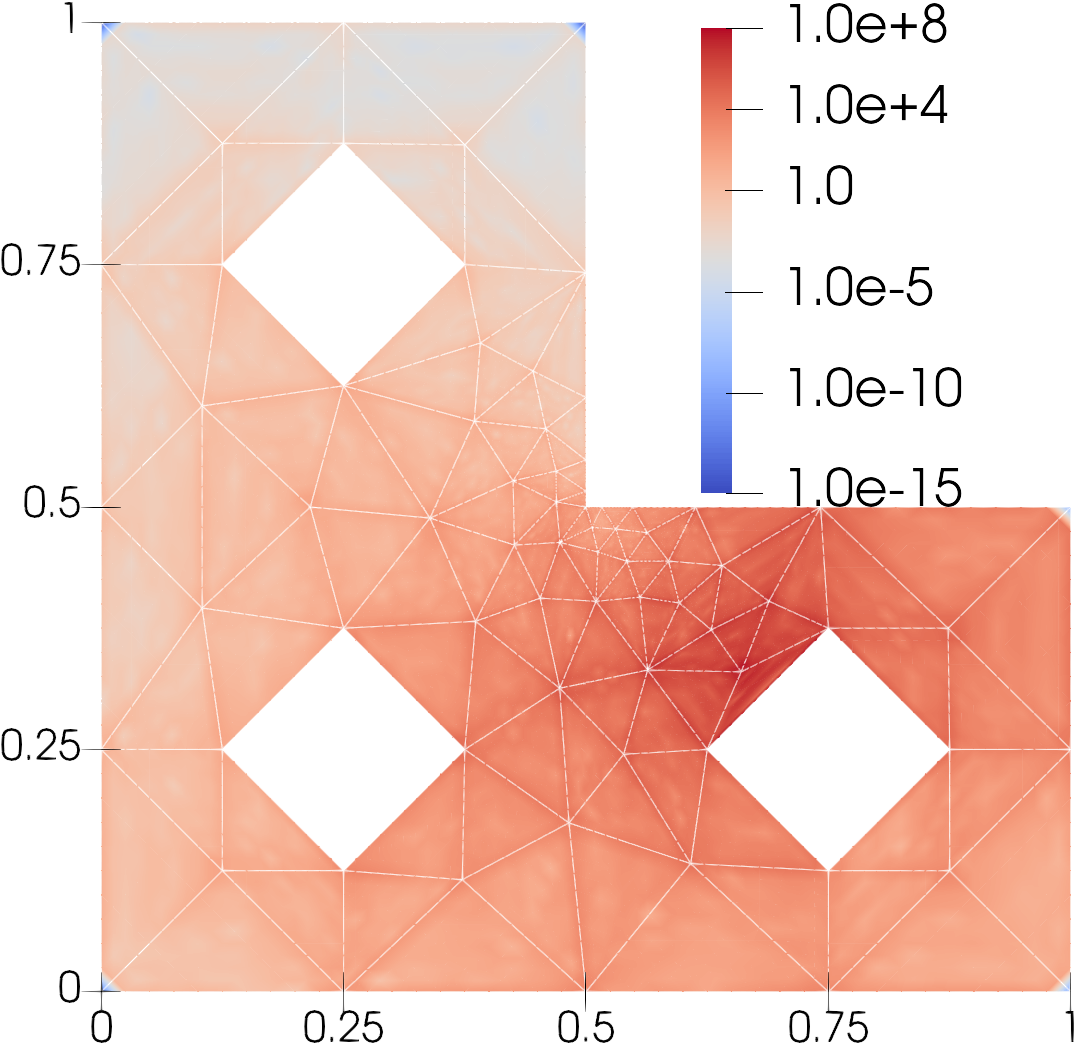}
		\caption{$|\grad \tilde{w}_0^{[2]} |$}
		\label{fig:kirchhoff-init-wtt-grad}
	\end{subfigure}
	\caption{(a) The initial acceleration $\tilde{w}_0^{[2]}$ \cref{eq:newmark-initial-dtt} and (b) $|\grad \tilde{w}_0^{[2]} |$ for the $p=10$ solution to the dynamic Kirchhoff plate problem.}
	\label{fig:kirchhoff-init-wtt-combo}
\end{figure}

\subsubsection{Time-stepping}

The approximation is advanced in time by solving the linear system \cref{eq:newmark-dtt-update} with $\Delta t = 2\cdot 10^{-4}$. The bilinear form appearing in problem \cref{eq:newmark-dtt-update} is of the form \cref{eq:time-stepping-problem} with
\begin{align}
	\label{eq:time-stepping-ms}
	C \left(h^4 p^{-8} + \frac{\beta (\Delta t)^2}{\rho \tau}\right) \|v\|_{2}^2 \leq   \|v\|^2 +  \frac{\beta (\Delta t)^2}{\rho \tau} a_s(\grad v, \grad v) 
	\qquad \forall v \in H_{\Gamma}^2(\Omega),
\end{align}
where $C$ is a positive constant independent of the discretization parameters $h$, $p$, and $\Delta t$. We compute the solution to \cref{eq:newmark-dtt-update} by solving the mixed problem \cref{eq:h2-mixed-problem-fem-any-inner} with $\tilde{\mathbb{W}}$ and $\mathbbb{G}$ chosen as in \cref{eq:w-c1-spline-tilde-w-g-choice} with $a(\cdot,\cdot) = (\beta (\Delta t)^2 /(\rho \tau)) a_s(\cdot,\cdot)$, $c(\cdot,\cdot) = (\cdot,\cdot)$, $F_1(\cdot) = -a_s(\hat{\bdd{\gamma}}_{n-1}, \cdot)/(\rho \tau)$, $F_2(\cdot) \equiv 0$, and $[\cdot,\cdot] = (\cdot,\cdot) + (\Delta t)^2 (\curl \cdot, \curl \cdot)$ using the iterated penalty method \cref{eq:iter-penalty-morgan-scott-gen} with penalty parameter $\lambda = 10^4$ and terminating when
\begin{align*}
	\left( \|\bdd{\Xi}_X(\tilde{w}^n, \bdd{\gamma}^n)\|^2 + (\Delta t)^2 \| \curl \bdd{\Xi}_X(\tilde{w}^n, \bdd{\gamma}^n)\|^2 \right)^{1/2} < 10^{-8}.
\end{align*}
Here,
\begin{align*}
	\hat{\bdd{\gamma}}_{n-1} := \bdd{\gamma}_{n-1} + (\Delta t) \bdd{\gamma}_{n-1}^{[1]} + \left( \frac{1}{2} - \beta \right) (\Delta t)^2 \bdd{\gamma}_{n-1}^{[2]},
\end{align*}
where $\bdd{\gamma}_{n}^{[2]}$ is the approximation to $\grad w_{n}^{[2]}$ given by the iterated penalty method and $\bdd{\gamma}_n$ and $\bdd{\gamma}_{n}^{[1]}$ are updated by the same expressions as in \cref{eq:newmark-pos-velocity-update} for $n \geq 1$. The initial value $\bdd{\gamma}_0$ is the approximation to $\grad w_{0}$ given by the iterated penalty method from the computation in \cref{sec:kirchhoff-static} and $\bdd{\gamma}_0^{[1]} := \bdd{0}$. 

\begin{figure}[htb]
	\centering
	\begin{subfigure}[b]{0.48\linewidth}
		\centering
		\includegraphics[width=\linewidth]{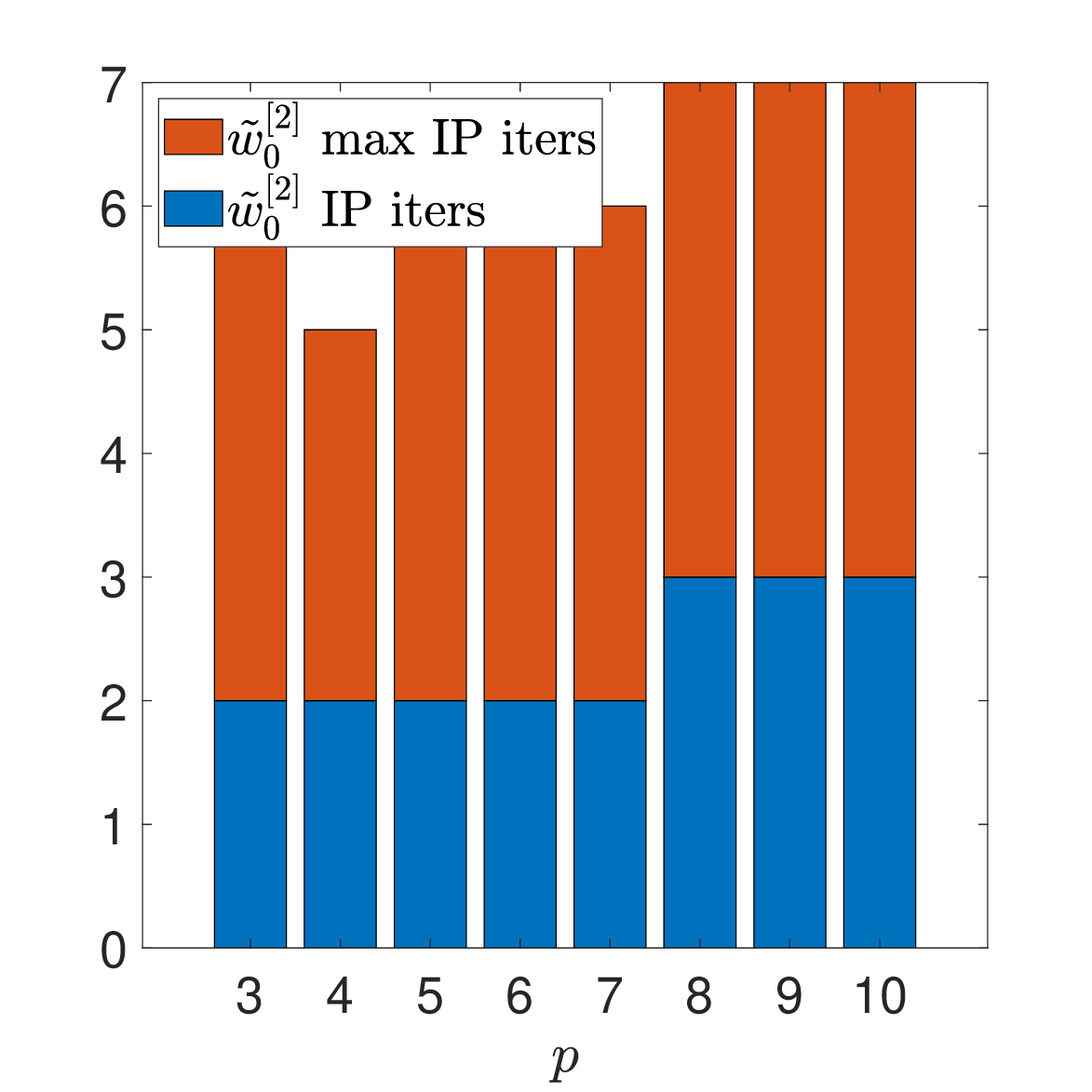}
		\caption{}
		\label{fig:kirchhoff-iters}
	\end{subfigure}
	\hfill 
	\begin{subfigure}[b]{0.48\linewidth}
		\centering
		\includegraphics[width=\linewidth]{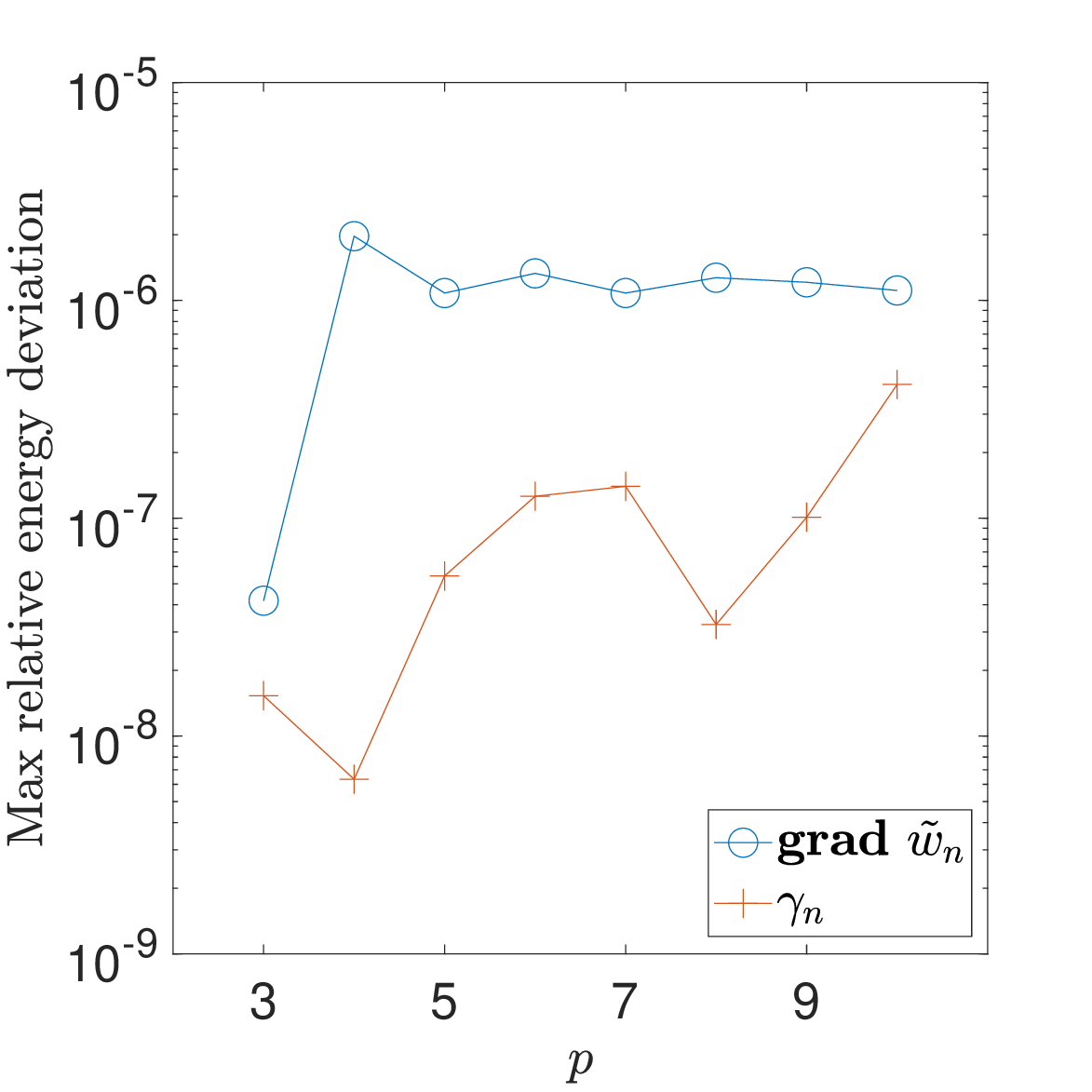}
		\caption{}
		\label{fig:kirchhoff-energy}
	\end{subfigure}
	\caption{(a) The number of iterated penalty (IP) iterations to compute the initial acceleration $w_0^{[2]}$ given by \cref{eq:newmark-initial-dtt} and the max number of iterations to advance the acceleration $w_n^{[2]}$ via \cref{eq:newmark-dtt-update}; and (b) the maximum relative energy deviation \cref{eq:newmark-energy-deviation} computed using $\grad \tilde{w}_n$ and $\bdd{\gamma}_n$. The metrics in (a) and (b) are computed over 251 time steps.}
\end{figure}

The largest number of iterated penalty iterations over 251 time steps, displayed in \cref{fig:kirchhoff-iters}, show that at most 4 iterations are required for $3 \leq p \leq 10$, but the iterated penalty method failed to converge for $p \geq 11$. Thus, the degeneracy in \cref{eq:time-stepping-ms} also arises in practice for sufficiently large $p$. If we instead choose $\dinner{\cdot}{\cdot} = (\cdot,\cdot)_{\curl}$, then the iterated penalty method fails to converge for $p \geq 5$, which again shows that the flexibility in the choice of $\dinner{\cdot}{\cdot}$ is helpful is pushing the boundaries of the method. A snapshot of the displacement $\tilde{w}_n$ and velocity $\tilde{w}_n^{[1]}$ at time $t=0.0502$ is displayed in \cref{fig:kirchhoff-w-wt-times} and the plots of the gradients confirm that the solutions are indeed $C^1$.

If the linear system \cref{eq:newmark-dtt-update} is solved exactly for all $n \geq 1$, then the energy \cref{eq:newmark-energy} is exactly conserved. To measure the quantify the effect of computing the solution to \cref{eq:newmark-dtt-update} using the iterated penalty method with the above tolerances, we display the maximum relative energy deviation
\begin{align}
	\label{eq:newmark-energy-deviation}
	\max_{1 \leq n \leq 251} \frac{ |\mathcal{E}_n - \mathcal{E}_0| }{\mathcal{E}_0}
\end{align}
in \cref{fig:kirchhoff-energy} for $3 \leq p \leq 10$. We compute $\mathcal{E}_n$ in two ways: (i) using $\grad \tilde{w}_n$ in \cref{eq:newmark-energy} and (ii) using $\bdd{\gamma}_n$ in place of $\grad \tilde{w}_n$. The maximum relative energy deviation is on the order of $2 \cdot 10^{-6} $ when $\grad \tilde{w}_n$ is used, while the relative energy deviation is on the order of $10^{-8}$ for $3 \leq p \leq 5$, but rises to $4\cdot 10^{-7}$ as $p$ increases to $11$ due to the accumulation of roundoff errors when $\bdd{\gamma}_n$ is used.

\begin{figure}[htb!]
	\centering
	\begin{subfigure}[b]{0.48\linewidth}
		\centering
		\includegraphics[width=\linewidth]{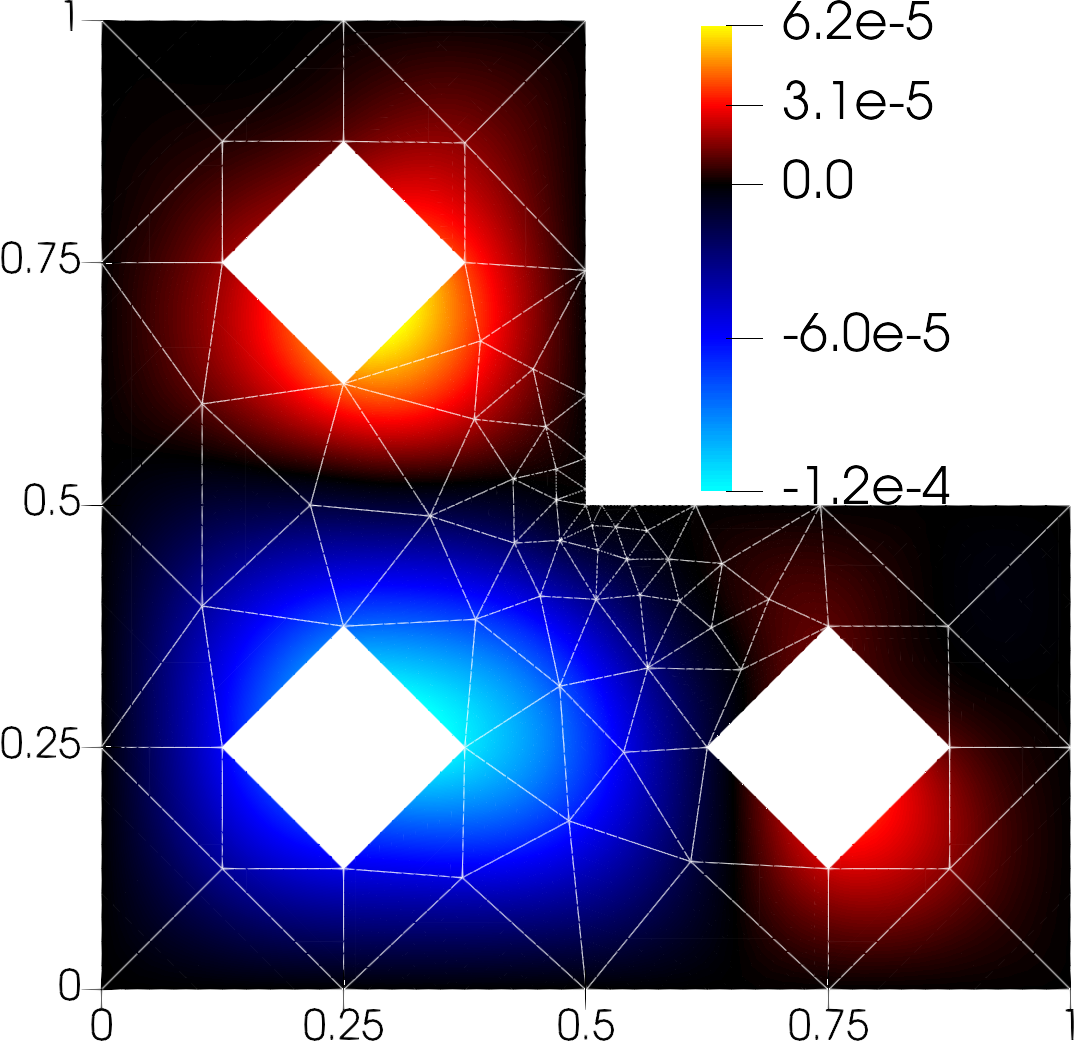}
		\caption{$\tilde{w}_n$}
		\label{fig:kirchhoff-t100}
	\end{subfigure}
	\hfill
	\begin{subfigure}[b]{0.48\linewidth}
		\centering
		\includegraphics[width=\linewidth]{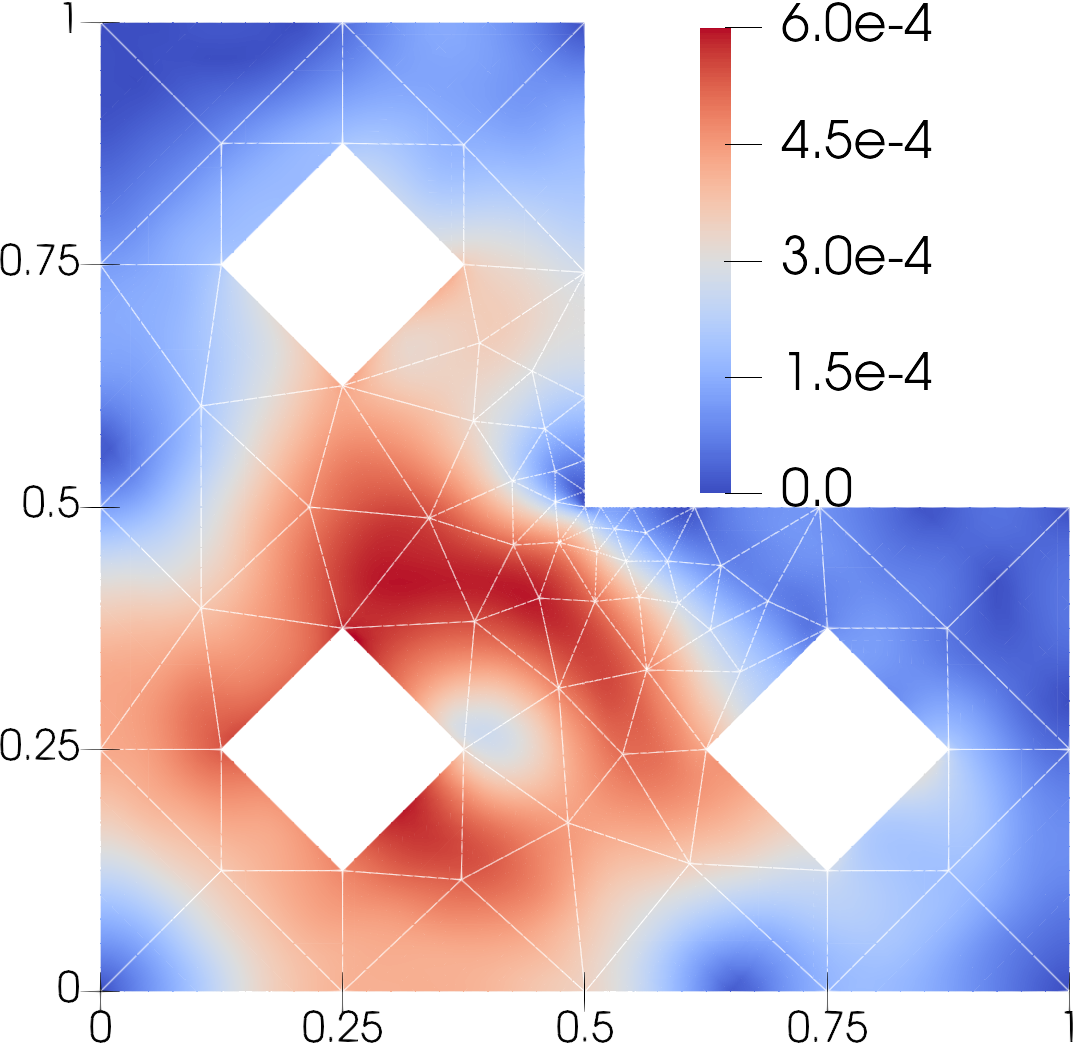}
		\caption{$|\grad \tilde{w}_n|$}
		\label{fig:kirchhoff-grad-t100}
	\end{subfigure} 
	\\
	\begin{subfigure}[b]{0.48\linewidth}
		\centering
		\includegraphics[width=\linewidth]{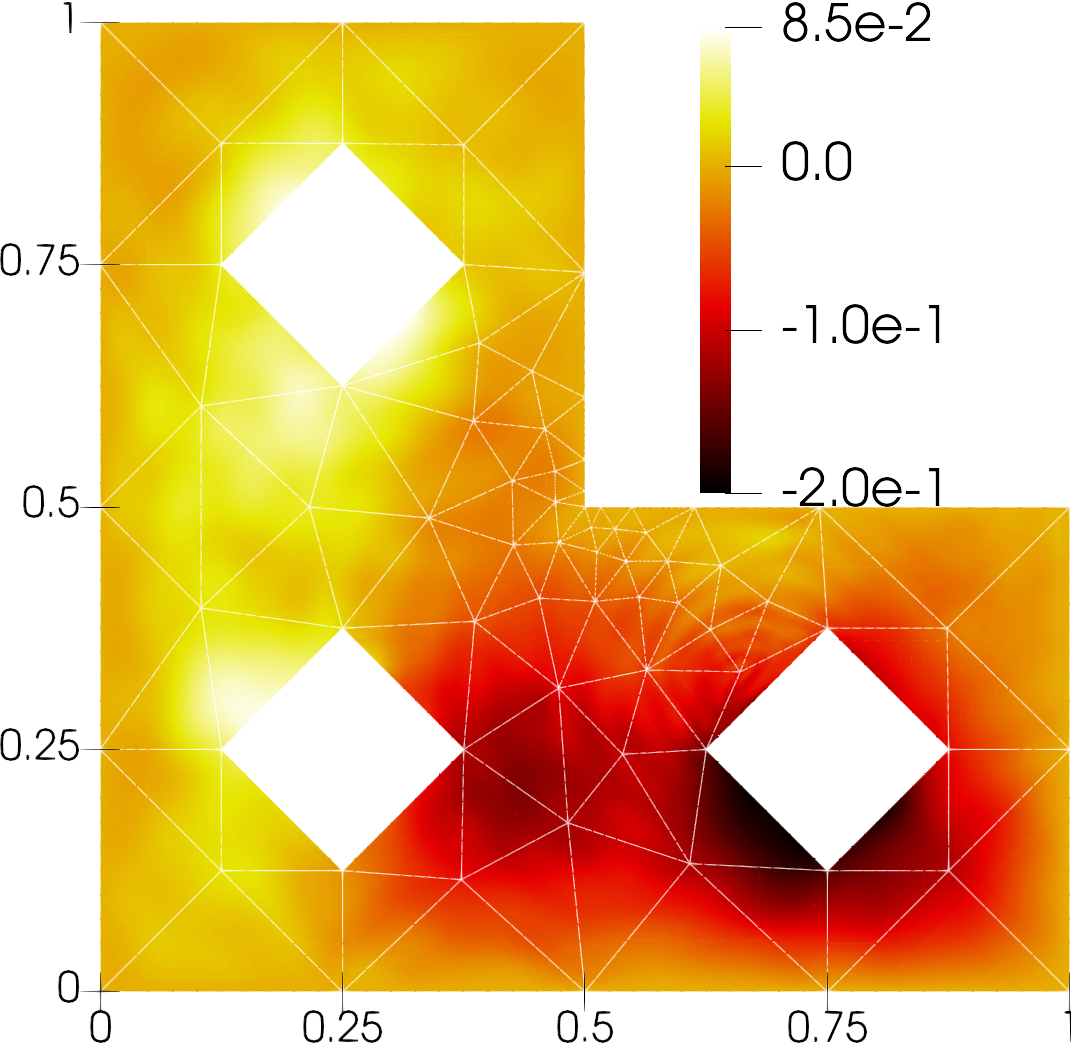}
		\caption{$\tilde{w}_n^{[1]}$}
		\label{fig:kirchhoff-wt-t100}
	\end{subfigure}
	\hfill
	\begin{subfigure}[b]{0.48\linewidth}
		\centering
		\includegraphics[width=\linewidth]{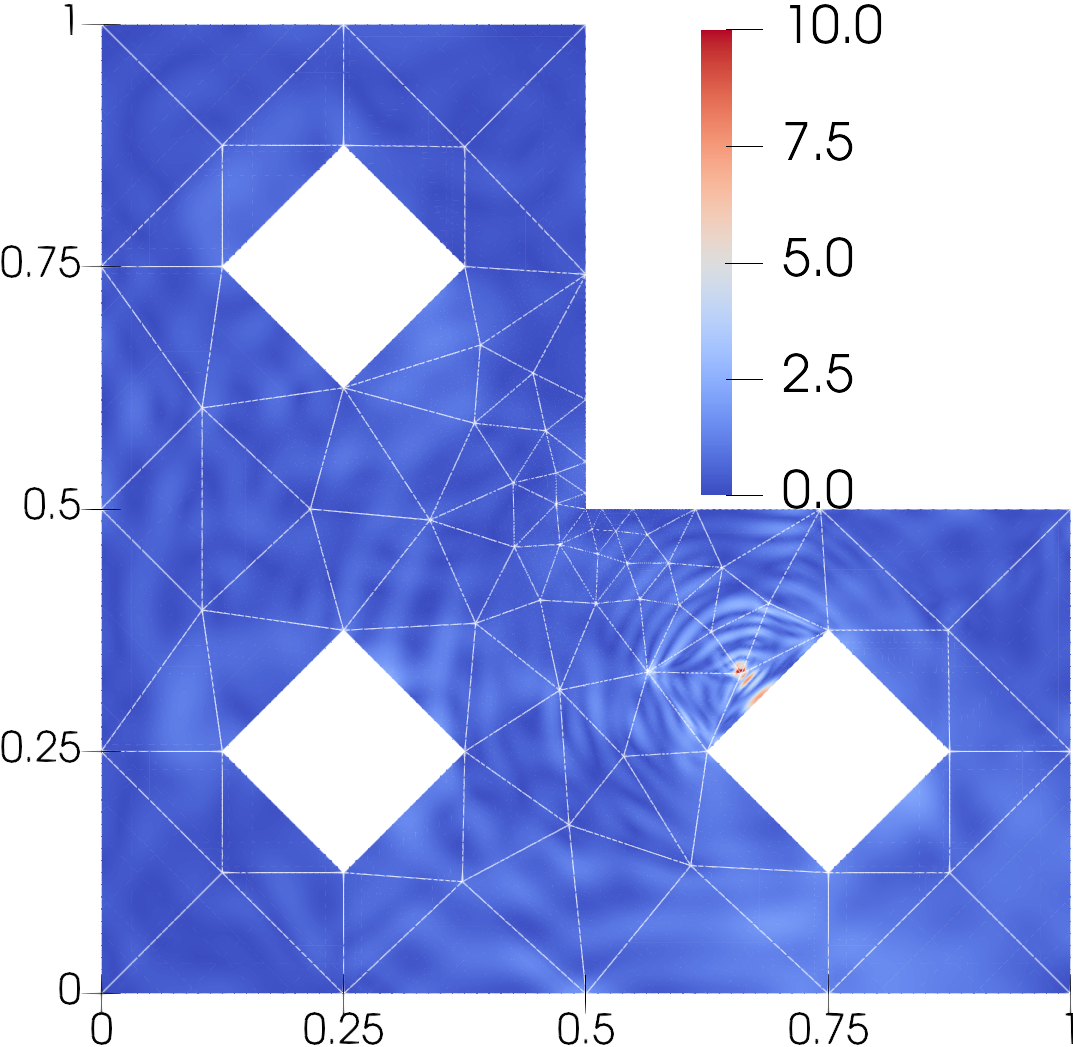}
		\caption{$| \grad \tilde{w}_n^{[1]} |$}
		\label{fig:kirchhoff-wt-grad-t100}
	\end{subfigure}
	\caption{The $p=10$ displacement $\tilde{w}_n$ and its gradient (a-b) and velocity $\tilde{w}_n^{[1]}$ and its gradient (c-d) for the dynamic Kirchhoff plate problem at time $t=0.0502$.}
	\label{fig:kirchhoff-w-wt-times}
\end{figure}

\section{Uniform stability of the Morgan-Scott space}
\label{sec:morgan-scott}

Throughout this section, let $d=2$ and $\mathbb{W}$ be Morgan-Scott space \cref{eq:w-c1-spline}, and let $\tilde{\mathbb{W}}$ and $\mathbbb{G}$ be chosen as in \cref{eq:w-c1-spline-tilde-w-g-choice}:
\begin{align*}
	\tilde{\mathbb{W}} &= \{ v \in C(\Omega) : v|_{K} \in \mathcal{P}_{p}(K) \ \forall K \in \mathcal{T} \}, \\
	\mathbbb{G} &= \{ \bdd{\theta} \in \bdd{C}(\Omega) : \bdd{\theta}|_{K} \in [\mathcal{P}_{p-1}(K)]^d \ \forall K \in \mathcal{T} \}.
\end{align*}
We shall show that when the polynomial degree is sufficiently large ($p \geq 5$) the space $\mathbbb{Q}_{\Gamma} = \image \bdd{\Xi}_X$ can be explicitly characterized and the inf-sup constant $\beta_X$ \cref{eq:inf-sup-xi} is independent of both the mesh size $h := \max_{K \in \mathcal{T}} \mathrm{diam}(K)$ and the polynomial degree $p$:
\begin{theorem}
	\label{thm:morgan-scott-inf-sup}
	For $p \geq 5$, there holds
	\begin{align}
		\label{eq:morgan-scott-inf-sup}
		\beta_X \geq \beta_0 \xi_{\mathcal{T}}
	\end{align}
	where $\beta_X$ is the inf-sup constant defined in \cref{eq:inf-sup-xi} with $\dinner{\cdot}{\cdot} = (\cdot,\cdot)_{\curl}$, $\beta_0 > 0$ is independent of $h$ and $p$, and $\xi_{\mathcal{T}} > 0$ is a quantity depending only on the topology of the mesh.
\end{theorem}
\Cref{thm:morgan-scott-inf-sup} is an immediate consequence of \cref{thm:invert-xi} below. The quantity $\xi_{\mathcal{T}}$ in \cref{eq:morgan-scott-inf-sup} depends on the topology of the mesh but not the mesh size or polynomial degree. While $\xi_{\mathcal{T}}$ may degenerate on certain mesh configurations, we shall assume that $\xi_{\mathcal{T}}$ is bounded away from zero for present purposes. Further details and an explicit formula for $\xi_{\mathcal{T}}$ can be found in \cite[eq. (A.3)]{AinCP23KirchI}.

The degree $p-1$ Brezzi-Douglas-Marini space \cite{Brezzi85} is defined by
\begin{align*}
	\bdd{BDM}_{\Gamma}^{p-1} :=  \left\{ \bdd{\gamma} \in \hcurlgamma : \bdd{\gamma}|_{K} \in \mathcal{P}_{p-1}(K)^2 \ \forall K \in \mathcal{T} \right\}.
\end{align*} 
Then, $\image \bdd{\Xi}_X$ is contained in the proper subspace of $\bdd{BDM}_{\Gamma}^{p-1}$ consisting of functions whose curl belongs to $\curl \mathbbb{G}_{\Gamma}$ since
\begin{align*}
	\curl \image \bdd{\Xi}_X = \curl \mathbbb{G}_{\Gamma} \implies  \image \bdd{\Xi}_X \subseteq \{ \bdd{\gamma} \in \bdd{BDM}_{\Gamma}^{p-1} : \curl \bdd{\gamma} \in \curl \mathbbb{G}_{\Gamma} \}.
\end{align*}
The above inclusion is in fact an equality. In order to see this, we construct a right inverse for the the $\bdd{\Xi}_X$ operator. To this end, it will be useful to start by considering the curl operator.

\subsection{Right inverse for the curl operator}

Let $\{ \Gamma_f^{(i)} \}_{i=1}^{N_f}$ denote the $N_f \geq 0$ connected components of $\Gamma_f$. Then, the following is a generalization of \cite[Lemma A.3]{AinCP23KirchI}. 
\begin{theorem}
	\label{thm:invert-rot-free-averages}
	Suppose that $\Gamma_f$ has $N_f$ connected components $\{ \Gamma_f^{(i)} \}_{i=1}^{N_f}$. Let
	\begin{align}
		\label{eq:l2gamma-def}
		r \in L^2_{\Gamma}(\Omega) := \{ r \in L^2(\Omega) : (r, 1) = 0 \text{ if } |\Gamma_f| = 0 \}.
	\end{align}
	and $\vec{\omega} \in \mathbb{R}^{N_f}$ satisfy
	\begin{align}
		\label{eq:kappaij-constraint}
		\sum_{i=1}^{N_f} \omega_{i} = \int_{\Omega} r \ d\bdd{x}.
	\end{align}
	Then, there exists $\bdd{\psi} \in \bdd{G}_{\Gamma}(\Omega)$ satisfying
	\begin{align}
		\label{eq:invert-rot-free-averages}
		\curl \bdd{\psi} = r, \quad (\unitvec{t} \cdot \bdd{\psi}, 1)_{\Gamma_f^{(i)}} = \omega_{i}, \quad 1 \leq i \leq N_f,
	\end{align}
	and
	\begin{align*}
		\|\bdd{\psi}\|_{1} \leq C \left( \|r\| + |\vec{\omega}| \right),	
	\end{align*}
	where $|\vec{\omega}| = \sum_{i=1}^{N_f} |\omega_i|$ and $C > 0$ is independent of $r$ and $\vec{\omega}$.
\end{theorem}
\begin{proof}
	First note that Lemma A.1 of \cite{AinCP21LE}, while stated for a simply-connected domain $\Omega$, extends to the present case (with the exact same proof). Then, the exact same arguments used for \cite[Lemma A.3]{AinCP23KirchI} completes the proof.
\end{proof}

The following corollary of \cref{thm:invert-rot-free-averages} shows that the operator  $\curl : \mathbbb{G}_{\Gamma} \to \curl \mathbbb{G}_{\Gamma}$ admits a continuous right inverse.
\begin{corollary}
	\label{cor:invert-rot-free-averages-discrete}
	Let $p \geq 5$. Let $r \in \curl \mathbbb{G}_{\Gamma}$ and $\vec{\omega} \in \mathbb{R}^{N_f}$ satisfy \cref{eq:kappaij-constraint}. Then, 	
	there exists $\bdd{\psi} \in \mathbbb{G}_{\Gamma}$ satisfying \cref{eq:invert-rot-free-averages} and
	\begin{align*}
		\|\bdd{\psi}\|_{1} \leq C \left( \xi_{\mathcal{T}}^{-1} \|r\| + |\vec{\omega}| \right),	
	\end{align*}
	where $C > 0$ is independent of $r$, $\vec{\omega}$, $h$, $p$, and $\xi_{\mathcal{T}}$. 
\end{corollary}
\begin{proof}
	First note that Corollary 5.1 of \cite{AinCP21LE}, stated for simply-connected $\Omega$, also holds for a connected Lipschitz polygon with the exact same proof. The result now follows from the exact same arguments as in the proof of \cite[Lemma A.4]{AinCP23KirchI} using \cref{thm:invert-rot-free-averages} in place of \cite[Lemma A.3]{AinCP23KirchI}.
\end{proof}

\subsection{\texorpdfstring{Right inverse for the $\bdd{\Xi}_X$ operator}{Right inverse for the $\Xi_X$ operator}}

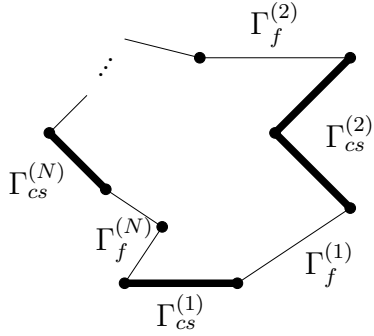
\begin{figure}[htb]
	\centering
	\begin{tikzpicture}	[scale=0.5]	
		\coordinate (A1) at (-8, -7);
		\coordinate (A2) at (-5, -7);
		\coordinate (A3) at (-2, -5);
		\coordinate (A4) at (-4, -3);
		\coordinate (A5) at (-2, -1);
		\coordinate (A6) at (-6, -1);
		
		\coordinate (P1) at (-8, -0.5);

		\coordinate (B1) at (-7, -5.5);
		\coordinate (B2) at (-8.5, -4.5);
		\coordinate (B3) at (-10, -3);
		
		\coordinate (P2) at (-9, -2);

		\draw[line width=1mm] (A1) -- (A2);
		\draw[line width=1mm] (A3) -- (A4) -- (A5);
		\draw[line width=1mm] (B2) -- (B3);
		\draw[line width=0.1mm] (A2) -- (A3);
		\draw[line width=0.1mm] (A5) -- (A6);
		\draw[line width=0.1mm] (A1) -- (B1) -- (B2);
		\draw[line width=0.1mm] (A6) -- (P1);
		\draw[line width=0.1mm] (B3) -- (P2);
		\draw[draw=none] (P1) -- (P2) node[pos=0.5, sloped]{...};

		\filldraw (A1) circle (4pt);
		\filldraw (A2) circle (4pt);	
		\filldraw (A3) circle (4pt);
		\filldraw (A4) circle (4pt);
		\filldraw (A5) circle (4pt);
		\filldraw (A6) circle (4pt);
		
		\filldraw (B1) circle (4pt);
		\filldraw (B2) circle (4pt);
		\filldraw (B3) circle (4pt);
		
		\draw ($($(A1)!0.5!(A2)$) + (0, 0) $) node[align=center,below]{$\Gamma_{cs}^{(1)}$};
		
		\draw ($($(A2)!0.5!(A3)$) + (0, -0.5) $) node[align=center,right]{$\Gamma_{f}^{(1)}$};
		
		\draw ($(A4) +(2, 0)$) node[align=center]{$\Gamma_{cs}^{(2)}$};
		
		\draw ($($(A5)!0.5!(A6)$) + (0, 0) $) node[align=center,above]{$\Gamma_{f}^{(2)}$};
		
		\draw ($(B1) +(-1, -0.35)$) node[align=center]{$\Gamma_{f}^{(N)}$};
		
		\draw ($($(B2)!0.5!(B3)$) + (0, -0.5) $) node[align=center,left]{$\Gamma_{cs}^{(N)}$};
	\end{tikzpicture}
	\caption{Ordering of $\Gamma_{cs}^{(i)}$ and $\Gamma_f^{(i)}$, $i = 1,\ldots,N$ for simply-connected $\Omega \subset \mathbb{R}^2$.}
	\label{fig:domain-example-bcs}
\end{figure}

We now turn to the construction of a right inverse of the $\bdd{\Xi}_X$ operator. The first result is a type of inverse trace theorem.
\begin{lemma}
	\label{lem:smooth-func-interp-piecewise-constants}
	Let $\{ \Gamma_{cs}^{(i)} \}_{i=1}^{N_{cs}}$ denote the $N_{cs}$ connected components of $\Gamma_{cs}$. For every $\vec{\kappa} \in \mathbb{R}^{N_{cs}}$, there exists $w \in H^2(\Omega)$ satisfying
	\begin{align}
		\label{eq:smooth-func-interp-piecewise-constants-cont}
		w|_{\Gamma_{cs}^{(i)}} = \kappa_{i}, \quad 1 \leq j \leq N_{cs}, \quad \partial_n w|_{\Gamma} = 0, \quad \text{and} \quad \|w\|_{2} \leq C |\vec{\kappa}|,
	\end{align}
	where $C$ is independent of $\vec{\kappa}$.
\end{lemma}	
\begin{proof}
	Consider first the case where $\Omega$ is simply-connected so that $N_f = N_{cs}$ and label the components $\{\Gamma_f^{(i)}\}$ and $\{\Gamma_{cs}^{(i)}\}$ in a counterclockwise orientation so that $\Gamma_f^{(i)}$ is located between $\Gamma_{cs}^{(i)}$ and $\Gamma_{cs}^{(i+1)}$, $i = 1,2,\ldots, N_f$, where $\Gamma_{cs}^{(N_f+1)} := \Gamma_{cs}^{(1)}$ ; see \cref{fig:domain-example-bcs}. Define $f \in L^2(\Gamma)$ as follows: $f$ is piecewise constant on $\Gamma$, taking the value $\kappa_i$ on $\Gamma_{cs}^{(i)}$ and $\kappa_{i+1}$ on $\Gamma_f^{(i)}$, where $\kappa_{N_{cs}+1} := \kappa_1$, except on the $N_f$ edges laying on $\Gamma_f^{(i)}$ that have a vertex at $\bar{\Gamma}_f^{(i)} \cap \bar{\Gamma}_{cs}^{(i)}$, where $f$ is a smooth transition from $\kappa_i$ to $\kappa_{i+1}$ with vanishing tangential derivatives at the endpoints of the edges.
	
	Note that by construction $f$ is continuous, $f|_{\gamma} \in C^{\infty}(\gamma)$ for all edges $\gamma$ of $\Gamma$, and the tangential derivative of $f$ vanishes at all vertices of $\Gamma$. Thanks to \cite[Theorem 6.1]{Arn88}, there exists $w \in H^2(\Omega)$ with
	\begin{align*}
		w|_{\Gamma} = f, \quad \partial_n w|_{\Gamma} = 0, \quad \text{and} \quad \|w\|_{2} \leq C \left( \|f\|_{\Gamma} + \|f' \unitvec{t} \|_{\frac{1}{2}, \Gamma} \right) \leq C |\vec{\kappa}|.
	\end{align*}

	We now return to the more general case where $\Omega$ satisfies the conditions in \cref{sec:gen-problem-setting}. Applying the above argument on each $\Omega_i$, we obtain $w_i \in H^2(\Omega_i)$, $0 \leq i \leq N_H$, satisfying 
	\begin{align*}
		w_i|_{\Gamma_{cs}^{(j)}} = \kappa_j \quad \text{if $|\Gamma_{cs}^{(j)} \cap \partial \Omega_i| > 0$}, \quad \partial_n w|_{\partial \Omega_i} = 0, \quad \text{and} \quad \|w_i\|_{2} \leq C  |\vec{\kappa}|,
	\end{align*}
	for all $1 \leq j \leq N_{cs}$. Now, for $1 \leq i \leq N_H$, \cite{Stein70} shows there exists an extension $\tilde{w}_i \in H^2(\mathbb{R}^2)$ satisfying $\tilde{w}_i = w_i$ on $\Omega_i$ and $\|\tilde{w}_i\|_{2, \mathbb{R}^2} \leq C \| w_i \|_{2, \Omega_i}$. Moreover, for $1 \leq i \leq N_{H}$, there exist $\zeta_{i} \in C^{\infty}(\mathbb{R}^2)$ satisfying $\zeta_i \equiv 1$ on $\Omega_i$ and $\zeta_i \equiv 0$ on $\bar{\Omega}_j$ for $1 \leq j \leq N_H$, $i \neq j$. Finally, there exists $\zeta_0 \in C^{\infty}(\mathbb{R}^2)$ satisfying $\zeta_0 \equiv 1$ on $\partial \Omega_0$ and $\zeta_0 \equiv 0$ on $\Omega_i$ for $1 \leq i \leq N_H$. The function $w := \sum_{i=0}^{N_H} \zeta_i \tilde{w}_i$ then satisfies \cref{eq:smooth-func-interp-piecewise-constants-cont}.
\end{proof}

\noindent The next result shows the existence of a continuous right inverse of the operator $\bdd{\Xi}_X$. 
\begin{theorem}
	\label{thm:invert-xi}
	Let $p \geq 5$. For every $\bdd{\eta} \in \{ \bdd{\gamma} \in \bdd{BDM}_{\Gamma}^{p-1} : \curl \bdd{\gamma} \in \curl \mathbbb{G}_{\Gamma} \}$, there exists $\tilde{v} \in \tilde{\mathbb{W}}_{\Gamma}$ and $\bdd{\psi} \in \mathbbb{G}_{\Gamma}$ satisfying 
	\begin{align}
		\label{eq:invert-xi}
		\bdd{\Xi}_X(\tilde{v}, \bdd{\psi}) = \bdd{\eta}  \quad \text{and} \quad \|\tilde{v}\|_1 + \|\bdd{\psi}\|_{1} \leq C \xi_{\mathcal{T}}^{-1} \| \bdd{\eta}\|_{\curl},
	\end{align}	
	where $C$ is independent of $\bdd{\eta}$, $\xi_{\mathcal{T}}$, $h$, and $p$.
	Consequently, $\image \bdd{\Xi}_X = \{ \bdd{\gamma} \in \bdd{BDM}_{\Gamma}^{p-1} : \curl \bdd{\gamma} \in \curl \mathbbb{G}_{\Gamma} \}$ and \cref{eq:morgan-scott-inf-sup} holds with $\beta_0 = C^{-1} \xi_{\mathcal{T}}$.
\end{theorem}
\begin{proof}
	Let $\bdd{\eta} \in \{ \bdd{\gamma} \in \bdd{BDM}_{\Gamma}^{p-1} : \curl \bdd{\gamma} \in \curl \mathbbb{G}_{\Gamma} \}$ be given, and let $\{ \Gamma^{(j)} \}_{j=1}^{J}$ denote the connected components of $\Gamma$. Let $r = \curl \bdd{\eta} \in \curl \mathbbb{G}_{\Gamma}$ and let $\vec{\omega} \in \mathbb{R}^{N_f}$ be any vector satisfying
	\begin{align*}
		\sum_{\substack{1 \leq i \leq N_f \\ \Gamma_f^{(i)} \subseteq \Gamma^{(j)}}} \omega_i = \langle \unitvec{t} \cdot \bdd{\eta}, 1 \rangle_{\Gamma^{(j)}}, \quad 1 \leq j \leq J, \quad \text{and} \quad |\vec{\omega}| \leq C \sum_{j=0}^{J} |\langle \unitvec{t} \cdot \bdd{\eta}, 1 \rangle_{\Gamma^{(j)}}|.
	\end{align*}
	Then, the pair $(r, \vec{\omega})$ satisfies 
	\begin{align*}
		\sum_{i=1}^{N_f} \omega_i = \sum_{j=0}^{J} \langle \unitvec{t} \cdot \bdd{\eta}, 1 \rangle_{\Gamma^{(j)}} = \langle \unitvec{t} \cdot \bdd{\eta}, 1 \rangle_{\Gamma} = \curl \bdd{\eta} = r \quad \text{and} \quad \|r\| + |\vec{\omega}| \leq C \|\bdd{\eta}\|_{\curl}, 
	\end{align*}
	where we used trace theorem to conclude $|\vec{\omega}| \leq C \| \bdd{\eta}\|_{\curl}$. Applying \cref{cor:invert-rot-free-averages-discrete} then shows that there exists $\bdd{\theta} \in \mathbbb{G}_{\Gamma}$ satisfying
	\begin{align*}
		\curl (\bdd{\theta} + \bdd{\eta}) = 0, \quad \langle \unitvec{t} \cdot (\bdd{\theta} + \bdd{\eta}), 1\rangle_{\Gamma^{(j)}} = 0, \quad 1 \leq j \leq J, \quad \text{and} \quad \|\bdd{\theta}\|_{1} \leq C \xi_{\mathcal{T}}^{-1} \| \bdd{\eta}\|_{\curl}.
	\end{align*}

	Thanks to \cite[p. 37 Theorem 3.1]{GiraultRaviart86}, there exists $\tilde{w} \in H^1(\Omega)$ satisfying $\grad \tilde{w} = \bdd{\theta} + \bdd{\eta}$. First note that $\grad \tilde{w} \in \bdd{BDM}_{\Gamma}^{p-1}$, and so
	\begin{align*}
		\tilde{w} \in \{ v \in H^1(\Omega) : v|_{K} \in \mathcal{P}_{p}(K) \ \forall K \in \mathcal{T} \} = \{ v \in C(\Omega) : v|_{K} \in \mathcal{P}_{p}(K) \ \forall K \in \mathcal{T} \}.
	\end{align*}	
	Moreover,
	\begin{align*}
		\langle \partial_t \tilde{w}, u\rangle_{\Gamma} =  \langle \unitvec{t} \cdot (\bdd{\theta} + \bdd{\eta}), u \rangle_{\Gamma} = 0 \qquad \forall u \in H^1(\Omega) : u|_{\Gamma_f} = 0,
	\end{align*}
	and so $\tilde{w}|_{\Gamma_{cs}^{(i)}} \in \mathbb{R}$ for $1 \leq i \leq N_{cs}$. The function $\tilde{u} := \tilde{w} - \tilde{w}|_{\Gamma_{cs}^{(1)}}$ then satisfies
	\begin{align*}
		\| \tilde{u} \|_{1} \leq C \|\grad \tilde{u} \| = C \| \grad \tilde{w} \| \leq C \xi_{\mathcal{T}}^{-1} \|\bdd{\eta}\|_{\curl}.
	\end{align*}
	
	To correct the trace of $\tilde{u}$, we appeal to \cref{lem:smooth-func-interp-piecewise-constants} for the existence of $\phi \in H^{2}(\Omega)$ satisfying
	\begin{align*}
		\phi|_{\Gamma_{cs}} = \tilde{u}|_{\Gamma_{cs}}, \quad \partial_n \phi|_{\Gamma} = 0, \quad \text{and} \quad \|\phi\|_{2} \leq C \sum_{i=1}^{N} | \tilde{u}|_{\Gamma_{cs}^{(i)}} | \leq C \|\tilde{u}\|_{1} \leq C \xi_{\mathcal{T}}^{-1} \|\bdd{\eta}\|_{\curl}.
	\end{align*}
	Let $\tilde{\phi} \in W^5$ denote the piecewise degree 5 Girault-Scott interpolant \cite{Girault02} of $\phi$ which satisfies
	\begin{align*}
		\tilde{\phi}|_{\Gamma_{cs}} = \phi|_{\Gamma_{cs}}, \quad \partial_n \tilde{\phi}|_{\Gamma} = 0, \quad \text{and} \quad \|\tilde{\phi}\|_{2} \leq C \|\phi\|_{2} \leq C \xi_{\mathcal{T}}^{-1} \|\bdd{\eta}\|_{\curl}.
	\end{align*}
	The functions $\tilde{v} = \tilde{u} - \tilde{\phi}$ and $\bdd{\psi} = \bdd{\theta} - \grad \tilde{\phi}$ then satisfy $\tilde{v} \in \tilde{\mathbb{W}}_{\Gamma}$, $\bdd{\psi} \in \mathbbb{G}_{\Gamma}$, and \cref{eq:invert-xi}. Consequently, $\{ \bdd{\gamma} \in \bdd{BDM}_{\Gamma}^{p-1} : \curl \bdd{\gamma} \in \curl \mathbbb{G}_{\Gamma} \} \subseteq \image \bdd{\Xi}_X$, and so $\image \bdd{\Xi}_X = \{ \bdd{\gamma} \in \bdd{BDM}_{\Gamma}^{p-1} : \curl \bdd{\gamma} \in \curl \mathbbb{G}_{\Gamma} \}$. The bound $\beta_X \geq C^{-1} \xi_{\mathcal{T}}$ then follows from \cref{eq:invert-xi}.
\end{proof}

\section{Relationship to previous work}
\label{sec:prev-work}

In previous work \cite{AinCP23KirchI}, we considered the special case that $d=2$, $c(\cdot,\cdot) \equiv 0$ for which the mixed problem \cref{eq:h2-mixed-nodual} can be written in an alternative form that avoids the vector-valued space $\hcurlgamma$. In this section, we show that the current method can be regarded as a generalization of the method in \cite{AinCP23KirchI} to the cases $d \in \{2,3\}$ and where $c(\cdot,\cdot)$ need not vanish. We begin with two properties of the space $\bdd{G}_{\Gamma}(\Omega)$. The first is an immediate consequence of \cref{thm:invert-rot-free-averages}: 
\begin{description}
	\item[(P1)\label{hp:gradient-infsup-cont}] For all $r \in L^2_{\Gamma}(\Omega)$ and $\vec{\kappa} \in \mathbb{R}^{N_f - 1}$, there exists $\bdd{\psi} \in \bdd{G}_{\Gamma}(\Omega)$ satisfying $\curl \bdd{\psi} = r$, $(\unitvec{t} \cdot \bdd{\psi}, 1)_{\Gamma_f^{(i)}} = \kappa_i$, $1 \leq i \leq N_f - 1$, and $\|\bdd{\psi}\|_1 \leq C(\|r\| + |\vec{\kappa}|)$,
	where $L^2_{\Gamma}(\Omega)$ is defined as in \cref{eq:l2gamma-def}.
\end{description}

\vspace{0.25em}

\noindent The second property, which follows from \cite[p. 37, Theorem 3.1]{GiraultRaviart86}, characterizes the gradients of $H^2_{\Gamma}(\Omega)$ functions in terms of curl-free vector-fields in $\bdd{G}_{\Gamma}(\Omega)$: 
\begin{description}
	\item[(P2)\label{hp:gradient-exactness-cont}]  $\grad H^2_{\Gamma}(\Omega) = \{ \bdd{\gamma} \in \bdd{G}_{\Gamma}(\Omega) : \curl \bdd{\gamma} \equiv 0 \text{ and } (\unitvec{t} \cdot \bdd{\gamma}, 1)_{\Gamma_f^{(i)}} = 0, \ 1 \leq i \leq N_f-1  \}$.
\end{description}

\vspace{0.25em}

\noindent With these properties in hand, we have the following representation result:
\begin{lemma}
	\label{lem:grad-tildew-perp-reisz-rep-cont}
	For every pair $(r, \vec{\kappa}) \in L^2_{\Gamma}(\Omega) \times \mathbb{R}^{N_f-1}$, there exists a unique vector-field
	\begin{align}
		\label{eq:grad-tildew-perp-cont}
		\bdd{\zeta} \in	(\grad H^1_{\Gamma}(\Omega))^{\perp} := \{ \bdd{\phi} \in \image \bdd{\Xi} : (\bdd{\phi}, \grad \tilde{v})_{\curl} = 0 \quad \forall \tilde{v} \in H^1_{\Gamma}(\Omega) \},
	\end{align}
	satisfying
	\begin{align}
		\label{eq:grad-tildew-perp-reisz-rep-cont}
		\tilde{b}(\bdd{\eta}; r, \vec{\kappa}) := (r, \curl \bdd{\eta}) + \sum_{i=1}^{N_f-1} (\unitvec{t} \cdot \bdd{\eta}, \kappa_i)_{\Gamma_f^{(i)}} = - 	(\bdd{\zeta}, \bdd{\eta})_{\curl} \qquad \forall \bdd{\eta} \in \image \bdd{\Xi}.
	\end{align}
	Conversely, for every $\bdd{\zeta} \in (\grad H^1_{\Gamma}(\Omega))^{\perp}$, there exists a unique pair $(r, \vec{\kappa}) \in L^2_{\Gamma}(\Omega) \times \mathbb{R}^{N_f-1}$ satisfying \cref{eq:grad-tildew-perp-reisz-rep-cont}.
\end{lemma}
\begin{proof}
	Let $(r, \vec{\kappa}) \in L^2_{\Gamma}(\Omega) \times \mathbb{R}^{N_f-1}$ be given. Thanks to the trace theorem, there holds
	\begin{align*}
		|\tilde{b}(\bdd{\eta}; r, \vec{\kappa})| \leq C (\|r\| + |\vec{\kappa}|) \| \bdd{\eta}\|_{\curl} \qquad \forall \bdd{\eta} \in \image \bdd{\Xi} = \hcurlgamma,
	\end{align*}
	and so the mapping $\hcurlgamma \ni \bdd{\eta} \mapsto \tilde{b}(\bdd{\eta}; r, \vec{\kappa})$ is a continuous linear functional on $\hcurlgamma$. By the Riesz representation theorem, there exists $\bdd{\zeta} \in \hcurlgamma$ satisfying
	\cref{eq:grad-tildew-perp-reisz-rep-cont} and $\|\bdd{\zeta}\|_{\curl} \leq C (\|r\| + |\vec{\kappa}|)$. Since $\tilde{b}(\grad \tilde{v}; r, \vec{\kappa}) = 0$ for all $\tilde{v} \in H^1_{\Gamma}(\Omega)$,  there holds $\bdd{\zeta} \in (\grad H^1_{\Gamma}(\Omega))^{\perp}$. Moreover, the the linear mapping $(r, \vec{\kappa}) \mapsto \bdd{\zeta}$ is injective since if $\tilde{b}(\bdd{\eta}; r, \vec{\kappa}) =  0$ for all $\bdd{\eta} \in \hcurlgamma$, then \ref{hp:gradient-infsup-cont} shows that $r \equiv 0$ and $\vec{\kappa} = \vec{0}$. 
	
	Now let $\bdd{\zeta} \in (\grad H^1_{\Gamma}(\Omega))^{\perp}$ be given. Thanks to \ref{hp:gradient-infsup-cont}, there exists $\tilde{\beta} > 0$ such that 
	\begin{align*}
		\inf_{ \substack{ (r, \vec{\kappa}) \in L^2_{\Gamma}(\Omega) \times \mathbb{R}^{N_f-1} \\ (r, \vec{\kappa}) \neq (0, \vec{0}) } } \sup_{ \substack{ \bdd{\theta} \in \bdd{G}_{\Gamma}(\Omega) \\ \bdd{\theta} \neq \bdd{0} } } \frac{\tilde{b}(\bdd{\theta}; r, \vec{\kappa})}{ (\|r\| + |\vec{\kappa}|) \|\bdd{\theta}\|_1 } \geq \tilde{\beta},
	\end{align*}
	and so  there exists $r \in L^2_{\Gamma}(\Omega)$ and $\vec{\kappa} \in \mathbb{R}^{N_f - 1}$ satisfying
	\begin{align*}
		\tilde{b}(\bdd{\psi}; r, \vec{\kappa}) = -(\bdd{\zeta}, \bdd{\psi})_{\curl} \qquad \forall \bdd{\psi} \in \bdd{G}_{\Gamma}(\Omega).
	\end{align*}
	Now let $\bdd{\eta} \in \image \bdd{\Xi}$ be decomposed as in \cref{lem:hrotgamma-grad-h10-decomp}: i.e. $\bdd{\eta} = \grad \tilde{v} + \bdd{\psi}$, where $\tilde{v} \in H^1_{\Gamma}(\Omega)$ and $\bdd{\psi} \in \bdd{G}_{\Gamma}(\Omega)$. Then, there holds
	\begin{align*}
		\tilde{b}(\bdd{\eta}; r, \vec{\kappa}) = \tilde{b}(\grad \tilde{v}; r, \vec{\kappa}) + \tilde{b}(\bdd{\psi}; r, \vec{\kappa}) = -(\bdd{\zeta}, \bdd{\psi})_{\curl} = -(\bdd{\zeta}, \bdd{\eta})_{\curl},
	\end{align*}
	where we used that $\tilde{b}(\cdot; r, \vec{\kappa})$ and $(\bdd{\zeta}, \cdot)_{\curl}$ vanish on $\grad H^1_{\Gamma}(\Omega)$. As a result, \cref{eq:grad-tildew-perp-reisz-rep-cont} holds, and we have shown that the continuous mapping $(r, \vec{\kappa}) \mapsto \bdd{\zeta}$ is bijective.
\end{proof}

The next result shows that solutions to \cref{eq:h2-mixed-nodual} satisfy a system of four equations consisting of two elliptic projections and a Stokes-like equation:
\begin{lemma}
	\label{lem:new-method-satisfies-old-cont}
	Suppose that $c(\cdot,\cdot) \equiv 0$. Let $F_1 \in \bdd{G}_{\Gamma}(\Omega)^*$, $F_2 \in H^1_{\Gamma}(\Omega)^*$, $(\tilde{w}, \bdd{\gamma}, \bdd{q}) \in H^1_{\Gamma}(\Omega) \times \bdd{G}_{\Gamma}(\Omega) \times \image \bdd{\Xi}$ be the solution to \cref{eq:h2-mixed-nodual}, and  $\tilde{z} \in H^1_{\Gamma}(\Omega)$ satisfy
	\begin{align}
		\label{eq:tildezx-def-cont}
		(\grad \tilde{z}, \grad \tilde{v}) &= F_2(\tilde{v}) \qquad & &\forall \tilde{v} \in H^1_{\Gamma}(\Omega).
	\end{align}
	Then, there exists unique $(r, \vec{\kappa}) \in L^2_{\Gamma}(\Omega) \times \mathbb{R}^{N_f - 1}$ such that
	\begin{subequations}
		\label{eq:stokes-system-cont}
		\begin{alignat}{2}
			\label{eq:stokes-system-cont-1}
			a(\bdd{\gamma}, \bdd{\psi}) + \tilde{b}(\bdd{\psi}; r, \vec{\kappa}) &= (\grad \tilde{z}, \bdd{\psi}) + F_1(\bdd{\psi}) \qquad & &\forall \bdd{\psi} \in \bdd{G}_{\Gamma}(\Omega), \\
			\label{eq:stokes-system-cont-2}
			\tilde{b}(\bdd{\gamma}; s, \vec{\mu}) &= 0 \qquad & &\forall (s, \vec{\mu}) \in L^2_{\Gamma}(\Omega) \times \mathbb{R}^{N_f - 1},
		\end{alignat}
	\end{subequations}
	and $\tilde{w}$ satisfies
	\begin{align}
		\label{eq:h1-postprocess-cont}
		(\grad \tilde{w}, \grad \tilde{u}) = (\bdd{\gamma}, \grad \tilde{u}) \qquad \forall \tilde{u} \in H^1_{\Gamma}(\Omega).
	\end{align}
\end{lemma}
\begin{proof}
	Let $\bdd{\zeta} = \bdd{q} - \grad \tilde{z} \in \image \bdd{\Xi}$. Choosing $\bdd{\psi} \equiv \bdd{0}$ in \cref{eq:h2-mixed-nodual-1} gives
	\begin{align*}
		(\bdd{\zeta}, \grad \tilde{v})_{\curl} = (\bdd{q}, \grad \tilde{v})_{\curl} - (\grad \tilde{z}, \grad \tilde{v})_{\curl} = 0 \qquad \forall \tilde{v} \in H^1_{\Gamma}(\Omega).
	\end{align*}
	Consequently, $\bdd{\zeta} \in (\grad H^1_{\Gamma}(\Omega))^{\perp}$ \cref{eq:grad-tildew-perp-cont}, and so there exists unique $(r, \vec{\kappa}) \in L^2_{\Gamma}(\Omega) \times \mathbb{R}^{N_f - 1}$ satisfying \cref{eq:grad-tildew-perp-reisz-rep-cont} thanks to \cref{lem:grad-tildew-perp-reisz-rep-cont}. Choosing $\tilde{v} \equiv 0$ in \cref{eq:h2-mixed-nodual-1} then gives
	\begin{align*}
		F_1(\bdd{\psi}) = a(\bdd{\gamma}, \bdd{\psi}) - (\bdd{\psi}, \bdd{q})_{\curl} = a(\bdd{\gamma}, \bdd{\psi})  -  (\bdd{\zeta}, \bdd{\psi})_{\curl} -  (\grad \tilde{z}, \bdd{\psi}) \qquad \forall \bdd{\psi} \in \bdd{G}_{\Gamma}(\Omega),
	\end{align*}
	and \cref{eq:stokes-system-cont-1} follows.
	
	Now let $(s, \vec{\mu}) \in L^2_{\Gamma}(\Omega) \times \mathbb{R}^{N_f - 1}$. Thanks to \cref{lem:grad-tildew-perp-reisz-rep-cont}, there exists $\bdd{\eta} \in (\grad H^1_{\Gamma}(\Omega))^{\perp}$ such that $b(\bdd{\gamma}; s, \vec{\mu}) = (\bdd{\eta}, \grad \tilde{w} - \bdd{\gamma})_{\curl} = 0$, and \cref{eq:stokes-system-cont-2} follows. Finally, choosing $\bdd{\eta} = \grad \tilde{u}$ for $\tilde{u} \in H^1_{\Gamma}(\Omega)$ in \cref{eq:h2-mixed-nodual-2} gives
	\begin{align*}
		(\grad \tilde{w}, \grad \tilde{u}) = (\grad \tilde{w}, \grad \tilde{u})_{\curl} = (\bdd{\gamma}, \grad \tilde{u})_{\curl} = (\bdd{\gamma}, \grad \tilde{u}),
	\end{align*}
	and \cref{eq:h1-postprocess-cont} follows.
\end{proof}

Finally, we show that the converse of \cref{lem:new-method-satisfies-old-cont} holds:
\begin{lemma}
	\label{lem:old-method-satisfies-new-cont}
	Let $c(\cdot,\cdot) \equiv 0$. For every $F_1 \in \bdd{G}_{\Gamma}(\Omega)^*$ and $F_2 \in H^1_{\Gamma}(\Omega)^*$, the following variational problem is uniquely solvable:	
	Find $(\tilde{z}, \bdd{\gamma}, r, \vec{\kappa}, \tilde{w}) \in   H^1_{\Gamma}(\Omega) \times \bdd{G}_{\Gamma}(\Omega) \times L^2_{\Gamma}(\Omega) \times \mathbb{R}^{N_f - 1} \times  H^1_{\Gamma}(\Omega)$ such that \cref{eq:tildezx-def-cont,eq:stokes-system-cont,eq:h1-postprocess-cont} holds. Moreover, there exists unique $\bdd{\zeta} \in (\grad H^1_{\Gamma}(\Omega))^{\perp}$ such that
	\begin{align}
		\label{eq:zeta-quad-system}
		(\bdd{\zeta}, \bdd{\eta})_{\curl} = -\tilde{b}(\bdd{\eta}; r, \vec{\kappa}) \qquad \forall \bdd{\eta} \in \image \bdd{\Xi},
	\end{align}
	and the triplet $(\tilde{w}, \bdd{\gamma}, \bdd{\zeta} + \grad \tilde{z})$ is the unique solution to \cref{eq:h2-mixed-nodual}.
\end{lemma}
\begin{proof}
	Existence of solutions to \cref{eq:tildezx-def-cont,eq:stokes-system-cont,eq:h1-postprocess-cont} follow from \cref{lem:new-method-satisfies-old-cont}. We now turn to uniqueness of solutions. Suppose that $F_1 \equiv 0$ and $F_2 \equiv 0$. Then \cref{eq:tildezx-def-cont} shows that $\tilde{z} \equiv 0$. \Cref{eq:stokes-system-cont-2} and \ref{hp:gradient-exactness-cont} mean that $\bdd{\gamma} = \grad w$ for some $w \in H^2_{\Gamma}(\Omega)$. Testing \cref{eq:stokes-system-cont-1} with $\bdd{\psi} = \grad v$ for $v \in H^2_{\Gamma}(\Omega)$ gives $a(\grad w, \grad v) = 0$ for all $v \in H^2_{\Gamma}(\Omega)$, and so $w \equiv 0$ thanks to \cref{eq:b-bilinear-infsup}. Testing \cref{eq:stokes-system-cont-1} with $\bdd{\psi}$ chosen as in  \ref{hp:gradient-infsup-cont} then gives $r = 0$ and $\vec{\kappa} = 0$. Since $\bdd{\gamma} = \grad w = \bdd{0}$, \cref{eq:h1-postprocess-cont} shows that $\tilde{w} \equiv 0$.
	
	The existence of $\bdd{\zeta} \in (\grad H^1_{\Gamma}(\Omega))^{\perp}$ satisfying \cref{eq:zeta-quad-system} follows from \cref{lem:grad-tildew-perp-reisz-rep-cont} and that the triplet $(\tilde{w}, \bdd{\gamma}, \bdd{\zeta} + \grad \tilde{z})$ satisfies \cref{eq:h2-mixed-nodual} follows from similar arguments as in the proof of \cref{lem:new-method-satisfies-old-cont}.
\end{proof}

We now repeat the same arguments at the discrete level to avoid the appearance of the space $\image \bdd{\Xi}_X$. We start by stating the following discrete analogues of \ref{hp:gradient-infsup-cont} and \ref{hp:gradient-exactness-cont} respectively: \\

\begin{description}
	\item[(B1)\label{hp:gradient-infsup}] For all $\vec{\kappa} \in \mathbb{R}^{N_f - 1}$, there exists $\bdd{\psi} \in \mathbbb{G}_{\Gamma}$ satisfying
	\begin{align*}
		\curl \bdd{\psi} \equiv 0 \quad \text{and} \quad  (\unitvec{t} \cdot \bdd{\psi}, 1)_{\Gamma_f^{(i)}} = \kappa_i, \quad 1 \leq i \leq N_f - 1.
	\end{align*}
\end{description}

\noindent In particular, note that \ref{hp:gradient-infsup} means that for all $r \in \curl \mathbbb{G}_{\Gamma}$ and $\vec{\kappa} \in \mathbb{R}^{N_f - 1}$, there exists $\bdd{\psi} \in \mathbbb{G}_{\Gamma}$ satisfying
\begin{align*}
	\curl \bdd{\psi} \equiv r \quad \text{and} \quad  (\unitvec{t} \cdot \bdd{\psi}, 1)_{\Gamma_f^{(i)}} = \kappa_i, \quad 1 \leq i \leq N_f - 1.
\end{align*}
Since $\mathbbb{G}_{\Gamma} \times \mathbb{R}^{N_f - 1}$ is finite dimensional, $\bdd{\psi}$ can also be chosen so that $\|\bdd{\psi}\|_{1} \leq C_X (\|r\| + |\vec{\kappa}|)$, where $C_X > 0$ is a constant independent of $r$ and $\vec{\kappa}$ possibly depending on the dimension of $\mathbbb{G}_{\Gamma} \times \mathbb{R}^{N_f - 1}$. Thus, \ref{hp:gradient-infsup} implies that \ref{hp:gradient-exactness-cont} holds with $\bdd{G}_{\Gamma}(\Omega)$ and $L^2_{\Gamma}(\Omega)$ replaced by $\mathbbb{G}_{\Gamma}$ and $\curl \mathbbb{G}_{\Gamma}$.

The discrete analogue of \ref{hp:gradient-exactness-cont} is precisely \ref{hp:gradient-exactness-cont} with $H^2_{\Gamma}(\Omega)$ and $\bdd{G}_{\Gamma}(\Omega)$ replaced by $\mathbb{W}_{\Gamma}$ and $\mathbbb{G}_{\Gamma}$: \\
\begin{description}
	\item[(B2)\label{hp:gradient-exactness}] $\grad \mathbb{W}_{\Gamma} = \{ \bdd{\gamma} \in \mathbbb{G}_{\Gamma} : \curl \bdd{\gamma} \equiv 0 \text{ and } (\unitvec{t} \cdot \bdd{\gamma}, 1)_{\Gamma_f^{(i)}} = 0, \ 1 \leq i \leq N_f-1  \}$. \\
\end{description}

\noindent Then, exactly the same arguments that led to \cref{lem:new-method-satisfies-old-cont,lem:old-method-satisfies-new-cont} give the following result:
\begin{theorem}
	\label{thm:two-methods-equivalent}
	Suppose that $c(\cdot,\cdot) \equiv 0$. Let $\tilde{\mathbb{W}}_{\Gamma}$, $\mathbbb{G}_{\Gamma}$, and $\mathbbb{Q}_{\Gamma}$ be given spaces satisfying \ref{hp:tildewgamma-ggamma-def}-\ref{hp:b-bilinear-infsup} and \ref{hp:gradient-infsup}-\ref{hp:gradient-exactness}.  Suppose that $F_2 \in \tilde{\mathbb{W}}_{\Gamma}^*$, $(\tilde{w}_X, \bdd{\gamma}_X, \bdd{q}_X) \in \tilde{\mathbb{W}}_{\Gamma} \times \mathbbb{G}_{\Gamma} \times \mathbbb{Q}_{\Gamma}$ is the solution to \cref{eq:h2-mixed-problem-fem}, and  $\tilde{z}_X \in \tilde{\mathbb{W}}_{\Gamma}$ satisfies
	\begin{align}
		\label{eq:tildezx-def}
		(\grad \tilde{z}_X, \grad \tilde{v}) &= F_2(\tilde{v}) \qquad & &\forall \tilde{v} \in \tilde{\mathbb{W}}_{\Gamma}.
	\end{align}
	Then, there exists unique $(r_X, \vec{\kappa}) \in \curl \mathbbb{G}_{\Gamma} \times \mathbb{R}^{N_f - 1}$ such that
	\begin{subequations}
		\label{eq:stokes-system}
		\begin{alignat}{2}
			\label{eq:stokes-system-1}
			a(\bdd{\gamma}_X, \bdd{\psi}) + \tilde{b}(\bdd{\psi}; r_X, \vec{\kappa}) &= (\grad \tilde{z}_X, \bdd{\psi}) + F_1(\bdd{\psi}) \qquad & &\forall \bdd{\psi} \in \mathbbb{G}_{\Gamma}, \\
			\label{eq:stokes-system-2}
			\tilde{b}(\bdd{\gamma}_X; s, \vec{\mu}) &= 0 \qquad & &\forall (s, \vec{\mu}) \in \curl \mathbbb{G}_{\Gamma} \times \mathbb{R}^{N_f - 1},
		\end{alignat}
	\end{subequations}
	and $\tilde{w}_X$ satisfies
	\begin{align}
		\label{eq:h1-postprocess}
		(\grad \tilde{w}_X, \grad \tilde{u}) = (\bdd{\gamma}_X, \grad \tilde{u}) \qquad \forall \tilde{u} \in \tilde{\mathbb{W}}_{\Gamma}.
	\end{align}
	
	Conversely, for all $F_1 \in \mathbbb{G}_{\Gamma}^*$ and $F_2 \in \tilde{\mathbb{W}}_{\Gamma}^*$, the following variational problem is uniquely solvable:	
	Find $(\tilde{z}_X, \bdd{\gamma}_X, r_X, \vec{\kappa}, \tilde{w}_X) \in  \tilde{\mathbb{W}}_{\Gamma} \times \mathbbb{G}_{\Gamma} \times \curl \mathbbb{G}_{\Gamma} \times \mathbb{R}^{N_f - 1} \times \tilde{\mathbb{W}}_{\Gamma}$ such that \cref{eq:tildezx-def,eq:stokes-system,eq:h1-postprocess} holds. Moreover, there exists unique vector-field
	\begin{align*}
		\bdd{\zeta} \in	(\grad \tilde{\mathbb{W}}_{\Gamma})^{\perp} := \{ \bdd{\phi} \in \image \bdd{\Xi}_X : (\bdd{\phi}, \grad \tilde{v})_{\curl} = 0 \quad \forall \tilde{v} \in \tilde{\mathbb{W}}_{\Gamma} \},
	\end{align*}
	such that
	\begin{align*}
		(\bdd{\zeta}_X, \bdd{\eta})_{\curl} = -\tilde{b}(\bdd{\eta}; r_X, \vec{\kappa}) \qquad \forall \bdd{\eta} \in \mathbbb{Q}_{\Gamma},
	\end{align*}
	and the triple $(\tilde{w}_X, \bdd{\gamma}_X, \bdd{\zeta}_X + \grad \tilde{z}_X)$ is the unique solution to \cref{eq:h2-mixed-problem-fem}.
\end{theorem}

Incidentally, \cref{eq:tildezx-def,eq:stokes-system,eq:h1-postprocess} is precisely the method we derived in \cite{AinCP23KirchI} under the assumption that $\Omega$ is simply-connected and $F_1 \equiv 0$. \Cref{thm:two-methods-equivalent} shows that the method still produces the $\mathbb{W}_{\Gamma}$-conforming finite element approximation \cref{eq:h2-bilinear-fem} under the more general domain assumptions in \cref{sec:gen-problem-setting} and for nonzero $F_1 \in \mathbbb{G}_{\Gamma}^*$.



\end{document}